\documentclass[11pt]{amsart}
\usepackage[a4paper,margin=1in]{geometry}
\usepackage[T1]{fontenc}
\usepackage{lmodern}
\usepackage{microtype}
\usepackage{amsmath,amssymb,amsthm,mathtools}
\usepackage{mathrsfs}
\usepackage{enumitem}
\usepackage{aliascnt}
\usepackage{booktabs}
\usepackage{xcolor}
\usepackage[colorlinks=true,linkcolor=blue,citecolor=blue,urlcolor=blue]{hyperref}
\usepackage[nameinlink,capitalise,noabbrev]{cleveref}

\allowdisplaybreaks

\numberwithin{equation}{section}
\newtheorem{theorem}{Theorem}[section]

\newaliascnt{proposition}{theorem}
\newtheorem{proposition}[proposition]{Proposition}
\aliascntresetthe{proposition}

\newaliascnt{lemma}{theorem}
\newtheorem{lemma}[lemma]{Lemma}
\aliascntresetthe{lemma}

\newaliascnt{corollary}{theorem}
\newtheorem{corollary}[corollary]{Corollary}
\aliascntresetthe{corollary}

\newaliascnt{definition}{theorem}
\newtheorem{definition}[definition]{Definition}
\aliascntresetthe{definition}

\newaliascnt{remark}{theorem}
\newtheorem{remark}[remark]{Remark}
\aliascntresetthe{remark}
\hypersetup{
  pdftitle={On the Spatially Homogeneous Boltzmann Equation with Mass Exchange},
  pdfauthor={Siwei Luo and Jian-Guo Liu},
  pdfsubject={Spatially homogeneous Boltzmann equation with continuous mass exchange},
  pdfkeywords={Boltzmann equation with mass exchange, spatially homogeneous, hard potentials, Cauchy problem, uniform integrability}
}

\crefname{theorem}{Theorem}{Theorems}
\Crefname{theorem}{Theorem}{Theorems}

\crefname{proposition}{Proposition}{Propositions}
\Crefname{proposition}{Proposition}{Propositions}

\crefname{lemma}{Lemma}{Lemmas}
\Crefname{lemma}{Lemma}{Lemmas}

\crefname{corollary}{Corollary}{Corollaries}
\Crefname{corollary}{Corollary}{Corollaries}

\crefname{definition}{Definition}{Definitions}
\Crefname{definition}{Definition}{Definitions}

\crefname{remark}{Remark}{Remarks}
\Crefname{remark}{Remark}{Remarks}

\newcommand{\R}{\mathbb R}
\newcommand{\Sph}{\mathbb S^{d-1}}
\newcommand{\X}{\mathcal X}
\newcommand{\one}{\mathbf 1}
\newcommand{\norm}[1]{\left\|#1\right\|}
\newcommand{\abs}[1]{\left|#1\right|}
\newcommand{\pair}[2]{\left\langle#1,#2\right\rangle}

\title{On the Spatially Homogeneous Boltzmann Equation with Mass Exchange}
\author{Siwei Luo}
\address{University of Science and Technology of China,
Hefei 230026, China}
\email{luosw@mail.ustc.edu.cn}
\author{Jian-Guo Liu}
\address{Departments of Mathematics and Physics, Duke University,
Durham, North Carolina 27708, USA}
\email{jian-guo.liu@duke.edu}
\date{July 2026}



\begin{document}
\bibliographystyle{amsalpha}

\begin{abstract}
    We study the spatially homogeneous Boltzmann equation with continuous mass exchange on $X=(0,\infty)_m\times\mathbb R^d_v$, with a Grad cut-off hard-potential collision kernel. For bounded continuous symmetric mass-exchange rates, every nonnegative initial datum with finite number, mass, and kinetic energy admits a global nonnegative $L^1$-integral weak solution in $W^{1,\infty}(0,\infty;L^1(X))$ with number and mass conserved, kinetic energy dissipated. If
    \[
    \int_X (m|v|^2)^{1+\delta} f_0(x)dx<\infty,
    \]
    for some $\delta>0$, this higher-energy moment propagates on every finite time interval and kinetic energy is conserved through the constructed solution. Moreover, every energy-dissipating solution propagates any such moment. Under the additional $1+\gamma$ moment assumption, the solution is unique among all energy-dissipating $L^1$-integral weak solutions with the same initial datum. We also establish a local theory for a linearly growing mass-exchange rate. With
    \[
    H_p(f)=\int_X (1+m+m|v|^2)^pfdx,
    \]
    every datum with $H_p(f_0)<\infty$, where $p\ge1+\gamma$ admits a conservative local $H_p$-solution. Moreover, an $H_p$-solution continues across every finite time $T$ for which $H_{1+\gamma}(f)\in L^1(0,T).$ This proof requires no detailed-balance or relative-entropy structures. It is based on a new two-stage bootstrap method. The first stage rules out mass concentration at $m=0$, while the second stage combines this control with collision geometry to establish uniform integrability. These estimates provide the compactness needed for the global solution and for the identification of the nonlinear collision form.
\end{abstract}\maketitle\tableofcontents

\section{Introduction}
Kinetic equations of Boltzmann type provide a mesoscopic description of large systems of interacting particles through binary collisions. Its mathematical theory combines the geometry of binary collisions with compactness, stability, and moment estimates for a nonlinear integral operator. Global weak and renormalized solution theories, as well as the spatially homogeneous theory for cut-off hard potentials, have been developed extensively; see, among many references, \cite{CercignaniIllnerPulvirenti1994,DiPernaLions1989,Villani2002,MischlerWennberg1999}.

In many applications, particles carry an internal degree of freedom in addition to velocity, and collisions may exchange or redistribute that internal quantity. For polyatomic gases, this additional variable typically represents internal energy, and the corresponding collision rules redistribute translational and internal energy while preserving the total energy of each colliding pair \cite{BourgatEtAl1994,GambaPavicColic2023}. Size-structured kinetic descriptions also appear in coagulation--fragmentation theory, where particles merge or split and the number of particles generally changes \cite{Aldous1999,BanasiakLambLaurencot2019}. At the kinetic level, continuous-mass models in which particles also
carry momentum or velocity have been studied for coalescence and coagulation--fragmentation \cite{EscobedoLaurencotMischler2004,Broizat2010}.
They are of type $2\to1$ or $1\to2$, so particle number is generally not conserved. Binary collisions with mass exchange occupy a distinct position between these settings. Each particle carries both a velocity and a physical mass and a collision redistributes the total mass between the two outgoing particles while retaining a $2\to 2$ event structure. Thus particle number remains invariant, while the mass variable enters directly into momentum, kinetic energy, and the post-collisional velocity geometry.

The Boltzmann equation with mass exchange (BME) was introduced by Degond and Liu in a discrete-mass setting \cite{DegondLiu2025}. In that model, particle masses are integer multiples of an elementary mass, and each binary collision redistributes the combined mass of the incoming pair. Degond and Liu derived the associated conservation laws, an H-theorem, equilibrium distributions, and formal macroscopic and relaxation descriptions. The continuous-mass equation is naturally suggested by refinement of the elementary mass scale. It also defines a kinetic model in its own right, for which a direct Cauchy theory is needed.

Related continuous exchange-driven models have recently been studied
in \cite{BarikDaCostaPintoSasportes2025,LamSchlichting2026}.
These models describe binary redistribution of a continuous mass
variable, but do not contain the velocity variable or the
momentum--energy collision geometry of the BME operator. Binary
exchange rules also arise in kinetic models of wealth and gambling
\cite{BassettiToscani2010}, again with a different state space and
interaction mechanism. The present work concerns specifically the
Cauchy theory for the continuous mass--velocity BME collision operator.

More precisely, we study the spatially homogeneous equation on
$X=(0,\infty)_m\times\mathbb R^d_v$, with $x=(m,v)$ and $dx= dm
dv$.
The velocity collision kernel has the Grad cut-off hard-potential form
\[
 B(E,\xi)=E^\gamma b(\xi),\qquad
 0<\gamma<1,\qquad b\in L^1(\Sph).
\]
Here $E$ denotes the relative kinetic energy of the incoming pair and
$\xi$ is the angular variable.  We consider two regimes for the
nonnegative, continuous, symmetric mass-exchange
rate.  In the first, $a$ is bounded; in the second,
\[
 a(m,m_1,\alpha)\leq A_a(1+m+m_1).
\]
The bounded regime admits a global theory based only on the physical
moments, while the linearly growing regime leads to a local theory
controlled by a higher weighted moment and a critical-moment
continuation criterion.

 For a nonnegative distribution $f$, define
 \[M_0(f)=\int_X f dx,\qquad M_1(f)=\int_X mf dx,\qquad M_2(f)=\int_X m|v|^2fdx.\]
 Our first main result states that every initial datum $f_0\ge 0$ satisfying $M_0(f_0)+M_1(f_0)+M_2(f_0)<\infty$ generates a global nonnegative $L^1$-integral weak solution. The collision operator $\mathbf Q(f,f)$ is realized as an $L^1(X)$-valued map, and the solution satisfies the equation as a global Bochner integral identity. Moreover,
 \[f\in W^{1,\infty}(0,\infty;L^1(X))\cap L^\infty(0,\infty;L^1(X;(1+m+m|v|^2)dx).\]
 Particle number and total mass are conserved, while the kinetic-energy inequality
 \[M_2(f(t))\leq M_2(f_0),\qquad t\geq0,\]
 holds. If
 \[\int_X (m|v|^2)^{1+\delta}f_0(x)dx<\infty\]
 for some $\delta>0$, this higher-energy moment propagates on every compact time interval and kinetic energy is conserved. Under the additional $1+\gamma$ moment assumption, the solution is unique among all global energy-dissipating $L^1$-integral weak solutions with the same initial datum.

 Our second main result concerns the linearly growing regime.  Define
 \[
  W(m,v)=1+m+m|v|^2,
  \qquad H_q(g)=\int_XW(x)^qg(x) dx.
 \]
 If $p\geq1+\gamma$ and $H_p(f_0)<\infty$, then there exists a
 nonnegative conservative local $L^1$-integral weak solution with
 locally bounded $H_p$-moment.  Writing
 $\vartheta=\gamma/(p-1)$, the truncated solutions satisfy
 \[
  \frac{d}{dt}H_p(f_n(t))
  \leq C_{p,\gamma}A_a\|b\|_{L^1}
  H_p(f_n(t))H_{1+\gamma}(f_n(t)),
  \qquad
  H_{1+\gamma}(f_n(t))
  \leq H_1(f_0)^{1-\vartheta}H_p(f_n(t))^\vartheta.
 \]
 Bihari's inequality therefore gives an explicit positive lower bound
 for the lifespan.  Every conservative
 $H_p$-solution satisfies
 \[
  H_p(f(t))\leq H_p(f(s))
  \exp\left(C_{p,\gamma}A_a\|b\|_{L^1}
  \int_s^tH_{1+\gamma}(f(\tau)) d\tau\right).
 \]
 It follows that a solution extends across every finite endpoint at
 which $H_{1+\gamma}(f)$ is time-integrable.

 The principal issue in constructing such solutions is compactness. In the classical Boltzmann theory, entropy estimates provide an important mechanism for obtaining uniform integrability and passing to limits in nonlinear collision operators \cite{DiPernaLions1989,Villani2002}. For the BME entropy structure currently available, the entropy is naturally measured relative to a mass-dependent reference weight compatible with the collision transformation \cite{DegondLiu2025}. Within this relative-entropy framework, compatibility of the reference weight with mass exchange leads to detailed-balance-type relations in the mass variable. Employing this mechanism in an existence proof would therefore add structural assumptions beyond the boundedness, continuity, and natural symmetries of the collision kernels. Instead, we develop an entropy-free compactness argument based directly on the geometry of mass-exchange collisions.

 Two coupled degeneracies make this problem substantially different from the usual spatially homogeneous Boltzmann equation. First, a bound on total mass does not prevent concentration of the particle-number density near $m=0$. Second, kinetic energy controls velocity only away from the zero-mass boundary. Particles may satisfy $m\downarrow 0$, $|v|\uparrow\infty$ but $m|v|^2=O(1).$ so a uniform bound on $M_2$ does not by itself provide a uniform velocity-tail estimate. Concentration at small mass can consequently produce escape toward arbitrarily large velocities without violating the kinetic-energy bound.

 The first ingredient of the proof is a new small-mass bootstrap. Let $f_N$ be the solutions of a family of bounded-kernel approximations and set
 \[F_N(r,t):= \int_0^r\int_{\mathbb R^d} f_N(t,m,v) dv dm.\]
 By separating collisions according to whether the total incoming mass is smaller or larger than an intermediate scale $\rho$, we derive
 \[\partial_t F_N(r,t)\leq CF_N(\rho,t)^{2-\gamma}+C\frac{r}{\rho},\qquad 0<r<\rho\ll 1.\]
 Choosing $\rho=\sqrt{r}$ and working on sufficiently short time intervals yields a bootstrap inequality of the form $L_I\leq C|I|L_I^{2-\gamma}$, where $L_I=\limsup_{r\downarrow 0} \sup_N \sup_{t\in I} F_N(r,t).$ Since $2-\gamma>1$, this enforces $L_I=0$. Iteration over consecutive time intervals rules out concentration at $m=0$ on every finite time interval. The same estimate resolves the degeneracy of the velocity tail.

 Tightness alone does not exclude concentration on sets of small Lebesgue measure. We therefore introduce a second bootstrap for the uniform-integrability modulus
 \[U_N(q,t)=\sup_{\substack{A\subset X\text{ measurable}\\|A|\leq q}}
 \int_Af_N(t,x)dx.\]
The collision space is decomposed into good and bad regions. The bad region contains the endpoints of the mass-exchange parameter, degenerate incoming mass ratios, and neighborhoods of the two resonance sets at which one of the one-particle collision Jacobians may vanish. Its contribution is controlled using the small-mass estimate, the total-mass moment, and the absolute continuity of the angular integral. On the good region, both one-particle output maps have Jacobians bounded uniformly away from zero. The area formula then bounds the gain into a small set $A$ in terms of $U_N(C|A|,t).$ A short-time absorption argument, followed by iteration, gives $\lim_{q\downarrow 0}\sup_{N} \sup_{t\in [0,T]} U_N(q,t)=0$ for every $T>0$.

Uniform integrability, tightness in mass and velocity, and time equicontinuity yield weak compactness through the Dunford--Pettis theorem \cite{DiestelUhl1977}. The remaining difficulty is the identification of the quadratic collision form $\mathcal Q_N(f_N,f_N)[\psi]\to \mathcal Q(f,f)[\psi]$. This is achieved by successively removing the kernel truncation and the high-relative-energy region, localizing the particle variables, excluding the endpoints of the mass-exchange parameter, approximating the angular kernel by continuous functions, and reducing the resulting two-particle integrands to finite sums of tensor products. The limiting collision form is subsequently represented by an $L^1(X)$ density, which gives the global integral formulation and the stated time regularity.
 
The uniqueness theory requires an additional particle-energy moment. We first prove that every energy-dissipating solution propagates any initially finite moment of $(m|v|^2)^p$, $p>1$, on compact time intervals.  We then prove a weighted Kato estimate in the space associated with $\varpi(m,v)=1+m|v|^2$. If $f,g$ are two energy-dissipating solutions, the difference $h=f-g$ satisfies an estimate of the form
\[\frac{d}{dt}\|h(t)\|_{L^1_\varpi}\leq C\mathcal A_{f+g}(t)\|h(t)\|_{L^1_\varpi},
 \]
where \[\mathcal A_{f+g}(t)=\int_X\varpi(x)[1+(m|v|^2)^\gamma] |f(t,x)+g(t,x)|dx\] is integrable because the preceding propagation result applies with $p=1+\gamma$. This closes the Gronwall argument and gives uniqueness in the energy-dissipating class.

For classical hard-potential Boltzmann equations, Povzner-type inequalities provide instantaneous production of higher velocity moments \cite{Povzner1962,Bobylev1997,Desvillettes1993,
MischlerWennberg1999,Wennberg1997,LuMouhot2012}. The mass-exchange geometry does not possess an analogous production mechanism for the individual particle-energy moments. We exhibit initial data with finite $M_0,M_1,M_2$, but infinite $1+\gamma$ energy moment, and construct a solution where this moment remains infinite
at every finite positive time. Consequently, the moment condition entering the uniqueness theorem cannot in general be recovered dynamically from the basic existence assumptions.

The remainder of the paper is organized as follows. \cref{sec:model} introduces the continuous-mass collision law, the assumptions, and the main theorem for the bounded mass-exchange kernel. \cref{sec:bounded} constructs the truncated bounded-collision-kernel solutions. \cref{sec:compact-est} establishes the small-mass and uniform-integrability bootstraps, together with the mass and velocity tightness estimates, which is the most important part of the method. \cref{sec:compact} proves weak compactness and identifies the nonlinear collision operator. \cref{sec:strong-global} derives the $L^1$-valued formulation, the conservation laws, and higher-energy moment propagation for the constructed solution. \cref{sec:uniqueness} proves higher-energy moment propagation for arbitrary energy-dissipating solutions, establishes uniqueness in that class, and presents the obstruction to instantaneous higher-moment generation. Finally, \cref{sec:linear-growth} develops the local $H_p$-theory for linearly growing mass-exchange kernels and proves the critical-moment continuation and blow-up criteria.
\section{Model, assumptions, and main result}
\label{sec:model}

Fix an integer \(d\geq1\).  Put
\[
 X=(0,\infty)_m\times\R^d_v,
 \qquad x=(m,v),\qquad x_1=(m_1,v_1),
 \qquad dx=dm d v.
\]
For a nonnegative density \(g\) on \(X\), the three moments used throughout
the paper are
\begin{align}
 M_0(g)&=\int_X g(x)\, d x,\label{eq:M0}\\
 M_1(g)&=\int_X m g(x)\, d x,\label{eq:M1}\\
 M_2(g)&=\int_X m|v|^2 g(x)\, d x.\label{eq:M2}
\end{align}

\subsection{Collision geometry}
We use the continuous-mass analogue of the collision geometry introduced by Degond and Liu in the discrete-mass setting \cite{DegondLiu2025}.
Let
\[
 S=m+m_1,\qquad \theta=\frac mS,\qquad
 V=\frac{mv+m_1v_1}{S},\qquad u=v-v_1.
\]
For \(\alpha\in(0,1)\) and \(\omega\in\Sph\), define
\[
 R_\omega u=u-2\pair{u}{\omega}\omega,
 \qquad
 s=\left(\frac{\theta(1-\theta)}{\alpha(1-\alpha)}\right)^{1/2}.
\]
The post-collisional variables are
\begin{align}
 m'&=\alpha S,&m_1'&=(1-\alpha)S,\label{eq:postmass}\\
 v'&=V+(1-\alpha)sR_\omega u,&
 v_1'&=V-\alpha sR_\omega u.\label{eq:postvel}
\end{align}
Direct substitution gives
\begin{align}
 m'+m_1'&=S=m+m_1,\label{eq:masscons}\\
 m'v'+m_1'v_1'
 &=\alpha S\bigl[V+(1-\alpha)sR_\omega u\bigr]
 +(1-\alpha)S\bigl[V-\alpha sR_\omega u\bigr]
 =SV=mv+m_1v_1,\label{eq:momcons}\\
 m'|v'|^2+m_1'|v_1'|^2
 &=S|V|^2+S\alpha(1-\alpha)s^2|u|^2\notag\\
 &=S|V|^2+S\theta(1-\theta)|u|^2
 =m|v|^2+m_1|v_1|^2.\label{eq:energycons}
\end{align}
The reduced kinetic energy is
\begin{equation}\label{eq:Edef}
 E=\frac{mm_1}{m+m_1}|v-v_1|^2
   =S\theta(1-\theta)|u|^2,
\end{equation}
and the last line above yields
\begin{equation}\label{eq:E-energy}
 0\leq E\leq S|V|^2+S\theta(1-\theta)|u|^2
 =m|v|^2+m_1|v_1|^2.
\end{equation}

\subsection{Kernel assumptions}

We impose the following assumptions.
\begin{enumerate}[label={},leftmargin=3.2em]
\item[\textnormal{(K1)}]
The mass-exchange kernel
\(a:(0,\infty)^2\times(0,1)\to[0,\infty)\) is measurable, bounded, and continuous.
\item[\textnormal{(K2)}]
The collision kernel has the Grad cutoff hard-potential form \cite{CercignaniIllnerPulvirenti1994,Villani2002}
\begin{equation}\label{eq:hardkernel}
B(E,\xi)=E^\gamma b(\xi),\qquad 0<\gamma<1,
\end{equation}
where \(b:[-1,1]\to[0,\infty)\) is measurable and, for one (and hence
every) \(e\in\Sph\),
\begin{equation}\label{eq:gradcutoff}
 \int_{\Sph}b(e\cdot\omega)\, d\omega
 =\norm b_{L^1(\Sph)}<\infty.
\end{equation}
At \(u=0\), the value assigned to \(u/|u|\) is immaterial because
\(E^\gamma=0\).
\item[\textnormal{(K3)}]
The mass-exchange kernel has the symmetries
\begin{equation}\label{eq:a-sym}
 a(m,m_1,\alpha)=a(m_1,m,\alpha)=a(m,m_1,1-\alpha).
\end{equation}
\end{enumerate}

The initial datum satisfies
\begin{equation}\label{eq:IC1}
 \textnormal{(IC1)}\qquad f_0\geq0,\qquad
 M_0(f_0)+M_1(f_0)+M_2(f_0)<\infty.
\end{equation}
For the energy-conservation part only, we use the additional assumption
\begin{equation}\label{eq:IC2}
 \textnormal{(IC2)}\qquad\text{there exists }\delta>0\text{ such that }
 \int_X\bigl(m|v|^2\bigr)^{1+\delta}f_0(m,v) dx<\infty.
\end{equation}
For uniqueness in the energy-dissipating class, we impose the more
specific higher-energy condition
\begin{equation}\label{eq:IC3}
 \textnormal{(IC3)}\qquad
 \int_X\bigl(m|v|^2\bigr)^{1+\gamma}
 f_0(m,v)\, d m\, d v<\infty.
\end{equation}
Since \(0<\gamma<1\), assumption \eqref{eq:IC3} implies
\eqref{eq:IC2} with \(\delta=\gamma\).

\subsection{Weak solutions}

For a bounded Borel measurable test function
\(\psi:X\to\mathbb R\), write
\[
 \Delta\psi
 =\psi(m',v')+\psi(m_1',v_1')-\psi(m,v)-\psi(m_1,v_1).
\]
For a nonnegative measurable density \(g\) on \(X\), whenever the
integral is absolutely convergent, define the static collision form
\begin{align}
 \mathcal Q(g,g)[\psi]
 :=\frac12\int_{X^2}\int_0^1\int_{\Sph}
 &a(m,m_1,\alpha)E^\gamma
 b\left(\frac u{|u|}\cdot\omega\right)
 \Delta\psi \notag\\
 &\hspace{35mm}\times g(x)g(x_1)
 \, d\omega\, d\alpha\, d x_1\, d x.
 \label{eq:weakQ}
\end{align}

\begin{definition}\label{def:weak}\label{def:L1-integral}
Let \(T>0\).  A nonnegative function
\(
 f\in L^\infty(0,T;L^1(X;(1+m+m|v|^2) d x))
\)
is a weak solution on \([0,T]\) with initial datum \(f_0\) if, for every
bounded Borel measurable \(\psi:X\to\mathbb R\) and every
\(\eta\in C_c^1([0,T))\),
\begin{align}
 -\int_0^T\eta'(t)\int_X f(t,x)\psi(x)\, d x\, d t
 &=\int_0^T\eta(t)\mathcal Q(f(t),f(t))[\psi]\, d t
 +\eta(0)\int_Xf_0(x)\psi(x)\, d x.
 \label{eq:weaksol}
\end{align}
It is a global weak solution if this identity holds on every finite time
interval.

\medskip
A global weak solution \(f\) is called a global
\(L^1\)-integral weak solution if it has a representative in
\(C([0,\infty);L^1(X))\), the collision form is represented for almost
every \(t>0\) by a density
\(\mathbf Q(f(t),f(t))\in L^1(X)\),
\[
 \int_X\psi(x)\mathbf Q(f(t),f(t))(x)\, d x
 =\mathcal Q(f(t),f(t))[\psi],
 \qquad \psi\in L^\infty(X),
\]
the map \(t\mapsto\mathbf Q(f(t),f(t))\) belongs to
\(L^1_{\mathrm{loc}}(0,\infty;L^1(X))\), and
\[
 f(t)=f_0+\int_0^t\mathbf Q(f(s),f(s))\, d s
 \quad\text{in }L^1(X),\qquad t\geq0,
\]
where the integral is a Bochner integral.

\medskip
Using its \(C([0,\infty);L^1(X))\) representative, such a solution is
called energy-dissipating if \(M_2(f(t))<\infty\) for every
\(t\geq0\) and
\begin{equation}\label{eq:energy-dissipating-class}
 M_2(f(t))\leq M_2(f_0),\qquad t\geq0.
\end{equation}
\end{definition}

The collision form in \eqref{eq:weakQ} is absolutely convergent for the
tests in \cref{def:weak}.  Indeed,
\(\abs{\Delta\psi}\leq4\norm\psi_\infty\), and the calculation in
\eqref{eq:hard-total-rate-bm} bounds the absolute-value version of the
right-hand side of \eqref{eq:weakQ} by
\(
4\norm\psi_\infty\norm a_\infty\norm b_{L^1}
M_0(f(t))^{2-\gamma}M_2(f(t))^\gamma
\)
for almost every \(t\in(0,T)\).  This bound is uniform in \(t\) by
the weighted \(L^\infty_tL^1_x\) membership in \cref{def:weak}.

\begin{theorem}[Global \(L^1\)-integral solution, moment laws, and uniqueness in the energy-dissipating class]
\label{thm:main}
Assume \textnormal{(K1)}, \textnormal{(K2)}, \textnormal{(K3)}, and
\eqref{eq:IC1}.  Then there exists a global
\(L^1\)-integral weak solution \(f\geq0\) in the sense of
\cref{def:L1-integral}.  Put
\begin{equation}\label{eq:main-Cstar}
 C_*=4\norm a_\infty\norm b_{L^1}
 M_0(f_0)^{2-\gamma}M_2(f_0)^\gamma.
\end{equation}
The solution satisfies
\begin{align}
 f\in W^{1,\infty}(0,\infty;L^1(X))\cap L^\infty(0,\infty;
 L^1(X;(1+m+m|v|^2)\, d m\, d v)).
 \label{eq:main-global-weighted-space}
\end{align}
The equation holds as 
\begin{align}
 \mathbf Q(f,f)\in L^\infty(0,\infty;L^1(X)),
 \qquad \partial_t f=\mathbf Q(f,f)
 \quad\text{in }L^1(X)\text{ for a.e. }t>0,
 \label{eq:main-global-Q}
\end{align}
Quantitatively,
\begin{align}
 \operatorname*{ess\,sup}_{t>0}
 \int_X(1+m+m|v|^2)f(t,m,v)\, d m\, d v
 &\leq M_0(f_0)+M_1(f_0)+M_2(f_0),
 \label{eq:main-global-weighted-bound}\\
 \operatorname*{ess\,sup}_{t>0}
 \norm{\partial_t f(t)}_{L^1(X)}
 =\operatorname*{ess\,sup}_{t>0}
 \norm{\mathbf Q(f(t),f(t))}_{L^1(X)}
 &\leq C_*.
 \label{eq:main-time-derivative-bound}
\end{align}

For every bounded Borel measurable \(\psi:X\to\mathbb R\) and every
\(t\geq0\),
\begin{equation}\label{eq:main-bounded-measurable}
 \int_Xf(t,x)\psi(x)\, d x
 =\int_Xf_0(x)\psi(x)\, d x
  +\int_0^t\mathcal Q(f(s),f(s))[\psi]\, d s,
\end{equation}
and the collision integral is absolutely convergent.  
The solution satisfies, for every \(t\geq0\),
\begin{align}
 M_0(f(t))&=M_0(f_0),\label{eq:mainM0}\\
 M_1(f(t))&=M_1(f_0),\label{eq:mainM1}\\
 M_2(f(t))&\leq M_2(f_0).\label{eq:mainM2ineq}
\end{align}
If \eqref{eq:IC2} also holds, set \(p=1+\delta\) and
\begin{equation}\label{eq:main-Kp}
 K_p=p2^{p-1+\gamma}\norm a_\infty\norm b_{L^1}M_2(f_0).
\end{equation}
Then, for every \(T>0\),
\begin{align}
 \sup_{0\leq t\leq T}\int_X(m|v|^2)^pf(t,m,v)\, d m\, d v
 &\leq
 \left[M_0(f_0)+\int_X(m|v|^2)^pf_0(m,v)\, d m\, d v\right]
 e^{K_pT}-M_0(f_0),
 \label{eq:main-higher-energy}\\
 M_2(f(t))&=M_2(f_0)\qquad\text{for every }t\geq0.
 \label{eq:mainM2}
\end{align}
If \eqref{eq:IC3} also holds, then \(f\) is the unique
energy-dissipating global \(L^1\)-integral weak solution with initial
datum \(f_0\).  In other words, if \(\widetilde f\) is any other global
\(L^1\)-integral weak solution with the same initial datum and
\begin{equation}\label{eq:main-uniqueness-class}
 M_2(\widetilde f(t))\leq M_2(f_0),\qquad t\geq0,
\end{equation}
then \(\widetilde f(t)=f(t)\) in \(L^1(X)\) for every \(t\geq0\).
\end{theorem}

\begin{remark}[Role of the higher moment and entropy-free mechanism]
Assumption \eqref{eq:IC2} is used only to obtain the higher-energy
estimate \eqref{eq:main-higher-energy}, which upgrades
\(M_2(f(t))\leq M_2(f_0)\) to equality through
\cref{prop:higherenergy,prop:M2}.  The global \(L^1\)-integral
formulation, the two global-in-time function-space bounds, conservation
of \(M_0,M_1\), and the energy inequality use \eqref{eq:IC1} alone.
Assumption \eqref{eq:IC3} is used in the uniqueness argument as
follows.  The propagation result in \cref{sec:uniqueness} shows that
every energy-dissipating competitor with initial datum \(f_0\) has the
locally uniform \(1+\gamma\) higher-energy moment needed in the
weighted stability estimate.  Hence the uniqueness class is determined
by the physically natural energy inequality
\eqref{eq:main-uniqueness-class}, rather than by assuming a propagated
higher moment as part of the definition of the class.
No detailed-balance weight, relative entropy, or entropy production is
used; uniform integrability follows from the two bootstraps in
\cref{sec:compact-est}.
\end{remark}

\section{The bounded-kernel equation}\label{sec:bounded}

For \(N\geq1\), set
\begin{equation}\label{eq:BN}
 B_N(E,\xi)=\min\{E^\gamma b(\xi),N\}.
\end{equation}
Thus
\begin{equation}\label{eq:BNbound}
0\leq B_N(E,\xi)\leq N
 \qquad E\geq0,\ \text{a.e. }\xi\in[-1,1].
\end{equation}

\subsection{The collision change of variables}

\begin{lemma}[Augmented collision diffeomorphism]\label{lem:augmented}
For every fixed \(\omega\in\Sph\), the map
\[
 \Phi_\omega:(m,m_1,\alpha,v,v_1)
 \longmapsto(m',m_1',\theta,v',v_1')
\]
is a \(C^1\)-diffeomorphism of
\((0,\infty)^2\times(0,1)\times\R^{2d}\) onto itself, and
\begin{equation}\label{eq:augjac}
 \left|\det D\Phi_\omega\right|=s^d,
 \qquad
  d m\, d m_1\, d\alpha\, d v\, d v_1
 =s^{-d} d m'\, d m_1'\, d\theta\, d v'\, d v_1'.
\end{equation}
\end{lemma}

\begin{proof} 
Step 1. Use the incoming coordinates
\[
 S=m+m_1,\qquad \theta=\frac mS,\qquad
 V=\theta v+(1-\theta)v_1,\qquad u=v-v_1.
\]
Their two nontrivial Jacobian blocks are
\[
 \frac{\partial(S,\theta)}{\partial(m,m_1)}
 =\begin{pmatrix}1&1\\ m_1/S^2&-m/S^2\end{pmatrix},
 \qquad
 \frac{\partial(V,u)}{\partial(v,v_{1})}
 =\begin{pmatrix}\theta I_d&(1-\theta)I_d\\I_d&-I_d\end{pmatrix},
\]
whose determinants are \(-S^{-1}\) and \((-1)^d\), respectively.  Since
the full derivative is block lower triangular,
\begin{equation}\label{eq:incoming-coordinate-jacobian}
 \left|\det
 \frac{\partial(S,\theta,\alpha,V,u)}
 {\partial(m,m_1,\alpha,v,v_1)}\right|=\frac1S.
\end{equation}

Step 2. For the outgoing variables use
\[
 S=m'+m_1',\qquad \alpha=\frac{m'}S,\qquad
 V=\alpha v'+(1-\alpha)v_1',\qquad u'=v'-v_1'.
\]
The same calculation gives
\begin{equation}\label{eq:outgoing-coordinate-jacobian}
 \left|\det
 \frac{\partial(S,\alpha,\theta,V,u')}
 {\partial(m',m_1',\theta,v',v_1')}\right|=\frac1S.
\end{equation}

Step 3. The collision formulas yield
\[
 u'=sR_\omega u,\qquad
 R_\omega=I-2\omega\otimes\omega,\qquad
 R_\omega^{\mathsf T}R_\omega=R_\omega^2=I.
\]
Consequently \(|\det(sR_\omega)|=s^d\).  In the coordinates above, the
intermediate derivative has the block form
\[
 \frac{\partial(S,\alpha,\theta,V,u')}
 {\partial(S,\theta,\alpha,V,u)}
 =\begin{pmatrix}P&0\\ *&sR_\omega\end{pmatrix},
 \qquad |\det P|=1,
\]
where the derivatives of \(s=s(\theta,\alpha)\) occur only in the
lower-left block.  

Step 4. Combining this determinant with
\eqref{eq:incoming-coordinate-jacobian} and
\eqref{eq:outgoing-coordinate-jacobian} gives
\[
|\det D\Phi_\omega|
=\left|\det
 \frac
{\partial(m',m_1',\theta,v',v_1')}
 {\partial(m,m_1,\alpha,v,v_1)}
\right|=s^d\,
\]
and the measure identity in
\eqref{eq:augjac} follows.

Step 5. It remains to verify global invertibility.  Given
\((m',m_1',\theta,v',v_1')\), define
\begin{align*}
 S&=m'+m_1',& \alpha&=m'/S,&
 m&=\theta S,&m_1&=(1-\theta)S,\\
 V&=\alpha v'+(1-\alpha)v_1',&u'&=v'-v_1',&
 s&=\left(\frac{\theta(1-\theta)}{\alpha(1-\alpha)}\right)^{1/2},\\
 u&=s^{-1}R_\omega u',&
 v&=V+(1-\theta)u,&v_1&=V-\theta u.
\end{align*}
Then
\[
 \theta v+(1-\theta)v_1=V,\qquad v-v_1=u,\qquad
 sR_\omega u=u',
\]
and hence
\[
 V+(1-\alpha)sR_\omega u=v',\qquad
 V-\alpha sR_\omega u=v_1'.
\]
These formulas give a two-sided \(C^1\) inverse on the stated open set.
\end{proof}

\subsection{The bounded collision operator}

Set
\begin{equation}\label{eq:weighted-space}
 W(m,v)=1+m+m|v|^2,\qquad
 \X=L^1(X;W(x)\, d x),\qquad
 \norm g_{\X}=\int_XW(x)|g(x)|\, d x.
\end{equation}
For \(g,h\in\X\), we will construct the strong collision operator $\mathbf Q_N(g,h)$  satisfying
\begin{align}
 \int_X\mathbf Q_N(g,h)\psi\, d x
 =\frac14\int_{X^2}\int_0^1\int_{\Sph}
 &aB_N\Delta\psi
 [g(x)h(x_1)+h(x)g(x_1)]
 \, d\omega\, d\alpha\, d x_1\, d x,
 \label{eq:polarizedQN}
\end{align}
for every measurable $|\psi|\lesssim W$ whenever the right-hand side is absolutely convergent.

In particular,
\begin{equation}\label{eq:strong-weakQN}
    \mathcal Q_N(g,g)[\psi]=\int_{X}\mathbf Q_N(g,g)(x) \psi(x) dx
\end{equation}
for every measurable $|\psi|\lesssim W$.

\begin{lemma}[Bounded bilinear collision operator]\label{lem:boundedQ}
There exists a bilinear map
\(\mathbf Q_N:\X\times\X\to\X\) satisfying \eqref{eq:polarizedQN} and \eqref{eq:strong-weakQN}.  With
\(
 C_N=\norm a_\infty N|\Sph|,
\)
one has
\begin{align}
 \norm{\mathbf Q_N(g,h)}_{\X}
 &\leq C_N\bigl(
 \norm g_{\X}\norm h_{L^1(X)}
 +\norm g_{L^1(X)}\norm h_{\X}\bigr),
 \label{eq:boundedQestimate}\\
 \norm{\mathbf Q_N(g,g)}_{\X}
 &\leq2C_N\norm g_{\X}^2,
 \label{eq:quadratic-growth}\\
 \norm{\mathbf Q_N(g,g)-\mathbf Q_N(h,h)}_{\X}
 &\leq2C_N(\norm g_{\X}+\norm h_{\X})\norm{g-h}_{\X}.
 \label{eq:boundedQ-Lipschitz}
\end{align}
\end{lemma}

\begin{proof}
On the incoming collision space put
\[
 \rho=\frac14a(m,m_1,\alpha)
 B_N\left(E,\frac{u}{|u|}\cdot\omega\right)
 [g(x)h(x_1)+h(x)g(x_1)].
\]
Then
\begin{equation}\label{eq:rho-L1}
 \int_{X^2}\int_0^1\int_{\Sph}|\rho|
 \, d\omega\, d\alpha\, d x_1\, d x
 \leq\frac12C_N\norm g_{L^1(X)}\norm h_{L^1(X)}<\infty.
\end{equation}
For fixed \(\omega\), let
\[
 \widetilde\rho(x',x_1',\theta,\omega)
 =\rho\bigl(\Phi_\omega^{-1}(x',x_1',\theta),\omega\bigr)s^{-d}.
\]
By \eqref{eq:augjac}, for every nonnegative measurable \(F\),
\begin{align}
 &\int_{X^2}\int_0^1\int_{\Sph}
 F(x',x_1',\theta,\omega)|\widetilde\rho|
 \, d\omega\, d\theta\, d x_1'\, d x'\notag\\
 &\qquad=
 \int_{X^2}\int_0^1\int_{\Sph}
 F(x',x_1',\theta,\omega)|\rho|
 \, d\omega\, d\alpha\, d x_1\, d x.
 \label{eq:rho-area-formula}
\end{align}
In particular, \(\widetilde\rho\in L^1\).  Fubini's theorem therefore
defines the four marginals
\begin{align*}
 G^1(z)&=\int_X\int_0^1\int_{\Sph}
 \widetilde\rho(z,x_1',\theta,\omega)
 \, d\omega\, d\theta\, d x_1',\\
 G^2(z)&=\int_X\int_0^1\int_{\Sph}
 \widetilde\rho(x',z,\theta,\omega)
 \, d\omega\, d\theta\, d x',\\
 L^1(z)&=\int_X\int_0^1\int_{\Sph}
 \rho(z,x_1,\alpha,\omega)
 \, d\omega\, d\alpha\, d x_1,\\
 L^2(z)&=\int_X\int_0^1\int_{\Sph}
 \rho(x,z,\alpha,\omega)
 \, d\omega\, d\alpha\, d x.
\end{align*}
The collision invariants give
\begin{equation}\label{eq:Wpair}
 W(x')+W(x_1')=W(x)+W(x_1).
\end{equation}
Applying \eqref{eq:rho-area-formula} to the two gain marginals, and then
using \eqref{eq:Wpair},  
\begin{align}
    \|G^1\|_\X+\|G^2\|_\X &\leq \int_{X^2} \int_0^1\int_{\Sph} [W(x')+W(x_1')]|\widetilde{\rho}|\, d\omega\, d\theta\, dx_1'\, dx'\notag\\
    &= \int_{X^2}\int_0^1\int_{\Sph} [W(x)+W(x_1)]|\rho| \, d\omega\, d\alpha \, dx_1\, dx.
\end{align}
All four weighted estimates combine into
\begin{align}
 &\|G^1\|_{\X}+\|G^2\|_{\X}
 +\|L^1\|_{\X}+\|L^2\|_{\X}\notag\\
 &\quad\leq
 2\int_{X^2}\int_0^1\int_{\Sph}
 |\rho|[W(x)+W(x_1)]
 \, d\omega\, d\alpha\, d x_1\, d x\notag\\
 &\quad\leq\frac{C_N}{2}\int_{X^2}
 [W(x)+W(x_1)]
 [|g(x)||h(x_1)|+|h(x)||g(x_1)|]
 \, d x_1\, d x\notag\\
 &\quad=C_N\bigl(
 \|g\|_{\X}\|h\|_{L^1(X)}
 +\|g\|_{L^1(X)}\|h\|_{\X}\bigr).
 \label{eq:four-marginals-condensed}
\end{align}
Thus we define
\begin{equation}\label{eq:QN-marginals}
 \mathbf Q_N(g,h):=G^1+G^2-L^1-L^2\in\X
\end{equation}
and \eqref{eq:boundedQestimate} follows. Since $\rho$ is bilinear with respect to $g,h$, so is $\mathbf Q_N(g,h)$.  Changing variables in the gain
terms by \eqref{eq:rho-area-formula} gives, for every
\(|\psi|\lesssim W\) so that the following integrals are absolutely convergent,
\begin{align*}
    \int_X G^1(x')\psi(x')\, dx'&=\int_{X^2} \int_0^1\int_{\Sph}  \widetilde{\rho}\psi(x')\, d\omega d\theta dx_1' dx'=\int_{X^2} \int_0^1\int_{\Sph}  \rho\psi(x') d\omega d\alpha dx_1 dx,\\
    \int_X G^2(x_1')\psi(x_1') dx_1'&=\int_{X^2} \int_0^1\int_{\Sph}  \widetilde{\rho}\psi(x_1') d\omega d\theta dx_1' dx'=\int_{X^2} \int_0^1\int_{\Sph}  \rho\psi(x'_1) d\omega d\alpha dx_1 dx.
\end{align*}
Hence
\[
 \int_X\mathbf Q_N(g,h)\psi\, d x
 =\int_{X^2}\int_0^1\int_{\Sph}\rho\,
 [\psi(x')+\psi(x_1')-\psi(x)-\psi(x_1)]
 \, d\omega\, d\alpha\, d x_1\, d x,
\]
which is \eqref{eq:polarizedQN}. Taking $g=h$ gives \eqref{eq:strong-weakQN}.

Since \(W\geq1\), \(\norm \cdot_{L^1(X)}\leq\norm \cdot_{\X}\).
Taking \(h=g\) in \eqref{eq:boundedQestimate} proves
\eqref{eq:quadratic-growth}.  Finally,
\[
 \mathbf Q_N(g,g)-\mathbf Q_N(h,h)
 =\mathbf Q_N(g-h,g)+\mathbf Q_N(h,g-h),
\]
and two applications of \eqref{eq:boundedQestimate} give
\[
 \norm{\mathbf Q_N(g-h,g)}_{\X}
 \leq2C_N\norm g_{\X}\norm{g-h}_{\X},\qquad
 \norm{\mathbf Q_N(h,g-h)}_{\X}
 \leq2C_N\norm h_{\X}\norm{g-h}_{\X},
\]
which proves \eqref{eq:boundedQ-Lipschitz}.
\end{proof}

\subsection{Global nonnegative solution for the bounded kernel}

\begin{theorem}[Global bounded-kernel solution]\label{thm:bounded}
Let \(N\geq1\), and assume \eqref{eq:IC1}.  There exists a unique
\[
 f_N\in C^1([0,\infty);\X),\qquad
 f_N'(t)=\mathbf Q_N(f_N(t),f_N(t)),\qquad f_N(0)=f_0.
\]
It is nonnegative and satisfies
\begin{equation}\label{eq:boundedmoments}
 M_j(f_N(t))=M_j(f_0),\qquad j=0,1,2,\qquad t\geq0.
\end{equation}
For every measurable \(\psi\) such that \(|\psi|\lesssim W\),
\begin{align}
 \int_Xf_N(t,x)\psi(x)\, d x
 &=\int_Xf_0(x)\psi(x)\, d x+\int_0^t \mathcal Q_N(f_N(s),f_N(s))[\psi] ds.
 \label{eq:bounded-full-weak-identity}
\end{align}
\end{theorem}

\begin{proof}
Step 1. We first construct a nonnegative local solution.  If \(g\in\X\) is
nonnegative, then the two gain marginals in
\eqref{eq:QN-marginals} are nonnegative.  Define
\begin{equation}\label{eq:collision-frequency}
 \nu_{N,g}(x)=\int_X\int_0^1\int_{\Sph}
 a(m,m_1,\alpha)B_N\left(E,\frac{u}{|u|}\cdot\omega\right)g(x_1)
 \, d\omega\, d\alpha\, d x_1.
\end{equation}
Hence
\begin{align*}
    L^1(z)=\dfrac{1}{2}g(z)\int_X \int_0^1\int_{\Sph} a(m,m_1,\alpha) B_N\left(E, \frac{v-v_1}{|v-v_1|}\cdot\omega\right) g(x_1)d\omega d\alpha dx_1 =\dfrac{\nu_{N,g}(z)g(z)}{2}.
\end{align*}
Taking $\widetilde \omega=-\omega$, the substitution
\((x,x_1,\alpha,\omega)\mapsto(x_1,x,\alpha,\widetilde\omega)\) satisfies $d\omega=d\widetilde{\omega}$ and
\[
 a(m_1,m,\alpha)=a(m,m_1,\alpha),\quad
 \frac{m_1m}{m_1+m}|v_1-v|^2=E,\quad
 \frac{v_1-v}{|v_1-v|}\cdot\widetilde\omega
 =\frac{v-v_1}{|v-v_1|}\cdot\omega,
\]which gives
\begin{align*}
    &L^2(z) =\dfrac{1}{2}g(z)\int_X\int_0^1\int_{\Sph} a(m_1,m,\alpha) B_N\left(E,\frac{v_1-v}{|v_1-v|}\cdot\omega\right) g(x) d\omega d\alpha dx\\& =\dfrac{1}{2}g(z)\int_X \int_0^1\int_{\Sph} a(m,m_1,\alpha) B_N\left(E,\dfrac{v-v_1}{|v-v_1|}\cdot \widetilde\omega\right) g(x_1) d\widetilde\omega d\alpha dx_1=\dfrac{1}{2}\nu_{N,g}(z)g(z).
\end{align*}
Hence
\begin{equation}\label{eq:gainloss}
 \mathbf Q_N(g,g)=\mathbf Q_N^+(g,g)-\nu_{N,g}g,\quad
 \mathbf Q_N^+(g,g)\geq0,\quad
 0\leq\nu_{N,g}\leq C_N\norm g_{L^1(X)}\leq C_N\|g\|_\X,
\end{equation}
where $\mathbf Q_N^+(g,h)=G^1+G^2$. The assertion is immediate if \(C_N=0\), so suppose \(C_N>0\).  Fix
\begin{equation}\label{eq:local-parameters}
 R>\norm{f_0}_{\X},\qquad \lambda=C_NR,\qquad
 H(t,g)=\lambda g+e^{-\lambda t}\mathbf Q_N(g,g).
\end{equation}
Since
\[
\lambda-e^{-{\lambda t}}\nu_{N,g}(x)\ge \lambda-\nu_{N,g}(x
)\geq C_N(R-\|g\|_\X),\]
for \(g\geq0\) with \(\norm g_{\X}\leq R\), \eqref{eq:quadratic-growth} and \eqref{eq:gainloss}
gives
\begin{equation}\label{eq:H-estimates}
 H(t,g)=e^{-\lambda t}\mathbf Q_N^+(g,g)
 +[\lambda-e^{-\lambda t}\nu_{N,g}]g\geq0,\qquad
 \norm{H(t,g)}_{\X}\leq\lambda R+2C_NR^2.
\end{equation}
If \(\norm g_{\X},\norm h_{\X}\leq R\), since
\[
H(t,g)-H(t,h)=\lambda(g-h)+e^{-\lambda t} (\mathbf Q_N(g,g)-\mathbf Q_N(h,h)),
\]then \eqref{eq:boundedQ-Lipschitz} gives
\begin{equation}\label{eq:H-Lipschitz-condensed}
 \norm{H(t,g)-H(t,h)}_{\X}
 \leq(\lambda+4C_NR)\norm{g-h}_{\X}.
\end{equation}
Choose
\begin{equation}\label{eq:tau-choice}
 0<\tau\leq\min\left\{
 \frac{R-\norm{f_0}_{\X}}{\lambda R+2C_NR^2},
 \frac1{2(\lambda+4C_NR)}\right\}.
\end{equation}
Let \(\mathcal C_{R,\tau}\) be the set of all
\(g\in C([0,\tau];\X)\) such that \(g(t)\geq0\) and
\(\sup_t\norm{g(t)}_{\X}\leq R\).  This set is complete: if
\(q_k\geq0\) and \(q_k\to q\) in \(\X\), then
\[
 \int_XWq_-\, d x\leq\int_XW|q-q_k|\, d x\longrightarrow0,
\]
so the positive cone, and hence \(\mathcal C_{R,\tau}\), is closed.
Define
\[
 (\mathcal Tg)(t)=f_0+\int_0^tH(s,g(s))\, d s.
\]
Equations \eqref{eq:H-estimates}--\eqref{eq:tau-choice} give
\begin{align*}
 \mathcal Tg(t)&\geq0,\\
 \sup_{0\leq t\leq\tau}\norm{\mathcal Tg(t)}_{\X}
 &\leq\norm{f_0}_{\X}+\tau(\lambda R+2C_NR^2)\leq R,\\
 \norm{\mathcal Tg-\mathcal Th}_{C([0,\tau];\X)}
 &\leq\tau(\lambda+4C_NR)
 \norm{g-h}_{C([0,\tau];\X)}
 \leq\frac12\norm{g-h}_{C([0,\tau];\X)}.
\end{align*}
Thus \(\mathcal T\) is a contraction of the complete set
\(\mathcal C_{R,\tau}\) into itself.  By Banach's fixed point theorem,
there exists a unique fixed point \(g\in\mathcal C_{R,\tau}\).
The identity
\[
 g(t)=f_0+\int_0^t
 [\lambda g(s)+e^{-\lambda s}\mathbf Q_N(g(s),g(s))]\, d s
\]
and continuity of \(\mathbf Q_N\) imply \(g\in C^1([0,\tau];\X)\).
For \(f_N(t)=e^{-\lambda t}g(t)\), bilinearity gives
\[
 f_N'(t)
 =-\lambda e^{-\lambda t}g(t)+e^{-\lambda t}g'(t)
 =e^{-2\lambda t}\mathbf Q_N(g(t),g(t))
 =\mathbf Q_N(f_N(t),f_N(t)).
\]
Therefore \(f_N\geq0\) is a local solution.

Step 2. Local solutions are unique.  Indeed, if \(f,\widetilde f\) solve the
equation on \([0,T]\) with the same datum and
\(R_T=\sup_{0\leq t\leq T}(\norm{f(t)}_{\X}
+\|\widetilde f(t)\|_{\X})<\infty\), then by \cref{lem:boundedQ},
\begin{align*}
f(t)-\widetilde f(t) &=\int_0^t \mathbf Q_N(f(s),f(s))-\mathbf Q_N(\widetilde f(s),\widetilde f(s)) ds,\\
    \|f(t)-\widetilde f(t)\|_{\X}
 &\leq2C_NR_T\int_0^t\|f(s)-\widetilde f(s)\|_{\X}\, d s.
\end{align*}
With \(Z(t)=\displaystyle \int_0^t\|f(s)-\widetilde f(s)\|_{\X}\, d s\),
\[
 \frac{ d}{ d t}\bigl(e^{-2C_NR_Tt}Z(t)\bigr)
 =e^{-2C_NR_Tt}[Z'(t)-2C_NR_TZ(t)]\leq0,
 \qquad Z(0)=0.
\]
Since the left-hand function is nonnegative, it vanishes; hence
\(f=\widetilde f\).  

Step 3. Local solutions can consequently be joined uniquely
on overlaps.  Denote their maximal interval by \([0,T_*)\); every local
construction preserves nonnegativity. We next establish the identities needed to rule out \(T_*<\infty\).
For any measurable \(|\psi|\lesssim W\), collision identities gives
\begin{equation}\label{eq:Delta-psi-bound}
 |\Delta\psi|
 \leq C[W(x')+W(x_1')+W(x)+W(x_1)]
 =2C[W(x)+W(x_1)].
\end{equation}
Thus, for \(t<T_*\),
\begin{align}
 &\frac12\int_{X^2}\int_0^1\int_{\Sph}
 aB_N|\Delta\psi|f_N(t,x)f_N(t,x_1)
 \, d\omega\, d\alpha\, d x_1\, d x\notag\\
 &\qquad\leq2CC_N\norm{f_N(t)}_{\X}
 \norm{f_N(t)}_{L^1(X)}<\infty.
 \label{eq:bounded-collision-absolute}
\end{align}
The functional \(q\mapsto\int_Xq\psi\, d x\) is continuous on \(\X\),
with absolute value at most \(C\norm q_{\X}\).  Applying it to the
differential equation and using \eqref{eq:polarizedQN} yields
\begin{equation}\label{eq:bounded-differential-weak}
 \frac{ d}{ d t}\int_Xf_N(t,x)\psi(x)\, d x
 =\frac12\int_{X^2}\int_0^1\int_{\Sph}
 aB_N\Delta\psi\,f_N(t,x)f_N(t,x_1)
 \, d\omega\, d\alpha\, d x_1\, d x.
\end{equation}
For \(\psi=1,m,m|v|^2\), respectively,
\[
 \Delta\psi=0,\qquad
 \Delta\psi=m'+m_1'-m-m_1=0,\qquad
 \Delta\psi=m'|v'|^2+m_1'|v_1'|^2-m|v|^2-m_1|v_1|^2=0.
\]
Integrating \eqref{eq:bounded-differential-weak} therefore proves
\eqref{eq:boundedmoments} on \([0,T_*)\).  Since \(f_N\geq0\),
\begin{equation}\label{eq:constant-X-norm}
 \norm{f_N(t)}_{\X}
 =M_0(f_0)+M_1(f_0)+M_2(f_0)=:C_0,\qquad 0\leq t<T_*.
\end{equation}

Suppose that \(T_*<\infty\).  For \(0\leq s<t<T_*\),
\eqref{eq:quadratic-growth} and \eqref{eq:constant-X-norm} imply
\begin{equation}\label{eq:time-Lipschitz-bounded}
 \norm{f_N(t)-f_N(s)}_{\X}
 \leq\int_s^t2C_N\norm{f_N(r)}_{\X}^2\, d r
 =2C_NC_0^2(t-s).
\end{equation}
Hence \(f_N(t)\) is Cauchy in \(\X\) as \(t\uparrow T_*\), and there is
\(f_*\in\X\) such that
\begin{equation}\label{eq:maximal-endpoint}
 \lim_{t\uparrow T_*}\norm{f_N(t)-f_*}_{\X}=0.
\end{equation}
Indeed, if \(t_k\uparrow T_*\), \eqref{eq:time-Lipschitz-bounded} makes
\(f_N(t_k)\) Cauchy, and for \(t<t_k\),
\[
 \norm{f_N(t)-f_*}_{\X}
 \leq2C_NC_0^2(t_k-t)+\norm{f_N(t_k)-f_*}_{\X},
\]
which gives \eqref{eq:maximal-endpoint}.  Closedness of the positive cone
gives \(f_*\geq0\), while norm continuity gives \(\norm{f_*}_{\X}=C_0\).

Apply the local construction to the datum \(f_*\), with \(R>C_0\), and
denote the resulting solution by \(\widehat f\in C^1([0,\tau];\X)\).
Define
\[
 \overline f_N(t)=
 \begin{cases}
 f_N(t),&0\leq t<T_*,\\
 \widehat f(t-T_*),&T_*\leq t\leq T_*+\tau.
 \end{cases}
\]
It is continuous at \(T_*\) by \eqref{eq:maximal-endpoint}.  Moreover,
\eqref{eq:boundedQ-Lipschitz} gives
\[
 \mathbf Q_N(f_N(t),f_N(t))
 \longrightarrow\mathbf Q_N(f_*,f_*)\quad\text{in }\X
 \qquad \text{as }t\uparrow T_*.
\]
For \(h<0\) and \(h>0\), respectively, its difference quotients at the
junction are
\[
 \frac1{-h}\int_{T_*+h}^{T_*}
 \mathbf Q_N(f_N(r),f_N(r))\, d r,\qquad
 \frac1h\int_0^h
 \mathbf Q_N(\widehat f(r),\widehat f(r))\, d r.
\]
Both converge in \(\X\) to \(\mathbf Q_N(f_*,f_*)\) as \(h\to0\).
Thus \(\overline f_N\in C^1([0,T_*+\tau];\X)\) solves the same equation,
contradicting maximality.  Therefore \(T_*=\infty\).

Finally, integrate \eqref{eq:bounded-differential-weak} from \(0\) to
\(t\).  Equations \eqref{eq:bounded-collision-absolute} and
\eqref{eq:constant-X-norm} give
\[
 \frac12\int_0^t\int_{X^2}\int_0^1\int_{\Sph}
 aB_N|\Delta\psi|f_N(s,x)f_N(s,x_1)
 \, d\omega\, d\alpha\, d x_1\, d x\, d s
 \leq2CC_NC_0^2t<\infty,
\]
which proves both \eqref{eq:bounded-full-weak-identity} and its absolute
convergence.
\end{proof}
\section{Uniform estimates without entropy}\label{sec:compact-est}

All constants in this section are independent of \(N\).  Their dependence on
\(T,d,\gamma,\norm a_\infty,\norm b_{L^1}\), and the three initial moments is
allowed.

\subsection{Hard-potential collision-rate bounds}

\begin{lemma}[Global and localized rate bounds]\label{lem:rate}
Let $f_N$ be the truncated solutions of the bounded-kernel equations in \cref{sec:bounded}. For every measurable \(A\subset X\), every \(N\), and every \(t\geq0\),
\begin{align}
 &\int_{A\times X}\int_0^1\int_{\Sph}
 aB_N f_N(t,x)f_N(t,x_1)
 \, d\omega\, d\alpha\, d x_1\, d x\notag\\
 &\quad\leq \norm a_\infty\norm b_{L^1}M_2(f_0)^\gamma
 \left[
 M_0(f_0)\left(\int_A f_N(t,x)\, d x\right)^{1-\gamma}
 +M_0(f_0)^{1-\gamma}\int_Af_N(t,x)\, d x
 \right].\label{eq:localrate}
\end{align}
In particular,
\begin{align}
 \int_{X^2}\int_0^1\int_{\Sph}aB_N f_N(x)f_{N}(x_1)
  d\omega d\alpha dx_1 dx &\leq2\norm a_\infty\norm b_{L^1}
 M_0(f_0)^{2-\gamma}M_2(f_0)^\gamma.
 \label{eq:globalrate}
\end{align}
\end{lemma}

\begin{proof}
Since $\gamma\le 1$, one has
\( E^\gamma\leq (m|v|^2)^\gamma+(m_1|v_1|^2)^\gamma.\)
H\"older's inequality gives
\begin{align}
 \int_A(m|v|^2)^\gamma f_N\, d x
 &=\int_A(m|v|^2f_N)^\gamma f_N^{1-\gamma} dx\leq M_2(f_0)^\gamma
       \left(\int_Af_N\, d x\right)^{1-\gamma},\label{eq:holderlocal}\\
 \int_X(m|v|^2)^\gamma f_N dx&=\int_X (m|v|^2f_N)^\gamma f_N^{1-\gamma}\, d x
 \leq M_2(f_0)^\gamma M_0(f_0)^{1-\gamma}.
 \label{eq:holderglobal}
\end{align}
Consequently, using $B_N\leq E^\gamma b,$
\begin{align*}
 &\int_{A\times X}\int_0^1\int_{\Sph}aB_Nf_N(x)f_N(x_1)
 \, d\omega\, d\alpha\, d x_1\, d x\\
 &\leq\norm a_\infty\norm b_{L^1}
 \left[
 \left(\int_A(m|v|^2)^\gamma f_N\, d x\right)M_0(f_0)
 +\left(\int_Af_N\, d x\right)
  \left(\int_X(m_1|v_1|^2)^\gamma f_N\, d x_1\right)
 \right]\\
 &\leq\norm a_\infty\norm b_{L^1}M_2(f_0)^\gamma
 \left[
 M_0(f_0)\left(\int_Af_N\, d x\right)^{1-\gamma}
 +M_0(f_0)^{1-\gamma}\int_Af_N\, d x
 \right].
\end{align*}
This is \eqref{eq:localrate}; setting \(A=X\) gives
\eqref{eq:globalrate}.
\end{proof}

\subsection{The small-mass gain bootstrap}

A central difficulty in the compactness argument is the possible concentration of the number density near $m=0$. Neither the conserved mass nor the kinetic energy controls the population in $\{m<r\}$, while the collision gain may continuously create particles of arbitrarily small mass. We overcome this difficulty through a small-mass gain bootstrap: the gain into $\{m<r\}$ is controlled by a superlinear quantity at a larger mass scale $\rho$, together with a geometric leakage term of order $r/\rho$. The superlinear exponent $2-\gamma>1$ is the mechanism that closes the resulting scale recursion uniformly in $N$ and $t\in [0,T]$ for $T>0$.

Let $f_N$ be the truncated solutions of the bounded-kernel equations in \cref{sec:bounded}. For \(r>0\), set
\begin{equation}\label{eq:Fr}
 F_N(r,t)=\int_0^r\int_{\R^d}f_N(t,m,v)\, d v\, d m.
\end{equation}

\begin{proposition}[No concentration at zero mass]\label{prop:smallmass}
For every \(T>0\),
\begin{equation}\label{eq:smallmass}
 \lim_{r\downarrow0}\sup_{N\geq1}\sup_{0\leq t\leq T}F_N(r,t)=0.
\end{equation}
\end{proposition}

\begin{proof}
Fix \(0<r<\rho<1\).  The bounded measurable test
\(\one_{\{m<r\}}\) is admissible in \cref{thm:bounded}, taking $S=m+m_1$, then
\begin{align}
    &\dfrac{d}{dt}F_N(r,t) =\dfrac{d}{dt}\int_X f_N\one_{\{m<r\}} dx=\mathcal Q_N(f_N,f_N)[\one_{\{m<r\}} ]\notag\\&=\dfrac{1}{2}\int_{X^2}\int_0^1\int_{\Sph} a B_N(\one_{\{\alpha S<r\}}+\one_{\{(1-\alpha) S<r\}}-\one_{\{m<r\}}-\one_{\{m_1<r\}})f_N(x)f_N(x_1)d\omega d\alpha dx_1 dx\notag\\&\leq \dfrac{1}{2}\int_{X^2}\int_0^1\int_{\Sph} a B_N(\one_{\{\alpha S<r\}}+\one_{\{(1-\alpha) S<r\}})f_N(x)f_N(x_1)d\omega d\alpha dx_1 dx
\end{align}

We divide the region $X^2=\{S<\rho\}\cup\{S\ge \rho\}$. Since $\{S<\rho\}\subset \{m<\rho\}\cap\{m_1<\rho\}$,
\begin{align}
    &\dfrac{1}{2}\int_{\{S<\rho\}}\int_0^1\int_{\Sph} a B_N(\one_{\{\alpha S<r\}}+\one_{\{(1-\alpha) S<r\}})f_N(x)f_N(x_1)d\omega d\alpha dx_1 dx\notag\\ &\leq\|a\|_\infty \|b\|_{L^1} \int_{\{S\leq \rho\}} E^\gamma f_N(x)f_N(x_1) dx_1 dx\notag\\ & \leq\|a\|_\infty\|b\|_{L^1}  \int_{\{m,m_1<\rho\}} [(m|v|^2)^\gamma+(m_1|v_1|^2)^\gamma]f_N(x)f_N(x_1)dx_1 dx\notag\\ &=2\|a\|_\infty\|b\|_{L^1} F_N(\rho,t)\int_{\{m<\rho\}} (m|v|^2)^\gamma f_N(x)dx\leq 2\|a\|_\infty\|b\|_{L^1} M_2(f_0)^\gamma F_N(\rho,t)^{2-\gamma},
\end{align}
with the last inequality given by \eqref{eq:holderlocal}.

On \(\{S\geq\rho\}\), since
\begin{equation}
    \int_0^1\one_{\{\alpha S<r\}} d\alpha=\min\left\{1,\dfrac{r}{S}\right\}\leq \dfrac{r}{\rho},\qquad \int_0^1\one_{\{(1-\alpha) S<r\}} d\alpha\leq \dfrac{r}{\rho},
\end{equation}
with \eqref{eq:holderlocal}, we have
\begin{align}
    &\dfrac{1}{2}\int_{\{S\ge\rho\}}\int_0^1\int_{\Sph} a B_N(\one_{\{\alpha S<r\}}+\one_{\{(1-\alpha) S<r\}})f_N(x)f_N(x_1)d\omega d\alpha dx_1 dx\notag\\ &\leq \|a\|_\infty\|b\|_{L^1}  \dfrac{r}{\rho} \int_{X^2} E^\gamma f_N(x)f_N(x_1)dx_1 dx\leq 2\|a\|_\infty\|b\|_{L^1} M_0(f_0)^{2-\gamma}M_2(f_0)^\gamma \dfrac{r}{\rho}.
\end{align}
We obtain,
for a constant \(C\) independent of \(N,r,\rho,t\),
\begin{equation}\label{eq:Fdiff}
 \frac{ d}{ d t}F_N(r,t)
 \leq C F_N(\rho,t)^{2-\gamma}+C\frac r\rho.
\end{equation}
The derivative exists because \(f_N\in C^1(L^1)\) and the indicator is an
\(L^\infty\) test.

Let \(I=[t_0,t_0+\tau]\subset[0,T]\), choose \(\rho=\sqrt r\), and integrate
\eqref{eq:Fdiff} from \(t_0\) to \(t\in I\):
\begin{equation}\label{eq:Finterval}
 F_N(r,t)\leq F_N(r,t_0)
 +C\int_{t_0}^tF_N(\sqrt r,s)^{2-\gamma}\, d s+C\tau\sqrt r.
\end{equation}
Define
\(
 L_I=\limsup_{r\downarrow0}\sup_N\sup_{t\in I}F_N(r,t).
\)
If
\(\lim_{r\downarrow0}\sup_NF_N(r,t_0)=0\), then taking the upper limit in
\eqref{eq:Finterval} yields
\begin{equation}\label{eq:Lbootstrap}
 L_I\leq C\tau L_I^{2-\gamma}.
\end{equation}
Because \(0\leq L_I\leq M_0(f_0)\), choose \(\tau>0\), independently of
\(t_0,N\), so that
\(
 C\tau M_0(f_0)^{1-\gamma}<1.
\)
Then \eqref{eq:Lbootstrap} implies
\(
 L_I\leq C\tau M_0(f_0)^{1-\gamma}L_I<L_I
\)
unless \(L_I=0\); hence \(L_I=0\).  At \(t_0=0\), one has
\[
 \sup_NF_N(r,0)
 =\int_X\one_{\{m<r\}}f_0(m,v)\, d m\, d v
 \longrightarrow0
 \qquad\text{as }r\downarrow0
\]
by dominated convergence, since
\(\one_{\{m<r\}}f_0\to0\) almost everywhere and
\(0\leq\one_{\{m<r\}}f_0\leq f_0\in L^1(X)\).
Applying the same conclusion successively at
\(t_0=0,\tau,2\tau,\ldots\) covers \([0,T]\) in finitely many steps and proves
\eqref{eq:smallmass}.
\end{proof}

\subsection{Gain-set geometry and uniform integrability}

\cref{prop:smallmass} rules out concentration at the boundary $m=0$. To obtain weak compactness of $f_N$ in $L^1$, one must obtain uniform integrability required by the Dunford--Pettis theorem \cite{DiestelUhl1977}. Instead of using (relative) entropy estimates in classical Boltzmann theory \cite{DiPernaLions1989, Villani2002}, which will lead to additional structural assumptions of the mass detailed-balance weights, we resolve it through a second gain bootstrap, which uses the collision geometry and the associated marginals to propagate a uniform modulus of absolute continuity.

For \(0<\varepsilon<1/2\) and \(0<\kappa<1/4\), define the good collision set
\begin{equation}\label{eq:goodset}
 G_{\varepsilon,\kappa}=\left\{(x,x_1,\alpha,\omega)\in X^2\times (0,1)\times \mathbb S^{d-1}:
 \begin{aligned}
 &\varepsilon\leq\alpha\leq1-\varepsilon,\qquad
 \kappa\leq\theta\leq1-\kappa,\\
 &\abs{\theta-\alpha}\geq \kappa,\qquad
 \abs{\theta-(1-\alpha)}\geq \kappa
 \end{aligned}\right\}.
\end{equation}

\begin{lemma}[The two one-particle Jacobians]\label{lem:goodjac}
Fix \(x,\alpha,\omega\).  On \(G_{\varepsilon,\kappa}\), both maps
\(
 x_1\mapsto x', x_1\mapsto x_1'
\)
are one-to-one local diffeomorphisms.  There is an explicit number
\begin{equation}\label{eq:jgood}
 j_{\varepsilon,\kappa}
 =\varepsilon\left(\frac{\kappa}{1+(\varepsilon \kappa)^{-1/2}}\right)^d>0
\end{equation}
such that the absolute value of each Jacobian is
\(j_{\varepsilon,\kappa}\).  Consequently, for every measurable \(A\subset X\),
\begin{align}
 \abs{\{x_1:(x,x_1,\alpha,\omega)\in G_{\varepsilon,\kappa},\ x'\in A\}}
 &\leq j_{\varepsilon,\kappa}^{-1}\abs A,\label{eq:area1}\\
 \abs{\{x_1:(x,x_1,\alpha,\omega)\in G_{\varepsilon,\kappa},\ x_1'\in A\}}
 &\leq j_{\varepsilon,\kappa}^{-1}\abs A.\label{eq:area2}
\end{align}
\end{lemma}

\begin{proof}
Holding \(m,v,\alpha,\omega\) fixed, the collision geometry gives
\[
 m'=\alpha(m+m_1),\quad
 v'=\theta v+(1-\theta)v_1+cR_\omega(v-v_1),
 \]
where
\(
 c=\left({\theta(1-\theta)(1-\alpha)}/\alpha\right)^{1/2}.
\)
The derivative with respect to \((m_1,v_1)\) is block lower triangular.  Its
upper-left entry is \(\alpha\), and its velocity block is
\((1-\theta)I-cR_\omega\).  The reflection has eigenvalue \(-1\) in the
\(\omega\)-direction and eigenvalue \(1\) on \(\omega^\perp\).  Thus the
velocity block has eigenvalues
\[
 \lambda_+=(1-\theta)+c
 \quad\text{once},\qquad
 \lambda_-=(1-\theta)-c
 \quad\text{with multiplicity }d-1.
\]
Writing
\(q=[\theta(1-\alpha)/(\alpha(1-\theta))]^{1/2}\), direct rationalization
gives
\begin{equation}\label{eq:lambdaminus}
 \abs{\lambda_-}
 =(1-\theta)\abs{1-q}
 =\frac{\abs{\alpha-\theta}}{\alpha(1+q)}.
\end{equation}
On the good set, 
\[
q\leq(\varepsilon \kappa)^{-1/2}, \quad
\lambda_+\geq1-\theta\geq \kappa, \quad
\abs{\lambda_-}\geq\kappa/(1+(\varepsilon\kappa)^{-1/2}).
\]
The first Jacobian is therefore at least
\[
\varepsilon\kappa[\kappa/(1+(\varepsilon\kappa)^{-1/2})]^{d-1} \ge \varepsilon\left(\frac{\kappa}{1+(\varepsilon \kappa)^{-1/2}}\right)^d\,
\] 
which is at least
\eqref{eq:jgood}.

For the second output,
\[
 m_1'=(1-\alpha)(m+m_1),\quad
 v_1'=\theta v+(1-\theta)v_1-d_0R_\omega(v-v_1),
\]
where
\(d_0=[\theta(1-\theta)\alpha/(1-\alpha)]^{1/2}\).  The mass derivative is
\(1-\alpha\), and the velocity block is
\((1-\theta)I+d_0R_\omega\).  Its eigenvalues are
\((1-\theta)-d_0\) in the \(\omega\)-direction and
\((1-\theta)+d_0\) with multiplicity \(d-1\).  Rationalization now gives
\[
 \abs{(1-\theta)-d_0}
 =\frac{\abs{1-\alpha-\theta}}
 {(1-\alpha)\left(1+\sqrt{\theta\alpha/((1-\theta)(1-\alpha))}\right)}
 \geq\frac{\kappa}{1+(\varepsilon \kappa)^{-1/2}}.
\]
Every other velocity eigenvalue is at least \(1-\theta\geq \kappa\), so the second
Jacobian is at least \((1-\alpha)[\kappa/(1+(\varepsilon \kappa)^{-1/2})]\kappa^{d-1}\),
again no smaller than \eqref{eq:jgood}.

For the first map, \(m'\) uniquely gives \(m_1=m'/\alpha-m\), after which
the invertible velocity block uniquely gives \(v_1\).  The second map has the
same property with \(m_1'= (1-\alpha)(m+m_1)\).  Thus both maps are one-to-one
on the good set.  Applying the area formula \cite{EvansGariepy2015} and the Jacobian lower bound gives
\eqref{eq:area1} and \eqref{eq:area2}.
\end{proof}

For \(\sigma>0\), introduce the angular absolute-continuity modulus
\begin{equation}\label{eq:Bmod}
 \mathcal B(\sigma)=
 \sup_{e\in\Sph}\ \sup_{\substack{O\subset\Sph\text{ measurable}\\
                                  \abs O\leq\sigma}}
 \int_O b(e\cdot\omega)\, d\omega.
\end{equation}

\begin{lemma}[Uniform angular absolute continuity]\label{lem:Bmod}
One has \(\mathcal B(\sigma)\to0\) as \(\sigma\downarrow0\).
\end{lemma}

\begin{proof}
Fix \(e_0\in\Sph\).  For every \(e\in\Sph\), choose an orthogonal map
\(Q_e\) with \(Q_ee_0=e\).  The change of variables
\(\omega=Q_e\zeta\) preserves surface measure and gives
\[
 \int_Ob(e\cdot\omega)\, d\omega
 =\int_{Q_e^{-1}O}b(e_0\cdot\zeta)\, d\zeta.
\]
The supremum in \eqref{eq:Bmod} is therefore the supremum of the integral of
the single \(L^1(\Sph)\) function \(b(e_0\cdot\zeta)\) over sets of measure at
most \(\sigma\).  The absolute continuity of the Lebesgue integral proves the
claim.
\end{proof}

Define
\begin{equation}\label{eq:UN}
 U_N(q,t):=\sup_{\substack{A\subset X\text{ measurable}\\\abs A\leq q}}
 \int_Af_N(t,x)\, d x.
\end{equation}
and for $D\subset X^2\times (0,1)\times\Sph$ measurable, define the collision-rate measure
\begin{equation}\label{eq:RN}
    \mathcal R_N(D,t):=\int_{X^2}\int_0^1\int_{\Sph} \one_D a B_N f_N(x)f_N(x_1) d\omega d\alpha dx_1 dx.
\end{equation}

\begin{lemma}[Bad collisions and good gain]\label{lem:goodgain}
Fix \(T>0\). Let $f_N$ be the truncated solutions of the bounded-kernel equations in \cref{sec:bounded}.  For \(\varepsilon \in (0,1/2), \kappa \in (0,1/4)\) and \(L>1\), let
\[
 z_{\kappa,L}:=\sup_N\sup_{0\leq t\leq T}F_N(\kappa L,t)+\frac{2M_1(f_0)}L,\qquad R_{\varepsilon,\kappa,L}:=\sup_N\sup_{t\in [0,T]} \mathcal R_N(G_{\varepsilon,\kappa}^c,t)
\]
Then 
\begin{align}
 R_{\varepsilon,\kappa,L}\lesssim 
 (\varepsilon+\kappa)M_0(f_0)^{2-\gamma}M_2(f_0)^\gamma
 +M_0(f_0)M_2(f_0)^\gamma z_{\kappa,L}^{1-\gamma}+M_0(f_0)^{1-\gamma}M_2(f_0)^\gamma z_{\kappa,L}.
 \label{eq:Rbad}
\end{align}
Moreover, for a measurable set \(A\subset X\) with $|A|<\infty$, define its total gain by
\begin{align}
 G_N(A,t):=\frac12\int_{X^2}\int_0^1\int_{\Sph}
 aB_N[\one_A(x')+\one_A(x_1')]f_N(t,x)f_N(t,x_1)
 \, d\omega\, d\alpha\, d x_1\, d x.
 \label{eq:GNdef}
\end{align}
Then, for every \(\sigma>0\),
\begin{align}
 G_N(A,t)&\leq R_{\varepsilon,\kappa,L}
 +C_1\mathcal B(\sigma)
 +C_2U_N\left(\frac{|\Sph|\abs A}
 {j_{\varepsilon,\kappa}\sigma},t\right),
 \label{eq:gainestimate}
\end{align}
where \(C_1,C_2\) are independent of all displayed auxiliary parameters,
\(N,t,A\).
\end{lemma}

\begin{proof}
Part 1: bad collision estimate. Notice that
\begin{align*}
G_{\varepsilon,\kappa}^c &=D_\alpha^\varepsilon\cup D_1^\kappa\cup D_2^\kappa\cup D_3^\kappa \cup D_4^\kappa,\\
    D_\alpha^\varepsilon &= \{(x,x_1,\alpha,\omega)\in X^2\times (0,1)\times \Sph: \alpha<\varepsilon \text{ or }\alpha>1-\varepsilon\},\\
    D_1^\kappa&=\{(x,x_1,\alpha,\omega)\in X^2\times (0,1)\times \Sph: |\theta-\alpha|<\kappa\},\\
    D_2^\kappa&=\{(x,x_1,\alpha,\omega)\in X^2\times (0,1)\times \Sph: |\theta-(1-\alpha)|<\kappa\},\\
    D_3^\kappa &= \{(x,x_1,\alpha,\omega)\in X^2\times (0,1)\times \Sph: \theta<\kappa\},\\
    D_4^\kappa &= \{(x,x_1,\alpha,\omega)\in X^2\times (0,1)\times \Sph: \theta>1-\kappa\}.
\end{align*}
Since the collision integrand is nonnegative,
\[
\mathcal R_N(G^c_{\varepsilon,\kappa},t)\leq \mathcal R_N(D_\alpha^\varepsilon,t)+\sum_{i=1}^4\mathcal R_N(D_i^\kappa,t).
\]
Since
\[
\int_0^1(\one_{D_\alpha^\varepsilon}+\one_{D_1^\kappa}+\one_{D_2^\kappa})d\alpha\leq 2\varepsilon+4\kappa,
\]
with $B_N\leq E^\gamma b(\xi)$ and  \eqref{eq:globalrate}, we have
\begin{align}\label{eq:Ralpha12}
    &\mathcal   R_N(D_\alpha^\varepsilon,t)+\mathcal R_N(D_1^\kappa,t)+\mathcal R_N(D_2^\kappa,t)\notag\\&\qquad\leq(2\varepsilon+4\kappa)\|a\|_\infty \|b\|_{L^1} \int_{X^2}[ (m|v|^2)^\gamma+ (m_1|v_1|^2)^\gamma]f_N(t,x)f_N(t,x_1)dx_1 dx\notag\\ &\qquad \leq 2(2\varepsilon+4\kappa)\|a\|_\infty\|b\|_{L^1} M_0(f_0)^{2-\gamma} M_2(f_0)^\gamma\lesssim (\varepsilon+\kappa)M_0(f_0)^{2-\gamma}M_2(f_0)^\gamma.
\end{align}

Define $A_{\kappa,L}=\{(m,v)\in X:m<\kappa L\}\cup\{(m,v)\in X:m>L/2\}$, we claim that
\[D_3^\kappa\cup D_4^\kappa\subset [(A_{\kappa,L}\times X)\cup (X\times A_{\kappa,L})]\times  (0,1)\times\Sph. \]
Indeed, if $x\not\in A_{\kappa,L}$ and $x_1\not\in A_{\kappa,L}$, then $\kappa L\leq m,m_1\leq L/2$, which implies $m+m_1\leq L$, $\theta =m/(m+m_1)\geq \kappa L/L=\kappa$ and $1-\theta=m_1/(m+m_1)\geq \kappa L/L=\kappa$, i.e. $\kappa\leq\theta\leq 1-\kappa$.
Hence by Markov's inequality and \eqref{eq:boundedmoments},
\begin{align}
    \int_{A_{\kappa,L}}f_N(t,x)dx \leq F_N(\kappa L,t)+\int_{\{m>L/2\}} f_N(t,x) dx\leq F_N(\kappa L,t)+\dfrac{2M_1(f_0)}{L}\leq z_{\kappa,L}.
\end{align}
\cref{lem:rate} gives
\begin{align}
    &\mathcal R_N(A_{\kappa,L}\times X\times (0,1)\times \Sph,t)\notag\\ &\leq \|a\|_\infty\|b\|_{L^1} M_2(f_0)^\gamma \left[M_0(f_0)\left(\int_{A_{\kappa,L} }f_N(t,x) dx\right)^{1-\gamma}+M_0(f_0)^{1-\gamma} \int_{A_{\kappa,L} }f_N(t,x) dx\right]\notag\\&\leq \|a\|_\infty\|b\|_{L^1} M_2(f_0)^\gamma \left[M_0(f_0)z_{\kappa,L}^{1-\gamma}+M_0(f_0)^{1-\gamma} z_{\kappa,L}\right].
\end{align}
The map $(x,x_1,\alpha,\omega)\mapsto (x_1,x,\alpha,-\omega)$ with Jacobian 1 gives \[\mathcal R_N(A_{\kappa,L}\times X\times (0,1)\times \Sph,t)=\mathcal R_N(X\times A_{\kappa,L}\times (0,1)\times \Sph,t).\]
Hence
\begin{align}\label{eq:R34}
\mathcal R_N(D_3^\kappa,t)+\mathcal R_N(D_4^\kappa,t) &\leq  2\mathcal R_N(A_{\kappa,L}\times X\times (0,1)\times \Sph,t)\notag\\ &\lesssim M_0(f_0)M_2(f_0)^\gamma z_{\kappa,L}^{1-\gamma}+M_0(f_0)^{1-\gamma}M_2(f_0)^\gamma z_{\kappa,L}.
\end{align}
Hence  \eqref{eq:Ralpha12} and \eqref{eq:R34} gives \eqref{eq:Rbad}.

Part 2: angular-area estimate. Consider the
first output.  Set
\[
 O_A^{(1)}(x,x_1,\alpha)=\{\omega\in\Sph:(x,x_1,\alpha,\omega)\in
 G_{\varepsilon,\kappa},\ x'\in A\}.
\]
Fix $x,\alpha$. Fubini's theorem and \eqref{eq:area1} imply
\[
 \int_X|O_A^{(1)}(x,x_1,\alpha)|\, d x_1
 =\int_{\Sph}\abs{\{x_1:(x,x_1,\alpha,\omega)\in G_{\varepsilon,\kappa},x'\in A\}}\, d\omega
 \leq|\Sph|j_{\varepsilon,\kappa}^{-1}\abs A.
\]
By Chebyshev inequality, the set \(D_A^{(1)}(x,\alpha)=\{x_1:|O_A^{(1)}(x,x_1,\alpha)|>\sigma\}\) satisfies
\[
 |D_A^{(1)}(x,\alpha)|\leq j_{\varepsilon,\kappa}^{-1}\sigma^{-1}|\Sph|\abs A.
\]
It follows that
\begin{align}\label{eq:angular-estimate}
 &\int_X f_N(t,x_1)\int_{\Sph}
 b\left(\frac{u}{|u|}\cdot\omega\right)
 \one_{G_{\varepsilon,\kappa}}\one_A(x')\, d\omega\, d x_1\notag\\
 &\qquad \leq \mathcal B(\sigma)\int_{(D^{(1)}_A)^c} f_N(x_1) dx_1+\|b\|_{L^1} \int_{D^{(1)}_A}f_N(x_1) dx_1\notag\\
 &\qquad \leq M_0(f_0)\mathcal B(\sigma)+\norm b_{L^1}U_N\left(
 \frac{|\Sph|\abs A}{j_{\varepsilon,\kappa}\sigma},t\right).
\end{align}
The same calculation using \eqref{eq:area2} holds with \(x_1'\) in place of
\(x'\) for the second output by defining corresponding $O_A^{(2)}$ and $D_A^{(2)}$.

Part 3: good-gain estimate. Using \(E^\gamma\leq (m|v|^2)^\gamma+(m_1|v_1|^2)^\gamma\) and $B_N\leq E^\gamma b(\xi)$, we have
\begin{align}
    &\dfrac{1}{2}\int_{G_{\varepsilon,\kappa}} aB_N[\one_A(x')+\one_A(x_1')]f_N(t,x)f_N(t,x_1) d\omega d\alpha dx_1dx\leq \dfrac{1}{2}\|a\|_\infty(T_{0,0}+T_{0,1}+T_{1,0}+T_{1,1}),\notag\\
    &\qquad T_{0,0}=\int_{G_{\varepsilon,\kappa}}b(\xi)(m|v|^2)^\gamma\one_A(x')f_N(t,x)f_N(t,x_1)d\omega d\alpha dx_1 dx,\notag\\
    &\qquad T_{0,1}=\int_{G_{\varepsilon,\kappa}}b(\xi)(m|v|^2)^\gamma\one_A(x_1')f_N(t,x)f_N(t,x_1)d\omega d\alpha dx_1 dx,\notag\\
    &\qquad T_{1,0}=\int_{G_{\varepsilon,\kappa}}b(\xi)(m_1|v_1|^2)^\gamma\one_A(x')f_N(t,x)f_N(t,x_1)d\omega d\alpha dx_1 dx,\notag\\
    &\qquad T_{1,1}=\int_{G_{\varepsilon,\kappa}}b(\xi)(m_1|v_1|^2)^\gamma\one_A(x_1')f_N(t,x)f_N(t,x_1)d\omega d\alpha dx_1 dx.
\end{align}

For $T_{0,0}$, fixing $x,\alpha$, \eqref{eq:angular-estimate} gives
\begin{align}
    T_{0,0} &\leq \int_0^1\int_X (m|v|^2)^\gamma f_N(t,x)[M_0(f_0)\mathcal B(\sigma)+\norm b_{L^1}U_N(
 j_{\varepsilon,\kappa}^{-1}\sigma^{-1}|\Sph|\abs A,t)] dxd\alpha\notag\\&\leq M_2(f_0)^\gamma M_0(f_0)^{1-\gamma} [M_0(f_0)\mathcal B(\sigma)+\norm b_{L^1}U_N(
 j_{\varepsilon,\kappa}^{-1}\sigma^{-1}|\Sph|\abs A,t)] .
\end{align}
The same estimate holds for $T_{0,1}.$ For $T_{1,0}$ and $T_{1,1}$, consider the changing of variables $(x,x_1,\alpha,\omega)\mapsto (x_1,x,1-\alpha,-\omega)$, which maps $\theta\mapsto 1-\theta,E\mapsto E,G_{\varepsilon,\kappa}\mapsto G_{\varepsilon,\kappa},a(m,m_1,\alpha)\mapsto a(m_1,m,1-\alpha)=a(m,m_1,\alpha)$ by (K3), we have $T_{1,0}=T_{0,1}$ and $T_{1,1}=T_{0,0}$. Hence
\begin{align*}
   & \dfrac{1}{2}\int_{G_{\varepsilon,\kappa}} aB_N[\one_A(x')+\one_A(x_1')]f_N(t,x)f_N(t,x_1) d\omega d\alpha dx_1dx\\ &\qquad \leq2\|a\|_{\infty} M_0(f_0)^{2-\gamma}M_2(f_0)^\gamma \mathcal B(\sigma)+2\|a\|_{\infty}\|b\|_{L^1} M_0(f_0)^{1-\gamma}M_2(f_0)^\gamma U_N(
 j_{\varepsilon,\kappa}^{-1}\sigma^{-1}|\Sph|\abs A,t).
\end{align*}

Part 4: conclusion. divide the total gain $G_N(A,t)$ into a bad part and a good part, \eqref{eq:gainestimate} then holds with
\[
 C_1=C\norm a_\infty M_0(f_0)^{2-\gamma}M_2(f_0)^\gamma,
 \quad
 C_2=C\norm a_\infty\norm b_{L^1}
 M_0(f_0)^{1-\gamma}M_2(f_0)^\gamma.
\]
\end{proof}

\begin{proposition}[Uniform integrability]\label{prop:UI}
For every \(T>0\),
\begin{equation}\label{eq:UI}
 \lim_{q\downarrow0}\sup_N\sup_{0\leq t\leq T}U_N(q,t)=0.
\end{equation}
\end{proposition}

\begin{proof}
Let \(I=[t_0,t_0+\tau]\subset[0,T]\).  Test the bounded-kernel equation by
\(\one_A\) gives, for $t\in I$, \eqref{eq:gainestimate} gives
\begin{align}
    &\int_A f_N(t,x) dx-\int_A f_N(t_0,x) dx =\int_{t_0}^t\mathcal Q_N(f_N(s),f_N(s))[\one_{A}]ds\notag\\ & =\dfrac{1}{2}\int_{t_0}^t \int_{X^2}\int_0^1\int_{\Sph} a B_N[\one_A(x')+\one_A(x_1')-\one_A(x)-\one_A(x_1)]f_N(s,x)f_N(s,x_1) d\omega d\alpha dx_1 dx ds\notag\\&\leq\int_{t_0}^t G_N(A,s) ds\leq (t-t_0)[R_{\varepsilon,\kappa,L}+C_1\mathcal B(\sigma)]+C_2\int_{t_0}^t U_N(
 j_{\varepsilon,\kappa}^{-1}\sigma^{-1}|\Sph|\abs A,s) ds.
\end{align}
Hence
\begin{align}
 \sup_N\sup_{t\in I}U_N(q,t)
 \leq\sup_NU_N(q,t_0)+\tau R_{\varepsilon,\kappa,L}
 +C_1\tau\mathcal B(\sigma)+C_2\tau\sup_N\sup_{t\in I}
 U_N\left(\frac{|\Sph|q}
 {j_{\varepsilon,\kappa}\sigma},t\right).
 \label{eq:UIbootstrap}
\end{align}
Choose \(\tau>0\), independently of \(t_0\), so that \(C_2\tau\leq1/2\).
Suppose
\(\lim_{q\downarrow0}\sup_NU_N(q,t_0)=0\), and let \(q\downarrow0\) in
\eqref{eq:UIbootstrap}, keeping \(\varepsilon,\kappa,L,\sigma\) fixed.  If
\(
 \Lambda_I=\limsup_{q\downarrow0}\sup_N\sup_{t\in I}U_N(q,t),
\)
then
\begin{equation}\label{eq:LambdaUI}
 \Lambda_I\leq2\tau R_{\varepsilon,\kappa,L}
 +2C_1\tau\mathcal B(\sigma).
\end{equation}

We now choose the auxiliary parameters in an order compatible with the term
\(F_N(\kappa L,t)\).  First choose \(L\) so large that \(2M_1(f_0)/L\) is as small
as desired.  With this \(L\) fixed, choose \(\kappa\) so small that
\(\kappa L\downarrow0\); \cref{prop:smallmass} then makes the first part of
\(z_{\kappa,L}\) small.  Next choose \(\varepsilon\) small.  Formula
\eqref{eq:Rbad} shows that \(R_{\varepsilon,\kappa,L}\) tends to zero under these
choices.  Finally choose \(\sigma\) small and apply \cref{lem:Bmod}.  The
right-hand side of \eqref{eq:LambdaUI} can therefore be made arbitrarily small,
so \(\Lambda_I=0\).

At \(t_0=0\), the starting property follows from the absolute continuity of
the integral of \(f_0\in L^1(X)\).  Repeating the conclusion on the finitely
many intervals of length \(\tau\) covering \([0,T]\) proves
\eqref{eq:UI}.
\end{proof}

\subsection{Tightness in mass and velocity}

\begin{proposition}[Number-density tightness]\label{prop:tight}
Let $f_N$ be the truncated solutions of the bounded-kernel equations in \cref{sec:bounded}. For every \(T>0\), 
\begin{equation}\label{eq:tight}
 \lim_{r\downarrow0,\ R\uparrow\infty}
 \sup_N\sup_{0\leq t\leq T}
 \int_{\{m<r\}\cup\{m>R\}\cup\{|v|>R\}}f_N(t,x)\, d x=0,
\end{equation}
where the limit means that the three tails are removed successively.
\end{proposition}

\begin{proof}
The small-mass tail is \cref{prop:smallmass}.  Conservation of \(M_1\) gives
\[
 \int_{\{m>R\}}f_N(t,x)\, d x\leq\frac{M_1(f_0)}R.
\]
For arbitrary \(r>0\), split the large-velocity set according to the mass:
\begin{align}
 \int_{|v|>R}f_N(t,x)\, d x
 &\leq F_N(r,t)
 +\int_{\substack{m\geq r\\|v|>R}}f_N(t,x)\, d x\notag\\
 &\leq F_N(r,t)+\frac{M_2(f_0)}{rR^2}.
 \label{eq:velocitytail}
\end{align}
Choose \(r\) using \cref{prop:smallmass}, and then choose \(R\) large.
This proves \eqref{eq:tight}.
\end{proof}

\section{Compactness, identification of the limit and global solution}\label{sec:compact}

\subsection{Compactness on a finite time interval}

\begin{theorem}[Weak compactness and time compactness]\label{thm:compact}
Fix \(T>0\). Let $f_N$ be the truncated solutions of the bounded-kernel equations in \cref{sec:bounded}. There exist a subsequence, still denoted by \(f_N\), and a
nonnegative function
\(
 f\in L^\infty(0,T;L^1(X;(1+m+m|v|^2)\, d x))
\)
such that
\begin{equation}\label{eq:Cweak}
 f_N\longrightarrow f
 \quad\text{in }C\bigl([0,T];L^1(X)\text{--weak}\bigr).
\end{equation}
Thus, for every \(\zeta\in L^\infty(X)\),
\begin{equation}\label{eq:Cweak-explicit}
 \lim_{N\to\infty}\sup_{0\leq t\leq T}
 \abs{\int_X\zeta(x)\bigl(f_N(t,x)-f(t,x)\bigr)\, d x}=0.
\end{equation}
The representative \(t\mapsto f(t)\) is weakly continuous,
\(f(0)=f_0\) in \(L^1(X)\), and
\begin{equation}\label{eq:limitmomentineq}
 M_j(f(t))\leq M_j(f_0),\qquad j=0,1,2,\quad 0\leq t\leq T.
\end{equation}
\end{theorem}
\begin{proof}

Step 1: relative weak compactness of the full range. Set
\[\mathcal F_T= \{f_N(t):N\geq1,\ 0\leq t\leq T\}\subset L^1(X).\]
Since \(f_N(t)\geq0\), conservation of \(M_0\) gives
\begin{equation}\label{eq:compact-L1-bound}
 \sup_{g\in\mathcal F_T}\|g\|_{L^1(X)}
 =
 M_0(f_0).
\end{equation}
For every \(\varepsilon>0\), Proposition~\ref{prop:UI} gives
\(q_\varepsilon>0\) such that, for every measurable \(A\subset X\),
\begin{equation}\label{eq:compact-small-sets}
 |A|\leq q_\varepsilon
 \quad\Longrightarrow\quad
 \sup_{g\in\mathcal F_T}\int_A|g(x)| dx
 \leq\varepsilon.
\end{equation}
Proposition~\ref{prop:tight} gives
\(0<r_\varepsilon<1<R_\varepsilon\) such that, with
\(K_\varepsilon=[r_\varepsilon,R_\varepsilon]\times\overline{B_{R_\varepsilon}(0)}\Subset X,
\)
one has
\begin{equation}\label{eq:compact-tail}
 \sup_{g\in\mathcal F_T}
 \int_{X\setminus K_\varepsilon}|g(x)| dx
 \leq\varepsilon.
\end{equation}
The set \(K_\varepsilon\) has finite Lebesgue measure.  Hence
\eqref{eq:compact-L1-bound}, \eqref{eq:compact-small-sets}, and
\eqref{eq:compact-tail} give uniform integrability on sets of finite
measure together with uniform tightness at infinity.  The
Dunford--Pettis theorem \cite{DiestelUhl1977} therefore implies that
\(\mathcal F_T\) is relatively weakly compact in \(L^1(X)\).

Step 2: compactness of a countable family of scalar trajectories. Since \(X\) is locally compact and second countable, \(C_0(X)\) is
separable in the uniform norm.  Choose
\(\{\psi_j\}_{j\geq1}\subset C_c^1(X)\) such that $\operatorname{span}\{\psi_j:j\geq1\}$ is dense in $C_0(X)$ under $\|\cdot\|_\infty$.
For \(N,j\geq1\), define
\[
 h_{N,j}(t)
 =
 \int_X\psi_j(x)f_N(t,x) dx,
 \qquad 0\leq t\leq T.
\]
By Theorem~\ref{thm:bounded},
\[
 h_{N,j}'(t)
 =
 \mathcal Q_N(f_N(t),f_N(t))[\psi_j].
\]
Since \(|\Delta\psi_j|\leq4\|\psi_j\|_\infty\), the global collision-rate
estimate gives
\[
 \begin{aligned}
 |h_{N,j}'(t)|
 &\leq
 \frac12
 \int_{X^2}\int_0^1\int_{\Sph}
 aB_N|\Delta\psi_j|f_N(t,x)f_N(t,x_1)
  d\omega d\alpha dx_1dx\\
 &\leq
 4\|\psi_j\|_\infty
 \|a\|_\infty\|b\|_{L^1}
 M_0(f_0)^{2-\gamma}M_2(f_0)^\gamma.
 \end{aligned}
\]
Consequently, for \(s,t\in[0,T]\),
\begin{equation}\label{eq:equiLip}
 |h_{N,j}(t)-h_{N,j}(s)|
 \leq
 4\|\psi_j\|_\infty
 \|a\|_\infty\|b\|_{L^1}
 M_0(f_0)^{2-\gamma}M_2(f_0)^\gamma
 |t-s|.
\end{equation}
Moreover,
\begin{equation}\label{eq:compact-scalar-bound}
 |h_{N,j}(t)|
 \leq
 \|\psi_j\|_\infty M_0(f_0).
\end{equation}
For each fixed \(j\), the family
\(\{h_{N,j}\}_{N\geq1}\) is therefore uniformly bounded and
equicontinuous in \(C([0,T])\).  Applying the Arzel\`a--Ascoli theorem
successively for \(j=1,2,\ldots\), and then taking a diagonal
subsequence, we obtain a single subsequence, still denoted by
\((f_N)\), and functions \(h_j\in C([0,T])\) such that
\begin{equation}\label{eq:compact-diagonal}
 \sup_{0\leq t\leq T}
 |h_{N,j}(t)-h_j(t)|
 \longrightarrow0
 \qquad\text{for every }j\geq1.
\end{equation}

Step 3: definition of \(f(t)\) and weak convergence at every fixed
time. We divide the argument into three parts.  First, at each fixed time
\(t\in[0,T]\), the relative weak compactness obtained in Step~1 yields
at least one weak cluster point of the sequence \((f_N(t))_N\).
Second, the identities obtained from the diagonal extraction
\eqref{eq:compact-diagonal} determine every such cluster point uniquely.
This allows us to define \(f(t)\).  Third, the uniqueness of the weak
cluster point implies that the entire extracted sequence converges
weakly to \(f(t)\).  We then verify that \(f(t)\) is nonnegative.

Step 3a: existence of a weak cluster point at each fixed time. Fix \(t\in[0,T]\).  By the relative weak compactness established in
Step~1, every subsequence of \((f_N(t))_N\) has a further subsequence
converging weakly in \(L^1(X)\).  Let
\[
 f_{N_k}(t)\rightharpoonup g(t)
 \qquad\text{weakly in }L^1(X).
\]
For every \(j\geq1\), \eqref{eq:compact-diagonal} gives
\[
 \int_X\psi_j(x)g(x,t) dx
 =
 \lim_{k\to\infty}
 \int_X\psi_j(x)f_{N_k}(t,x) dx
 =
 h_j(t).
\]

Step 3b: uniqueness of the weak cluster point. These identities determine \(g\) uniquely.  Indeed, suppose that
\(g_1,g_2\in L^1(X)\) satisfy
\[
 \int_X\psi_jg_1 dx
 =
 \int_X\psi_jg_2 dx
 \qquad\text{for every }j\geq1.
\]
If \(\phi\in C_0(X)\), choose
\(\phi_n\in\operatorname{span}\{\psi_j:j\geq1\}\) such that
\(\|\phi_n-\phi\|_\infty\to0\).  Then
\[
 \begin{aligned}
 \left|\int_X\phi(x)(g_1-g_2)(x) dx\right|
 &\leq
 \left|\int_X\phi_n(x)(g_1-g_2)(x) dx\right|+
 \|\phi-\phi_n\|_\infty
 \bigl(\|g_1\|_{L^1}+\|g_2\|_{L^1}\bigr)\\
 &=
 \|\phi-\phi_n\|_\infty
 \bigl(\|g_1\|_{L^1}+\|g_2\|_{L^1}\bigr)
 \to0.
 \end{aligned}
\]
Thus the finite signed Radon measures \(g_1(x) dx\) and
\(g_2(x) dx\) have the same pairing with every \(\phi\in C_0(X)\).
Hence \(g_1=g_2\) almost everywhere.

Denote the unique possible weak cluster point at time \(t\) by \(f(t)\).

Step 3c: convergence of the entire sequence and nonnegativity. We now prove that the entire extracted sequence converges weakly to
\(f(t)\).  Let \(\zeta\in L^\infty(X)\).  If $f_N(t)\not\rightharpoonup f(t)$, then there exist \(\delta>0\) and a subsequence \((N_k)\) such that
\[
 \left|
 \int_X\zeta(x)\bigl(f_{N_k}(t,x)-f(t,x)\bigr) dx
 \right|
 \geq\delta
 \qquad\text{for every }k.
\]
By relative weak compactness, \((f_{N_k}(t))_k\) has a further
subsequence converging weakly to some \(g\in L^1(X)\).  The uniqueness
of weak cluster points proved above gives \(g=f(t)\), contradicting
the last inequality.  Therefore
\begin{equation}\label{eq:compact-pointwise-weak}
 f_N(t)\rightharpoonup f(t)
 \qquad\text{weakly in }L^1(X)
 \quad\text{for every }t\in[0,T].
\end{equation}

The limit is nonnegative.  Indeed, let
\(
 A_t=\{x\in X:f(t,x)<0\}.
\)
Using \(\mathbf1_{A_t}\in L^\infty(X)\) in
\eqref{eq:compact-pointwise-weak}, we obtain
\[
 \int_{A_t}f(t,x) dx
 =
 \lim_{N\to\infty}\int_{A_t}f_N(t,x) dx
 \geq0.
\]
If \(|A_t|>0\), then \(f(t,x)<0\) almost everywhere on \(A_t\), which
would imply \(\int_{A_t}f(t,x) dx<0\).  Hence \(|A_t|=0\), and
\(f(t)\geq0\) almost everywhere.

Step 4: weak continuity and convergence uniformly in time. Let \(\zeta\in L^\infty(X)\).  The bounded-measurable-test formulation
of Theorem~\ref{thm:bounded} gives, for \(0\leq s\leq t\leq T\),
\[
 \int_X\zeta(x)\bigl(f_N(t,x)-f_N(s,x)\bigr) dx
 =
 \int_s^t
 \mathcal Q_N(f_N(r),f_N(r))[\zeta] dr.
\]
Since \(|\Delta\zeta|\leq4\|\zeta\|_\infty\), the global collision-rate
estimate yields
\begin{equation}\label{eq:compact-all-test-Lip-N}
 \left|
 \int_X\zeta(x)\bigl(f_N(t,x)-f_N(s,x)\bigr) dx
 \right|
 \leq L_\zeta|t-s|,
\end{equation}
where
\[
 L_\zeta
 =
 4\|\zeta\|_\infty
 \|a\|_\infty\|b\|_{L^1}
 M_0(f_0)^{2-\gamma}M_2(f_0)^\gamma.
\]
Passing to the limit in \eqref{eq:compact-all-test-Lip-N} at the two
fixed times \(s,t\), using \eqref{eq:compact-pointwise-weak}, gives
\begin{equation}\label{eq:compact-all-test-Lip-limit}
 \left|
 \int_X\zeta(x)\bigl(f(t,x)-f(s,x)\bigr) dx
 \right|
 \leq L_\zeta|t-s|.
\end{equation}
Thus \(t\mapsto f(t)\) is weakly continuous on \([0,T]\).

We next prove convergence uniformly in time.  Suppose first that
\(L_\zeta>0\).  Given \(\varepsilon>0\), choose a finite partition
\(
 0=t_0<t_1<\cdots<t_K=T
\)
whose mesh is smaller than \(\varepsilon/(3L_\zeta)\).  By
\eqref{eq:compact-pointwise-weak}, for all sufficiently large \(N\),
\[
 \max_{0\leq k\leq K}
 \left|
 \int_X\zeta(x)\bigl(f_N(t_k,x)-f(t_k,x)\bigr) dx
 \right|
 <\frac{\varepsilon}{3}.
\]
For any \(t\in[0,T]\), choose \(t_k\) such that
\(|t-t_k|<\varepsilon/(3L_\zeta)\).  Then
\[
 \begin{aligned}
 &\left|
 \int_X\zeta(x)\bigl(f_N(t,x)-f(t,x)\bigr) dx
 \right|\leq
 \left|
 \int_X\zeta(x)\bigl(f_N(t,x)-f_N(t_k,x)\bigr) dx
 \right|\\&\qquad+
 \left|
 \int_X\zeta(x)\bigl(f_N(t_k,x)-f(t_k,x)\bigr) dx
 \right|+
 \left|
 \int_X\zeta(x)\bigl(f(t_k,x)-f(t,x)\bigr) dx
 \right|<\varepsilon.
 \end{aligned}
\]
Hence
\begin{equation}\label{eq:Cweak-explicit-proof}
 \lim_{N\to\infty}
 \sup_{0\leq t\leq T}
 \left|
 \int_X\zeta(x)\bigl(f_N(t,x)-f(t,x)\bigr) dx
 \right|
 =0
 \qquad\text{for every }\zeta\in L^\infty(X).
\end{equation}
Together with weak continuity of \(f\), this proves
\eqref{eq:Cweak} and \eqref{eq:Cweak-explicit}.

Step 5: identification of the initial value.

Since \(f_N(0)=f_0\), \eqref{eq:compact-pointwise-weak} at \(t=0\)
gives
\[
 \int_X\zeta(x)\bigl(f(0,x)-f_0(x)\bigr) dx=0
 \qquad\text{for every }\zeta\in L^\infty(X).
\]
Taking
\(
 \zeta(x)=\operatorname{sgn}\bigl(f(0,x)-f_0(x)\bigr)
\)
yields
\(
 \|f(0)-f_0\|_{L^1(X)}=0.
\)
Thus \(f(0)=f_0\) in \(L^1(X)\).

Step 6: passage of the three moment bounds to the limit. Choose a nondecreasing function
\(\rho\in C^\infty([0,\infty);[0,1])\) such that
\[
 \rho(s)=0\quad\text{for }0\leq s\leq1,
 \qquad
 \rho(s)=1\quad\text{for }s\geq2.
\]
For \(k\geq2\), define
\[
 \chi_k(m,v)
 =
 \rho(km)
 \rho\left(\frac{k}{m}\right)
 \rho\left(\frac{k}{\sqrt{1+|v|^2}}\right).
\]
Then \(0\leq\chi_k\leq1\), the sequence \((\chi_k)_k\) is pointwise
nondecreasing, and
\(
 \chi_k(x)\to 1
\) pointwise for $x\in X$.
Moreover,
\(
 \operatorname{supp}\chi_k
 \subset
 [1/k,k]\times
 \overline{B_{\sqrt{k^2-1}}(0)}
 \Subset X.
\)
Set
\[
 w_0(m,v)=1,\qquad
 w_1(m,v)=m,\qquad
 w_2(m,v)=m|v|^2.
\]
For every fixed \(k\), the function \(w_j\chi_k\) is bounded on \(X\).
Hence \eqref{eq:compact-pointwise-weak} and
\eqref{eq:boundedmoments} give, for every \(t\in[0,T]\),
\[
 \begin{aligned}
 \int_Xw_j(x)\chi_k(x)f(t,x) dx
 &=
 \lim_{N\to\infty}
 \int_Xw_j(x)\chi_k(x)f_N(t,x) dx\\
 &\leq M_j(f_0),
 \qquad j=0,1,2.
 \end{aligned}
\]
Since \(f(t)\geq0\) and \(w_j\chi_k\uparrow w_j\), the monotone
convergence theorem yields
\[
 M_j(f(t))
 =
 \lim_{k\to\infty}
 \int_Xw_j(x)\chi_k(x)f(t,x) dx
 \leq M_j(f_0),
 \qquad j=0,1,2.
\]
This proves \eqref{eq:limitmomentineq}.  Summing the three estimates,
we obtain
\begin{equation}\label{eq:compact-weight-bound}
 \sup_{0\leq t\leq T}
 \int_X(1+m+m|v|^2)f(t,x) dx
 \leq
 M_0(f_0)+M_1(f_0)+M_2(f_0).
\end{equation}

Step 7: strong measurability in the weighted space. The map \(t\mapsto f(t)\) is weakly continuous from \([0,T]\) into
\(L^1(X)\), hence weakly measurable.  Since \(L^1(X)\) is separable,
Pettis' measurability theorem \cite{DiestelUhl1977} implies that
\(
 t\mapsto f(t)
\)
is strongly measurable as an \(L^1(X)\)-valued map.

Let
\[
 W(m,v)=1+m+m|v|^2,
 \qquad
 W_\ell(m,v)=\min\{W(m,v),\ell\},
 \qquad \ell\geq1.
\]
Multiplication by \(W_\ell\) is a bounded linear map from \(L^1(X)\)
to itself.  Therefore
\(
 t\mapsto W_\ell f(t)
\)
is strongly measurable in \(L^1(X)\).  For every fixed \(t\),
\eqref{eq:compact-weight-bound} and monotone convergence give
\[
 \begin{aligned}
 \|W_\ell f(t)-Wf(t)\|_{L^1(X)}=
 \int_X\bigl(W(x)-W_\ell(x)\bigr)f(t,x) dx\to 0
 \qquad\text{as }\ell\to\infty.
 \end{aligned}
\]
Thus \(t\mapsto Wf(t)\) is a pointwise \(L^1(X)\)-limit of strongly
measurable maps, and is therefore strongly measurable in \(L^1(X)\).

The map
\[
 \mathcal J:
 L^1(X;W(x) dx)\longrightarrow L^1(X),
 \qquad
 \mathcal Jg=Wg,
\]
is an isometric isomorphism.  Hence \(t\mapsto f(t)\) is strongly
measurable as an \(L^1(X;W dx)\)-valued map.  Finally,
\eqref{eq:compact-weight-bound} gives
\[
 \operatorname*{ess\,sup}_{0\leq t\leq T}
 \|f(t)\|_{L^1(X;W dx)}
 \leq
 M_0(f_0)+M_1(f_0)+M_2(f_0).
\]
Therefore
\(
 f\in
 L^\infty\bigl(
 0,T;L^1(X;(1+m+m|v|^2) dx)
 \bigr),
\)
which completes the proof.
\end{proof}

\subsection{Identification of the hard-potential limit}\label{sec:ident}

\begin{lemma}[Collision-tail estimates]\label{lem:collisiontails}
Let \(g\geq0\) satisfy
\(M_0(g)\leq M_0(f_0)\) and \(M_2(g)\leq M_2(f_0)\).  For every
measurable \(A\subset X\),
\begin{align}
 &\int_{A\times X}\int_{\Sph}E^\gamma
 b\left(\frac{v-v_1}{|v-v_1|}\cdot\omega\right)
 g(x)g(x_1)\, d\omega\, d x_1\, d x\notag\\
 &\quad\leq\norm b_{L^1}M_2(f_0)^\gamma
 \left[M_0(f_0)\left(\int_Ag\, d x\right)^{1-\gamma}
 +M_0(f_0)^{1-\gamma}\int_Ag\, d x\right].
 \label{eq:collisiontailA}
\end{align}
Moreover, for \(L>0\),
\begin{align}
 &\int_{X^2}\int_{\Sph}E^\gamma\one_{\{E>L\}}
 b\left(\frac{v-v_1}{|v-v_1|}\cdot\omega\right)
 g(x)g(x_1)\, d\omega\, d x_1\, d x\notag\\
 &\quad\leq2\norm b_{L^1}L^{\gamma-1}M_0(f_0)M_2(f_0).
 \label{eq:Etail}
\end{align}
Both estimates apply uniformly to \(g=f_N(t)\) and to \(g=f(t)\).
\end{lemma}

\begin{proof}
Using \eqref{eq:gradcutoff}, Tonelli's theorem shows that the left-hand
side of \eqref{eq:collisiontailA} is at most
\begin{align*}
 \norm b_{L^1}\bigg[
 &M_0(g)\int_A(m|v|^2)^\gamma g(x)\, d x
 +\left(\int_Ag(x)\, d x\right)
 \int_X(m_1|v_1|^2)^\gamma g(x_1)\, d x_1\bigg].
\end{align*}
H\"older's inequality gives
\begin{align*}
 \int_A(m|v|^2)^\gamma g\, d x
 &\leq M_2(f_0)^\gamma\left(\int_Ag\, d x\right)^{1-\gamma},\\
 \int_X(m|v|^2)^\gamma g\, d x
 &\leq M_2(f_0)^\gamma M_0(f_0)^{1-\gamma},
\end{align*}
which proves \eqref{eq:collisiontailA}.  On \(\{E>L\}\),
\(E^\gamma\leq L^{\gamma-1}E\).  Hence \eqref{eq:E-energy} gives
\begin{align*}
 &\int_{X^2}\int_{\Sph}E^\gamma\one_{\{E>L\}}
 b\left(\frac{v-v_1}{|v-v_1|}\cdot\omega\right)
 g(x)g(x_1)\, d\omega\, d x_1\, d x\\
 &\quad\leq L^{\gamma-1}\norm b_{L^1}
 \int_{X^2}(m|v|^2+m_1|v_1|^2)
 g(x)g(x_1)\, d x_1\, d x\\
 &\hspace{35mm}=2L^{\gamma-1}\norm b_{L^1}M_0(g)M_2(g),
\end{align*}
which proves \eqref{eq:Etail}.  The last assertion follows from
\eqref{eq:boundedmoments} and \eqref{eq:limitmomentineq}.
\end{proof}

\begin{proposition}[Bilinear limit identification]\label{prop:identify}
Let $f_N$ be the truncated solutions in \cref{sec:bounded} and $f$ be the subsequential limit of $f_N$ in \cref{thm:compact}. Along the subsequence in \cref{thm:compact}, for every
\(\psi\in C_c^1(X)\),
\begin{equation}\label{eq:Qlimit}
 \lim_{N\to\infty}\sup_{t\in [0,T]}|
 \mathcal Q_N(f_N(t),f_N(t))[\psi]
 -\mathcal Q(f(t),f(t))[\psi]|=0.
\end{equation}
\end{proposition}

\begin{proof}
Throughout the proof,
\(
 \abs{\Delta\psi}\leq4\norm\psi_\infty.
\)
For every nonnegative \(g\) satisfying the two moment bounds in
\cref{lem:collisiontails}, H\"older's inequality give
\begin{equation}\label{eq:identify-Egamma-rate}
 \int_{X^2}E^\gamma g(x)g(x_1)\, d x_1\, d x
 \leq2M_0(f_0)^{2-\gamma}M_2(f_0)^\gamma.
\end{equation}

The proof is divided into seven steps. We first remove the truncation $B_N$
to the collision kernel $B$.  We then localize the incoming variables $(x,x_1)$ in a compact set $K_q\times K_q$,
exclude the endpoints of the mass-exchange parameter $(0,\rho)\cup (1-\rho,1)$, and approximate
the angular kernel $b$ by continuous kernels $b_j$.  On the resulting compact
set $K_q\times K_q$, we prove convergence of the regularized bilinear forms and finally
remove all auxiliary parameters.

Step 1: removal of the collision-kernel truncation. Define
\begin{equation}\label{eq:identify-betaH}
 \beta_H=\sup_{e\in\Sph}\int_{\Sph}
 b(e\cdot\omega)\one_{\{b(e\cdot\omega)>H\}}\, d\omega.
\end{equation}
Rotation invariance of surface measure shows that the integral is
independent of \(e\); hence the absolute continuity of the integral in
\eqref{eq:gradcutoff} gives
\(\beta_H\to0\) as \(H\to\infty\).
If \(N\geq L^\gamma H\), then \(B_N=B\) on
\(\{E\leq L\}\cap\{b\leq H\}\). Hence
\[
0\leq E^\gamma b-B_N\leq E^\gamma b[\one_{\{E>L\}}+\one_{\{b> H\}}].
\]
Consequently,
\begin{align*}
 &\abs{\mathcal Q_N(g,g)[\psi]-\mathcal Q(g,g)[\psi]}\\
 &\quad\leq2\norm\psi_\infty\norm a_\infty
 \int_{X^2}\int_{\Sph}E^\gamma
 b\left(\frac u{|u|}\cdot\omega\right)
 \left[
 \one_{\{E>L\}}+
 \one_{\{b(u/|u|\cdot\omega)>H\}}
 \right]
 g(x)g(x_1)\, d\omega\, d x_1\, d x.
\end{align*}
The first indicator is controlled by \eqref{eq:Etail}.  For the second,
\eqref{eq:identify-Egamma-rate} and \eqref{eq:identify-betaH} give
\begin{align*}
 &\int_{X^2}\int_{\Sph}E^\gamma
 b\left(\frac u{|u|}\cdot\omega\right)
 \one_{\{b(u/|u|\cdot\omega)>H\}}
 g(x)g(x_1)\, d\omega\, d x_1\, d x\\
 &\quad\leq\beta_H\int_{X^2}E^\gamma
 g(x)g(x_1)\, d x_1\, d x
 \leq2\beta_HM_0(f_0)^{2-\gamma}M_2(f_0)^\gamma.
\end{align*}
Thus, uniformly for every admissible \(g\),
\begin{align}
 &\abs{\mathcal Q_N(g,g)[\psi]-\mathcal Q(g,g)[\psi]}\notag\\
 &\quad\leq4\norm\psi_\infty\norm a_\infty
 \left[L^{\gamma-1}\norm b_{L^1}M_0(f_0)M_2(f_0)
 +\beta_HM_0(f_0)^{2-\gamma}M_2(f_0)^\gamma\right].
 \label{eq:identify-BN-error}
\end{align}

Step 2: localization of the incoming variables $(x,x_1)$.  Given \(q>0\), choose by
\cref{prop:tight}
\(
 K_q=[r_q,R_q]\times\overline{B_{R_q}(0)}\Subset X
\)
so that
\begin{equation}\label{eq:identify-Kq-fN}
 \sup_{N,t}\int_{K_q^c}f_N(t,x)\, d x\leq q.
\end{equation}
Testing \eqref{eq:Cweak-explicit} with \(\one_{K_q^c}\) gives the same
bound for \(f\):
\begin{equation}\label{eq:identify-Kq-f}
 \sup_t\int_{K_q^c}f(t,x)\, d x\leq q.
\end{equation}
Let \(\mathcal Q^{K_q}\) denote \(\mathcal Q\) with
\((x,x_1)\in K_q\times K_q\).  Since the complement of this product is
contained in \((K_q^c\times X)\cup(X\times K_q^c)\), \eqref{eq:collisiontailA} yield, for
\(g=f_N(t)\) or \(g=f(t)\),
\begin{align}
 &\abs{\mathcal Q(g,g)[\psi]-\mathcal Q^{K_q}(g,g)[\psi]}\notag\\
 &\quad\leq4\norm\psi_\infty\norm a_\infty\norm b_{L^1}M_2(f_0)^\gamma
 \left[M_0(f_0)q^{1-\gamma}+M_0(f_0)^{1-\gamma}q\right].
 \label{eq:identify-K-error}
\end{align}

Step 3: exclusion of the endpoints \((0,\rho)\cup(1-\rho,1)\) in the mass-exchange variable. For \(0<\rho<1/2\), let \(\mathcal Q^{K_q,\rho}\) be the same form with
\(\alpha\in[\rho,1-\rho]\).  The omitted set has length \(2\rho\), so
\eqref{eq:identify-Egamma-rate} give
\begin{align}
 &\abs{\mathcal Q^{K_q}(g,g)[\psi]
 -\mathcal Q^{K_q,\rho}(g,g)[\psi]}\notag\\
 &\quad\leq8\rho\norm\psi_\infty\norm a_\infty\norm b_{L^1}
 M_0(f_0)^{2-\gamma}M_2(f_0)^\gamma.
 \label{eq:identify-alpha-error}
\end{align}

Step 4: approximation of the angular kernel $b$. For \(d\geq2\), the zonal integration formula gives
\[
\omega_{d-2}=|\mathbb S^{d-2}|,\qquad 
  d\nu_d(s)=\omega_{d-2}(1-s^2)^{(d-3)/2}\, d s,
 \qquad
 \int_{\Sph}h(e\cdot\omega)\, d\omega
 =\int_{-1}^1h(s)\, d\nu_d(s).
\]
Since \(\nu_d\) is finite, choose \(c_j\in C([-1,1])\) with
\(c_j\to b\) in \(L^1(\nu_d)\), and set \(b_j=(c_j)_+\).  The inequality
\(\abs{(c_j)_+-b}\leq\abs{c_j-b}\) gives
\begin{equation}\label{eq:identify-deltaj}
 \delta_j:=\sup_{e\in\Sph}\int_{\Sph}
 \abs{b_j(e\cdot\omega)-b(e\cdot\omega)}\, d\omega
 \longrightarrow0.
\end{equation}
For \(d=1\), one may take \(b_j(\pm1)=b(\pm1)\), so that
\(\delta_j=0\).  Denote by \(\mathcal Q^{K_q,\rho,j}\) the form obtained
from \(\mathcal Q^{K_q,\rho}\) by replacing \(b\) with \(b_j\).  Then
\begin{align}
 \abs{\mathcal Q^{K_q,\rho,j}(g,g)[\psi]
 -\mathcal Q^{K_q,\rho}(g,g)[\psi]}\leq4\norm\psi_\infty\norm a_\infty\delta_j
 M_0(f_0)^{2-\gamma}M_2(f_0)^\gamma.
 \label{eq:identify-bj-error}
\end{align}

Step 5: continuity of the localized regularized kernel $\mathscr K_{q,\rho,j}$. For fixed \(q,\rho,j\), define on \(K_q\times K_q\)
\begin{align}
 \mathscr K_{q,\rho,j}(x,x_1)=\int_\rho^{1-\rho}\int_{\Sph}
 &a(m,m_1,\alpha)E^\gamma
 b_j\left(\frac{v-v_1}{|v-v_1|}\cdot\omega\right)
 \Delta\psi\, d\omega\, d\alpha.
 \label{eq:identify-regular-kernel}
\end{align}
This kernel is continuous.  Indeed, away from \(v=v_1\), continuity
follows from dominated convergence, because all variables range over a
compact set and \(\alpha(1-\alpha)\geq\rho(1-\rho)\).  At \(v=v_1\),
\(
 E^\gamma\leq R_q^\gamma|v-v_1|^{2\gamma},
\) since $m,m_1\in [r_q,R_q].$
and the remaining factors are bounded uniformly by
\(4\norm\psi_\infty\norm a_\infty\norm{b_j}_\infty\); hence the
integrand converges uniformly to zero.  Thus
\(
 \mathscr K_{q,\rho,j}\in C(K_q\times K_q).
\)

Step 6: convergence of the localized regularized forms. Notice that the algebra of finite sums of products \(\sum_{\ell=1}^r\phi_\ell(x)\chi_\ell(x_1)\),
with \(\phi_\ell,\chi_\ell\in C(K_q)\), are uniformly dense in
\(C(K_q\times K_q)\) by the Stone--Weierstrass theorem.  Fix an
approximation $P_\varepsilon(x,x_1)=\sum_{\ell=1}^r \phi_\ell(x)\chi_\ell(x_1)$, satisfying $\|\mathscr K_{q,\rho,j}-P_\varepsilon\|_{C(K_q^2)}\leq \varepsilon$. Hence
\begin{equation}
    \left|\mathcal Q^{K_q,\rho,j}(g,g)[\psi]-\dfrac12\int_{K_q^2}P_\varepsilon(x,x_1)g(x)g(x_1)dx_1dx\right|\leq \dfrac\varepsilon2\left(\int_{K_q} g(x) dx\right)^2\leq \dfrac\varepsilon2M_0(f_0)^2.
\end{equation}
On the other hand,
\[
\int_{K_q^2}P_\varepsilon(x,x_1)g(x)g(x_1)dx_1dx=\sum_{\ell=1}^r \left(\int_{K_q} \phi_\ell (x)g(x) dx\right)\left(\int_{K_q} \chi_\ell (x_1)g(x_1) dx\right)
\]

Extending
each factor by zero outside \(K_q\) produces an \(L^\infty(X)\) test, so
\eqref{eq:Cweak-explicit} gives, uniformly in \(t\),
\[
 \int_{K_q}\phi_\ell f_N(t)\, d x\to
 \int_{K_q}\phi_\ell f(t)\, d x,
 \qquad
 \int_{K_q}\chi_\ell f_N(t)\, d x\to
 \int_{K_q}\chi_\ell f(t)\, d x.
\]
If these four integrals are denoted respectively by
\(A_{\ell,N}(t),A_\ell(t),C_{\ell,N}(t),C_\ell(t)\), then
\begin{align*}
 &\sup_t\abs{A_{\ell,N}(t)C_{\ell,N}(t)-A_\ell(t)C_\ell(t)}\\
 &\quad\leq\sup_t\abs{A_{\ell,N}-A_\ell}\sup_t\abs{C_{\ell,N}}
 +\sup_t\abs{A_\ell}\sup_t\abs{C_{\ell,N}-C_\ell}\longrightarrow0,
\end{align*}
because every factor is bounded by its uniform norm times \(M_0(f_0)\).
Summing over the finitely many indices and comparing the exact kernel
with its product approximation yields
\begin{align*}
 \limsup_{N\to\infty}\sup_{0\leq t\leq T}
 \abs{\mathcal Q^{K_q,\rho,j}(f_N(t),f_N(t))[\psi]
 -\mathcal Q^{K_q,\rho,j}(f(t),f(t))[\psi]}\leq\varepsilon M_0(f_0)^2.
\end{align*}
Letting \(\varepsilon\downarrow0\), we obtain
\begin{equation}\label{eq:identify-local-uniform-limit}
 \sup_{0\leq t\leq T}
 \abs{\mathcal Q^{K_q,\rho,j}(f_N(t),f_N(t))[\psi]
 -\mathcal Q^{K_q,\rho,j}(f(t),f(t))[\psi]}
 \longrightarrow0.
\end{equation}

Step 7: removal of the auxiliary parameters. We now combine the preceding estimates.  For fixed
\(L,H,q,\rho,j\), take first \(N\to\infty\), with
\(N\geq L^\gamma H\).  The triangle inequality and
\eqref{eq:identify-BN-error}, \eqref{eq:identify-K-error}, \eqref{eq:identify-alpha-error}, \eqref{eq:identify-bj-error}  and \eqref{eq:identify-local-uniform-limit}
give
\begin{align}
 &\limsup_{N\to\infty}\sup_{0\leq t\leq T}
 \abs{\mathcal Q_N(f_N(t),f_N(t))[\psi]
 -\mathcal Q(f(t),f(t))[\psi]}\notag\\
 &\leq4\norm\psi_\infty\norm a_\infty
 \left[L^{\gamma-1}\norm b_{L^1}M_0(f_0)M_2(f_0)
 +\beta_HM_0(f_0)^{2-\gamma}M_2(f_0)^\gamma\right]\notag\\
 &\quad+8\norm\psi_\infty\norm a_\infty\norm b_{L^1}M_2(f_0)^\gamma
 \left[M_0(f_0)q^{1-\gamma}+M_0(f_0)^{1-\gamma}q\right]\notag\\
 &\quad+16\rho\norm\psi_\infty\norm a_\infty\norm b_{L^1}
 M_0(f_0)^{2-\gamma}M_2(f_0)^\gamma\notag\\
 &\quad+8\norm\psi_\infty\norm a_\infty\delta_j
 M_0(f_0)^{2-\gamma}M_2(f_0)^\gamma.
 \label{eq:identify-final-bound}
\end{align}
Let successively \(L\to\infty\), \(H\to\infty\), \(q\downarrow0\),
\(\rho\downarrow0\), and \(j\to\infty\).  The right-hand side tends
to zero because \(0<\gamma<1\) and \eqref{eq:identify-deltaj}.  We have
therefore proved \eqref{eq:Qlimit}.
\end{proof}

\begin{corollary}[Finite-time limit equation]\label{cor:finite}
The limit supplied by \cref{thm:compact} is nonnegative, has the
weighted bounds stated there, and satisfies \eqref{eq:weaksol} for every
\(\psi\in C_c^1(X)\) and every
\(\eta\in C_c^1([0,T))\).
\end{corollary}

\begin{proof}
For every \(N\), the time-integrated bounded-kernel equation is
\eqref{eq:weaksol} with \(\mathcal Q_N\).  Its linear terms converge by
\eqref{eq:Cweak}, and its collision term converges by
\cref{prop:identify}.  Passing to the limit gives
\eqref{eq:weaksol} for every \(\psi\in C_c^1(X)\); nonnegativity and the
weighted bounds follow from \cref{thm:compact}.
\end{proof}

\subsection{Global diagonal extraction}

\begin{proposition}[Global compact-time limit]\label{prop:global-limit}
Let $f_N$ be the truncated solution of the bounded-kernel equation in \cref{sec:bounded}. There exist a strictly increasing sequence \((N_k)_{k\geq1}\) and a
nonnegative function \(f:[0,\infty)\to L^1(X)\) such that, for every
\(T>0\),
\begin{equation}\label{eq:global-Cweak}
 f_{N_k}\longrightarrow f
 \quad\text{in }C\bigl([0,T];L^1(X)\textnormal{--weak}\bigr).
\end{equation}
Equivalently, for every \(\zeta\in L^\infty(X)\),
\begin{equation}\label{eq:global-Cweak-explicit}
 \lim_{k\to\infty}\sup_{0\leq t\leq T}
 \abs{\int_X\zeta(x)[f_{N_k}(t,x)-f(t,x)]\, d x}=0.
\end{equation}
The representative \(t\mapsto f(t)\) is weakly continuous on every
compact time interval, \(f(0)=f_0\) in \(L^1(X)\), and
\begin{equation}\label{eq:global-moment-inequalities}
 M_j(f(t))\leq M_j(f_0),\qquad j=0,1,2,\qquad t\geq0.
\end{equation}
It may be chosen jointly measurable on \((0,\infty)\times X\).
Moreover,
\begin{align}
 f&\in L^\infty(0,\infty;
 L^1(X;(1+m+m|v|^2)\, dx)),
 \label{eq:global-weighted-Linfty}\\
 \operatorname*{ess\,sup}_{t>0}
 \int_X(1+m+m|v|^2)f(t,m,v)dx
 &\leq M_0(f_0)+M_1(f_0)+M_2(f_0).
 \label{eq:global-weighted-Linfty-bound}
\end{align}
For every \(T>0\), every \(\psi\in C_c^1(X)\), and every
\(\eta\in C_c^1([0,T))\),
\begin{align}
 -\int_0^T\eta'(t)\int_Xf(t,x)\psi(x)\, d x\, d t
 &=\int_0^T\eta(t)\mathcal Q(f(t),f(t))[\psi]\, d t
 +\eta(0)\int_Xf_0(x)\psi(x)\, d x.
 \label{eq:global-compact-weak-identity}
\end{align}
\end{proposition}

\begin{proof}
Apply \cref{thm:compact,cor:finite} on \([0,1]\).  There
are an infinite set \(\mathcal I_1\subset\mathbb N\) and a limit
\(f^{(1)}\) on \([0,1]\) for which the convergence in
\eqref{eq:Cweak} holds along \(N\in\mathcal I_1\).  Suppose that
\(\mathcal I_J\) has been constructed.  Applying the same results
on \([0,J+1]\) to the sequence indexed by \(\mathcal I_J\) gives an
infinite subset \(\mathcal I_{J+1}\subset\mathcal I_J\) and a limit
\(f^{(J+1)}\) on \([0,J+1]\).

Choose \(N_1\in\mathcal I_1\), and, after \(N_k\) has been chosen,
choose
\(
 N_{k+1}\in\mathcal I_{k+1}, N_{k+1}>N_k.
\)
Fix \(J\geq1\).  Since \(\mathcal I_k\subset\mathcal I_J\) for
\(k\geq J\), the tail \((N_k)_{k\geq J}\) is a subsequence of the
sequence indexed by \(\mathcal I_J\).  Hence, for every
\(\zeta\in L^\infty(X)\),
\begin{equation}\label{eq:global-stage-J-limit}
 \lim_{k\to\infty}\sup_{0\leq t\leq J}
 \abs{\int_X\zeta(x)[f_{N_k}(t,x)-f^{(J)}(t,x)]\, d x}=0.
\end{equation}
For \(J\geq1\), the same diagonal sequence converges on \([0,J]\) to
both \(f^{(J)}\) and \(f^{(J+1)}\).  Therefore
\[
 \int_X\zeta(x)[f^{(J+1)}(t,x)-f^{(J)}(t,x)]\, d x=0
\]
for every \(t\in[0,J]\) and \(\zeta\in L^\infty(X)\).  Taking
\(\zeta=\operatorname{sgn}(f^{(J+1)}(t)-f^{(J)}(t))\) gives
\begin{equation}\label{eq:global-overlap}
 f^{(J+1)}(t)=f^{(J)}(t)
 \quad\text{in }L^1(X),\qquad 0\leq t\leq J.
\end{equation}
This is uniqueness of the weak \(L^1\) limit along the fixed diagonal
sequence.  At this stage, no uniqueness theorem for solutions of the
limiting equation is used.

For \(t\geq0\), choose any integer \(J>t\) and define
\(
 f(t)=f^{(J)}(t).
\)
Equation \eqref{eq:global-overlap} makes the definition independent of
\(J\).  For every \(T>0\) and every integer \(J>T\),
\begin{equation}\label{eq:global-restriction}
 f|_{[0,T]}=f^{(J)}|_{[0,T]}.
\end{equation}
Combining \eqref{eq:global-stage-J-limit} with
\eqref{eq:global-restriction} proves
\eqref{eq:global-Cweak} and \eqref{eq:global-Cweak-explicit}.  Nonnegativity, weak
continuity, the initial value, and \eqref{eq:global-moment-inequalities}
follow from \cref{thm:compact} on every interval \([0,J]\).

The pointwise moment bounds give
\[
 \int_X(1+m+m|v|^2)f(t,m,v)\, d m\, d v
 \leq M_0(f_0)+M_1(f_0)+M_2(f_0).
\]
On each \((0,J)\), weighted strong measurability is part of
\cref{thm:compact}.  The countable exhaustion
\((0,\infty)=\bigcup_{J\geq1}(0,J)\) proves global weighted strong
measurability.  Since the weighted \(L^1\) space is separable, \(f\)
admits a jointly measurable representative.  The same bounds give
\eqref{eq:global-weighted-Linfty} and \eqref{eq:global-weighted-Linfty-bound}.
Finally, for fixed \(T\), choose \(J>T\) and use
\eqref{eq:global-restriction}; then
\eqref{eq:global-compact-weak-identity} is precisely the finite-time
identity supplied by \cref{cor:finite}.
\end{proof}

\section{Strong formulation and moment conservation}
\label{sec:strong-global}

\subsection{The \texorpdfstring{\(L^1\)}{L1}-valued global Bochner equation and bounded Borel tests}
\label{sec:bounded-measurable-tests}

The compactness argument first verifies \eqref{eq:weaksol} for
\(C_c^1(X)\) spatial tests.  We now prove the full test class required
by \cref{def:weak}: the collision term is an \(L^1(X)\)-valued function
of time, so the equation holds for every bounded Borel measurable
spatial test.

We first need to realize the strong form of the hard-potential collision operator in \(L^1.\)

\begin{lemma}[Static \(L^1\) realization of the hard collision form]
\label{lem:hard-Q-density}
Let \(g\geq0\) satisfy \(M_0(g)+M_2(g)<\infty\).  Then there exists
\(\mathbf Q(g,g)\in L^1(X)\) such that, for every bounded Borel
measurable \(\psi:X\to\mathbb R\),
\begin{align}
 \int_X\psi(x)\mathbf Q(g,g)(x)\, d x
 &=\mathcal Q(g,g)[\psi]
 \label{eq:hard-Q-density-pairing}
\end{align}
Furthermore,
\begin{align}
 \norm{\mathbf Q(g,g)}_{L^1(X)}
 \leq4\norm a_\infty\norm b_{L^1}
 M_0(g)^{2-\gamma}M_2(g)^\gamma,
 \label{eq:hard-Q-static-bound}
\end{align}
\end{lemma}

\begin{proof}

Take $\rho\in L^1(X^2\times (0,1)\times\Sph).$ Let $\Phi_\omega$ be the collision diffeomorphism introduced in \cref{lem:augmented} and take $\widetilde{\rho}(x',x_1',\theta,\omega)=\rho(\Phi_\omega^{-1}(x',x_1',\theta),\omega)s^{-d}$ for fixed $\omega$. Hence \eqref{eq:rho-area-formula} gives $\widetilde{\rho}\in L^1$, and Fubini's theorem therefore defines the 4 marginals in $L^1(X):$
\begin{align}
    G^1(z) &= \int_X\int_0^1\int_{\Sph} \widetilde{\rho}(z,x_1',\theta,\omega) d\omega d\theta dx_1',\notag\\
    G^2(z) &= \int_X\int_0^1\int_{\Sph} \widetilde{\rho}(x',z,\theta,\omega) d\omega d\theta dx',\notag\\
    L^1(z) &= \int_X\int_0^1\int_{\Sph} {\rho}(z,x_1,\alpha,\omega) d\omega d\alpha dx_1,\notag\\
    L^2(z) &= \int_X\int_0^1\int_{\Sph} {\rho}(x,z,\alpha,\omega) d\omega d\alpha dx.
\end{align}

Set
\[
\mathcal T\rho=G^1+G^2-L^1-L^2.
\]
Since each of the four marginals depends linearly on \(\rho\), $\mathcal T$ is linear. The collision change of variables
\eqref{eq:rho-area-formula} for the gain marginals and Fubini's theorem for the loss marginals show that
\begin{equation}\label{eq:Tbound}
    \|\mathcal T\rho\|_{L^1(X)}\leq \|G^1\|_{L^1(X)}+\|G^2\|_{L^1(X)}+\|L^1\|_{L^1(X)}+\|L^2\|_{L^1(X)}\leq 4\|\rho\|_{L^1(X^2\times (0,1)\times \Sph)}.
\end{equation}
Change of variables of the gain marginals also gives, for
every bounded Borel measurable \(\psi\),
\begin{equation}\label{eq:hard-marginal-pairing}
 \int_X\psi(z)\mathcal T\rho(z)\, d z
 =\int_{X^2}\int_0^1\int_{\Sph}\rho
 [\psi(x')+\psi(x_1')-\psi(x)-\psi(x_1)]
 \, d\omega\, d\alpha\, d x_1\, d x.
\end{equation}

Now set
\[
 \rho_g(x,x_1,\alpha,\omega)
 =\frac12a(m,m_1,\alpha)E^\gamma
 b\left(\frac u{|u|}\cdot\omega\right)g(x)g(x_1).
\]
By $E^\gamma\leq (m|v|^2)^\gamma+(m_1|v_1|^2)^\gamma$ and H\"older's inequality,
\begin{align}
 &\int_{X^2}\int_0^1\int_{\Sph}
 aE^\gamma b\left(\frac u{|u|}\cdot\omega\right)
 g(x)g(x_1)\, d\omega\, d\alpha\, d x_1\, d x\notag\\
 &\quad\leq2\norm a_\infty\norm b_{L^1}M_0(g)
 \int_X(m|v|^2)^\gamma g(x)\, d x\notag\\
 &\quad\leq2\norm a_\infty\norm b_{L^1}
 M_0(g)^{2-\gamma}M_2(g)^\gamma.
 \label{eq:hard-total-rate-bm}
\end{align}
Then
\[
 \norm{\rho_g}_{L^1(X^2\times(0,1)\times \Sph)}
 \leq\norm a_\infty\norm b_{L^1}
 M_0(g)^{2-\gamma}M_2(g)^\gamma.
\]
Define
\(
 \mathbf Q(g,g):=\mathcal T\rho_g.
\)
Applying \eqref{eq:hard-marginal-pairing} with $\rho=\rho_g$
 gives \eqref{eq:hard-Q-density-pairing}. Combining \eqref{eq:Tbound} with the preceding estimate of $\|\rho_g\|_{L^1}$
 gives \eqref{eq:hard-Q-static-bound}.
\end{proof}
Set
\begin{equation}\label{eq:global-CQ}
 C_Q=4\norm a_\infty\norm b_{L^1}
 M_0(f_0)^{2-\gamma}M_2(f_0)^\gamma.
\end{equation}

\begin{proposition}[Global \(L^1\)-integral formulation and bounded
Borel tests]\label{prop:bounded-measurable-tests}
Let $f(t)$ be constructed as in \cref{prop:global-limit}. For every \(t\geq0\), let $\mathbf Q(f,f)(t)$ represent the map
\(
 t\mapsto \mathbf Q(f(t),f(t)),
\)
wThen
\begin{equation}\label{eq:hard-Q-L1-bound}
  \mathbf Q(f,f)\in L^\infty(0,\infty;L^1(X)),\qquad
 \operatorname*{ess\,sup}_{t>0}\norm{\mathbf Q(f,f)(t)}_{L^1(X)}\leq C_Q,
\end{equation}
\begin{equation}\label{eq:L1-valued-equation}
 f\in W^{1,\infty}(0,\infty;L^1(X))
 \subset C([0,\infty);L^1(X)),\qquad
 \partial_tf=\mathbf Q(f,f)\quad\text{for a.e. }t>0.
\end{equation}
For every \(0\leq r\leq t<\infty\),
\begin{align}
 f(t)-f(r)&=\int_r^t\mathbf Q(f(s),f(s))\, d s
 \quad\text{in }L^1(X),
 \label{eq:two-time-L1-integrated}\\
 f(t)&=f_0+\int_0^t\mathbf Q(f(s),f(s))\, d s
 \quad\text{in }L^1(X).
 \label{eq:L1-integrated-bm}
\end{align}
All time integrals are Bochner integrals, and
\begin{equation}\label{eq:global-L1-Lipschitz}
 \norm{f(t)-f(r)}_{L^1(X)}\leq C_Q\abs{t-r},\qquad r,t\geq0.
\end{equation}
For every bounded Borel measurable \(\psi:X\to\mathbb R\) and every
\(t\geq0\),
\begin{align}
 \int_Xf(t,x)\psi(x)\, d x
 &=\int_Xf_0(x)\psi(x)\, d x+\int_0^t\mathcal Q(f(s),f(s))[\psi] ds
 \label{eq:bounded-measurable-identity}
\end{align}
Consequently, \eqref{eq:weaksol} holds on every finite interval for
every bounded Borel measurable \(\psi\); thus \(f\) is a global weak
solution in the sense of \cref{def:weak}.
\end{proposition}

\begin{proof}
The jointly measurable representative furnished by
\cref{prop:global-limit} makes
\[
 \rho(t,x,x_1,\alpha,\omega)
 =\frac12a(m,m_1,\alpha)E^\gamma
 b\left(\frac u{|u|}\cdot\omega\right)f(t,x)f(t,x_1)
\]
jointly measurable.  For every \(T>0\),
\eqref{eq:hard-total-rate-bm} and \eqref{eq:global-moment-inequalities} give
\(
 \rho\in L^\infty(0,T;L^1(X^2\times(0,1)\times\Sph)).
\)
Since \(L^1(X^2\times(0,1)\times\Sph)\) is separable, joint measurability implies
strong measurability of \(t\mapsto\rho(t)\) in \(L^1(X^2\times(0,1)\times\Sph)\).
Since the operator \(\mathcal T\) is bounded by
\eqref{eq:Tbound},
\(\mathbf Q(f,f)(t)=\mathcal T\rho(t)\) is strongly measurable.  The static bound
\eqref{eq:hard-Q-static-bound} and
\eqref{eq:global-moment-inequalities} prove
\eqref{eq:hard-Q-L1-bound}.

Fix \(T>0\) and define the Bochner integral
\begin{equation}\label{eq:F-Bochner-bm}
 F(t)=f_0+\int_0^t\mathbf Q(f,f)(s)\, d s,\qquad 0\leq t\leq T.
\end{equation}
Then
\(
 F\in W^{1,\infty}(0,T;L^1(X))
\), \(F'(t)=\mathbf Q(f,f)(t)\) for almost every \(t\), and \(F(0)=f_0\).
For \(\psi\in C_c^1(X)\) and \(\eta\in C_c^1([0,T))\), applying
\(g\mapsto\int_Xg\psi\, d x\) to \eqref{eq:F-Bochner-bm}, integrating
by parts in time, and using \eqref{eq:hard-Q-density-pairing} give
\begin{align*}
 -\int_0^T\eta'(t)\int_XF(t,x)\psi(x)\, d x\, d t
 &=\int_0^T\eta(t)\mathcal Q(f(t),f(t))[\psi]\, d t
 +\eta(0)\int_Xf_0\psi\, d x.
\end{align*}
Subtracting this identity from
\eqref{eq:global-compact-weak-identity} yields
\[
 -\int_0^T\eta'(t)\int_X[f(t,x)-F(t,x)]\psi(x)\, d x\, d t=0.
\]
Define
\[
 H_\psi(t)=\int_X[f(t)-F(t)]\psi\, d x
\]
Taking $\eta\in C_c^1([0,T))$ shows that the distribution derivative $H_\psi'=0$ in $\mathcal D'(0,T)$. By \cref{prop:global-limit}, $f$ is weakly continuous in $L^1(X)$, while $F\in C([0,T];L^1(X))$, which implies the continuity of $H_\psi(t)$ on $[0,T]$. Since $H_\psi(0)=0$,
\[
 \int_X[f(t,x)-F(t,x)]\psi(x)\, d x=0,
 \qquad 0\leq t\leq T,\quad\psi\in C_c^1(X).
\]
Uniform density of \(C_c^1(X)\) in \(C_0(X)\), followed by separation of
finite signed Radon measures by \(C_0(X)\), gives
\(
 f(t)=F(t)
\) in \(L^1(X)\).  Since \(T\) is arbitrary,
\eqref{eq:L1-integrated-bm} holds globally; subtracting it at two times
gives \eqref{eq:two-time-L1-integrated}.  Together with
\eqref{eq:global-weighted-Linfty} and \eqref{eq:hard-Q-L1-bound}, these identities
give \eqref{eq:L1-valued-equation} and
\eqref{eq:global-L1-Lipschitz}.

Let \(\psi\) be bounded and Borel measurable.  Pairing
\eqref{eq:L1-integrated-bm} with \(\psi\) and using
\eqref{eq:hard-Q-density-pairing} proves
\eqref{eq:bounded-measurable-identity}.  Integration by parts in time
then gives \eqref{eq:weaksol} for every bounded Borel test.
\end{proof}

\subsection{Conservation of moments}\label{sec:moments}

\subsubsection{Conservation of number and mass}

\begin{lemma}[Mass-weighted rate]\label{lem:massrate}
For every nonnegative \(g\) with finite \(M_0(g),M_1(g),M_2(g)\),
\begin{align}
 &\int_{X^2}\int_0^1\int_{\Sph}
 aE^\gamma b\left(\frac u{|u|}\cdot\omega\right)
 (m+m_1)g(x)g(x_1)
 \, d\omega\, d\alpha\, d x_1\, d x\notag\\
 &\quad\leq C_\gamma\norm a_\infty\norm b_{L^1}
 M_0(g)^{1-\gamma}M_1(g)M_2(g)^\gamma.
 \label{eq:massrate}
\end{align}
\end{lemma}

\begin{proof}
Let \(S=m+m_1\) and \(c_\gamma=\max\{1,2^{2\gamma-1}\}\).  Since
\((m+m_1)^{1-\gamma}\leq m^{1-\gamma}+m_1^{1-\gamma}\) and
\(|v-v_1|^{2\gamma}\leq
c_\gamma(|v|^{2\gamma}+|v_1|^{2\gamma})\),
\begin{align}
 SE^\gamma
 &\leq c_\gamma\Bigl[
 m m_1^\gamma|v|^{2\gamma}
 +m m_1^\gamma|v_1|^{2\gamma}
 +m^\gamma m_1|v|^{2\gamma}
 +m^\gamma m_1|v_1|^{2\gamma}\Bigr].
 \label{eq:massrate-fourterms}
\end{align}
Hence the left side of $\eqref{eq:massrate}$ is controlled by $c_\gamma\|a\|_\infty \|b\|_{L^1} (I_1+I_2+I_3+I_4)$, where
\begin{align}
    I_1 &= \int_{X^2} mm_1^\gamma |v|^{2\gamma}  g(x)g(x_1)dx_1 dx,\notag\\
    I_2 &= \int_{X^2} mm_1^\gamma |v_1|^{2\gamma}  g(x)g(x_1)dx_1 dx,\notag\\
    I_3 &= \int_{X^2} m^\gamma m_1 |v|^{2\gamma}  g(x)g(x_1)dx_1 dx,\notag\\
    I_4 &= \int_{X^2} m^\gamma m_1|v_1|^{2\gamma}  g(x)g(x_1)dx_1 dx.
\end{align}
Since $I_1=I_4, I_2=I_3$, H\"older's inequality gives
\begin{align*}
 &I_1=I_4=
 \left(\int_Xm|v|^{2\gamma}g(x)\, d x\right)
 \left(\int_Xm^\gamma g(x)\, d x\right)
 \leq M_0(g)^{1-\gamma}M_1(g)M_2(g)^\gamma,\\
 &I_2=I_3=
 \left(\int_Xm g(x)\, d x\right)
 \left(\int_X(m_1|v_1|^2)^\gamma
 g(x_1)\, d x_1\right)\leq M_1(g)M_0(g)^{1-\gamma}M_2(g)^\gamma.
\end{align*}
This gives \eqref{eq:massrate} with
\(C_\gamma=4c_\gamma\).
\end{proof}

\begin{proposition}[Conservation of \(M_0\) and \(M_1\)]\label{prop:M01}
The global solution constructed in \cref{prop:bounded-measurable-tests} satisfies
\[
 M_0(f(t))=M_0(f_0),\qquad M_1(f(t))=M_1(f_0),
 \qquad t\geq 0.
\]
\end{proposition}

\begin{proof}
Taking \(\psi=1\) in \eqref{eq:bounded-measurable-identity} gives
\(\Delta\psi=0\).  Hence \( M_0(f(t))=M_0(f_0).\)

It remains to treat the unbounded test \(m\).  Choose
\(0\leq\chi_k\in C_c^1(X)\) increasing pointwise to one, as the definition in the proof of \cref{thm:compact} and set
\(\psi_k(x)=m\chi_k(x)\).  Then \(\psi_k\) increases
pointwise to \(m\) and lies in \(C_c^1(X)\).  Conservation
of pair mass gives \(\Delta m=m'+m_1'-m-m_1=0\), while the bound on \(\psi_k\) gives
\(
 \abs{\Delta\psi_k}\leq
 m'+m_1'+m+m_1=2(m+m_1).
\)
Moreover, \(\Delta\psi_k\to\Delta m=0\) for every collision configuration. Fix $T>0$, by \cref{lem:massrate} and
\eqref{eq:limitmomentineq},
\begin{align}
 &\int_0^T\int_{X^2}\int_0^1\int_{\Sph}
 aE^\gamma b\left(\frac u{|u|}\cdot\omega\right)
 (m+m_1)f(s,x)f(s,x_1)
 \, d\omega\, d\alpha\, d x_1\, d x\, d s\notag\\
 &\leq TC_\gamma\norm a_\infty\norm b_{L^1}
 M_0(f_0)^{1-\gamma}M_1(f_0)M_2(f_0)^\gamma<\infty.
 \label{eq:M1-dominating-rate}
\end{align}
Since \(\psi_k\) is bounded, \eqref{eq:bounded-measurable-identity} gives,
for every \(t\in[0,T]\),
\begin{equation}\label{eq:M1-integrated}
 \int_X\psi_k(x)f(t,x)\, d x
 =\int_X\psi_k(x)f_0(x)\, d x
 +\int_0^t\mathcal Q(f(s),f(s))[\psi_k]\, d s.
\end{equation}
Since \(\Delta\psi_k\to0\) pointwise and
\[
 \frac12aE^\gamma b\abs{\Delta\psi_k}f(s,x)f(s,x_1)
 \leq aE^\gamma b(m+m_1)f(s,x)f(s,x_1),
\]
\eqref{eq:M1-dominating-rate} and dominated convergence give
\begin{equation}\label{eq:M1-collision-vanishes}
 \sup_{0\leq t\leq T}
 \abs{\int_0^t\mathcal Q(f(s),f(s))[\psi_k]\, d s}
 \leq\int_0^T\abs{\mathcal Q(f(s),f(s))[\psi_k]}\, d s
 \to 0.
\end{equation}
Taking \(k\to\infty\) in \eqref{eq:M1-integrated} and applying monotone
convergence to both one-particle terms yields
\(M_1(f(t))=M_1(f_0)\) for every \(t\in[0,T]\). Since $T$ is arbitrary, the assertion holds for all $t\geq 0$.
\end{proof}

\subsubsection{A higher energy moment and energy conservation}

Polynomial moment estimates for the classical spatially homogeneous
hard-potential Boltzmann equation are commonly obtained through
Povzner-type inequalities and their refinements; see, for example,
\cite{Desvillettes1993,MischlerWennberg1999,Wennberg1997,
LuMouhot2012}. These methods yield propagation and, under suitable
assumptions, production of higher velocity moments. In the present
mass-exchange model, the relevant one-particle energy is
\(m|v|^2\), and pairwise energy conservation does not provide the
classical dissipative Povzner mechanism for convex powers of this
quantity; see \cref{rem:no-moment-generation}. The result below
therefore concerns propagation of a higher-energy moment already
present in the initial datum. Its proof relies directly on pair-energy
conservation and an elementary convexity estimate.

\begin{proposition}[Uniform propagation of higher energy moment]
\label{prop:higherenergy}
Assume \eqref{eq:IC2} and put \(p=1+\delta\).  Let $f_N$ be the truncated solutions of the bounded-kernel equation in \cref{sec:bounded}. For every \(T>0\), $N\ge 1$ and $t\in [0,T]$
\begin{align}
 \int_X(m|v|^2)^p f_N(t,m,v)\, d m\, d v
 &\leq e^{K_pT}
 \left[M_0(f_0)+\int_X(m|v|^2)^pf_0(m,v)\, d m\, d v\right]-M_0(f_0),
 \label{eq:highenergy}\\
 K_p&=p2^{p-1+\gamma}\norm a_\infty\norm b_{L^1}M_2(f_0).
 \label{eq:highenergy-Kp}
\end{align}
\end{proposition}

\begin{proof}
For fixed \(N\), define
\[
 W_p(m,v)=1+m+m|v|^2+(m|v|^2)^p,
 \qquad \X_p=L^1(X;W_p d x).
\]
Pair-energy conservation $m'|v'|^2+m_1'|v_1'|^2=m|v|^2+m_1|v_1|^2$ and
\[
 (m'|v'|^2)^p+(m_1'|v_1'|^2)^p\leq (m'|v'|^2+m_1'|v_1'|^2)^p
 =(m|v|^2+m_1|v_1|^2)^p
 \leq2^{p-1}[(m|v|^2)^p+(m_1|v_1|^2)^p]
\]
give
\[
 W_p(m',v')+W_p(m_1',v_1')
 \leq 2^{p-1}[W_p(m,v)+W_p(m_1,v_1)].
\]
The four nonnegative marginal measures used in the construction of
\(\mathbf Q_N(g,h)\) in \cref{lem:boundedQ} give the total-variation
estimate
\begin{align}
 \norm{\mathbf Q_N(g,h)}_{\X_p}
 &\leq\frac14\int_{X^2}\int_0^1\int_{\Sph}
 aB_N\left[\abs{g(x)}\abs{h(x_1)}
 +\abs{h(x)}\abs{g(x_1)}\right]\notag\\
 &\qquad\times
 \left[W_p(m',v')+W_p(m_1',v_1')
 +W_p(m,v)+W_p(m_1,v_1)\right]
 \, d\omega\, d\alpha\, d x_1\, d x.
 \label{eq:Xp-four-marginal}
\end{align}
Since \(a\leq\norm a_\infty\), \(B_N\leq N\), and
\(|\Sph|<\infty\), the preceding pair-weight inequality implies
\begin{align}
 \norm{\mathbf Q_N(g,h)}_{\X_p}
 &\leq
 \frac{1}4 (2^{p-1}+1)N\norm a_\infty|\Sph|
 \int_{X^2}
 [\abs{g(x)}\abs{h(x_1)}
 +\abs{h(x)}\abs{g(x_1)}]\notag\\
 &\qquad\times
 \left[W_p(m,v)+W_p(m_1,v_1)\right]
 \, d x_1\, d x\notag\\
 &=\frac{1}2(2^{p-1}+1)N\norm a_\infty|\Sph|
 [
 \norm g_{\X_p}\norm h_{L^1}
 +\norm g_{L^1}\norm h_{\X_p}
 ].
 \label{eq:Xp-bilinear-bound}
\end{align}
Bilinearity then gives the corresponding local Lipschitz bound on \(\X_p\), while
Picard--Lindel\"of theory \cite{Deimling1977} supplies a local \(\X_p\)-solution; the continuous
embedding \(\X_p\subset\X\) and uniqueness in \(\X\) identify it with the
solution of \cref{thm:bounded} for as long as its \(\X_p\)-norm is finite.

We now derive a bound independent of \(N\).  The four particle energies
are nonnegative, and pair-energy conservation gives \(m'|v'|^2+m_1'|v_1'|^2=m|v|^2+m_1|v_1|^2.\)
On \(X_1^2=\{m|v|^2\geq m_1|v_1|^2\}\), convexity and the
mean-value theorem give
\begin{align}
 &\left[(m'|v'|^2)^p+(m_1'|v_1'|^2)^p
 -(m|v|^2)^p-(m_1|v_1|^2)^p\right]_+\notag\\
 &\quad\leq
 (m|v|^2+m_1|v_1|^2)^p
 -(m|v|^2)^p-(m_1|v_1|^2)^p\notag\\
 &\quad\leq
 (m|v|^2+m_1|v_1|^2)^p-(m|v|^2)^p\notag\\
 &\quad\leq p2^{p-1}(m|v|^2)^{p-1}m_1|v_1|^2.
 \label{eq:pinc}
\end{align}
On the same set, \eqref{eq:E-energy} gives
\(E^\gamma\leq2^\gamma(m|v|^2)^\gamma\).
Since \(p-1+\gamma<p\), one has
\((m|v|^2)^{p-1+\gamma}\leq1+(m|v|^2)^p\).  Therefore
\begin{align}
 &E^\gamma[(m'|v'|^2)^p+(m_1'|v_1'|^2)^p
 -(m|v|^2)^p-(m_1|v_1|^2)^p]_+\notag\\*
 &\quad\leq p2^{p-1+\gamma}
 [1+(m|v|^2)^p]m_1|v_1|^2
 \quad\text{on }X_1^2.
 \label{eq:pclosure}
\end{align}
On \(X_2^2=\{m_1|v_1|^2>m|v|^2\}\), interchanging the two incoming
particles gives
\begin{align}
 &E^\gamma\left[(m'|v'|^2)^p+(m_1'|v_1'|^2)^p
 -(m|v|^2)^p-(m_1|v_1|^2)^p\right]_+\notag\\*
 &\quad\leq p2^{p-1+\gamma}
 [1+(m_1|v_1|^2)^p]m|v|^2.
 \label{eq:pclosure-symmetric}
\end{align}
As long as the local \(\X_p\)-solution exists, pairing its strong equation
with \((m|v|^2)^p\) gives
\begin{align}
 &\frac{ d}{ d t}\int_X(m|v|^2)^pf_N(t,x)\, dx\notag\\&=\frac12\int_{X^2}\int_0^1\int_{\Sph}
 aB_N[
 (m'|v'|^2)^p+(m_1'|v_1'|^2)^p-(m|v|^2)^p-(m_1|v_1|^2)^p
 ]
 f_N(t,x)f_N(t,x_1)\notag\\
 &\leq\frac12
 \int_{X_1^2}
 \int_0^1\int_{\Sph}
 aE^\gamma b[
 (m'|v'|^2)^p+(m_1'|v_1'|^2)^p-(m|v|^2)^p-(m_1|v_1|^2)^p
 ]_+
 f_N(t,x)f_N(t,x_1)\notag\\
 &+\frac12
 \int_{X_2^2}
 \int_0^1\int_{\Sph}
 aE^\gamma b[
 (m'|v'|^2)^p+(m_1'|v_1'|^2)^p-(m|v|^2)^p-(m_1|v_1|^2)^p
 ]_+
 f_N(t,x)f_N(t,x_1).
 \label{eq:pderivative-split}
\end{align}
Using \eqref{eq:pclosure} and \eqref{eq:pclosure-symmetric}, enlarging each of the two
integration regions $X_1^2$ and $X_2^2$ to \(X^2\), and factoring the resulting product
integrals gives
\begin{align}
 &\frac{ d}{ d t}\int_X(m|v|^2)^pf_N(t,x)\, d x
 \notag\\ &\leq\frac{1}2p2^{p-1+\gamma}
 \norm a_\infty\norm b_{L^1}\left[
 \left(\int_X[1+(m|v|^2)^p]f_N(t,x)\, d x\right)
 \left(\int_Xm_1|v_1|^2f_N(t,x_1)\, d x_1\right)
 \right.\notag\\
 &\qquad\left.+
 \left(\int_Xm|v|^2f_N(t,x)\, d x\right)
 \left(\int_X[1+(m_1|v_1|^2)^p]
 f_N(t,x_1)\, d x_1\right)
 \right]\notag\\
 &=p2^{p-1+\gamma}\norm a_\infty\norm b_{L^1}M_2(f_0)
 \left[M_0(f_0)+\int_X(m|v|^2)^pf_N(t,x)\, d x\right].
 \label{eq:pderivative-factorized}
\end{align}
In particular,
\begin{align}
 \frac{ d}{ d t}\int_X(m|v|^2)^pf_N(t,x)\, d x
 &\leq p2^{p-1+\gamma}\norm a_\infty\norm b_{L^1}M_2(f_0)
 \left[M_0(f_0)+\int_X(m|v|^2)^pf_N(t,x)\, d x\right].
 \label{eq:pgronwall}
\end{align}
Gronwall's inequality and \eqref{eq:IC2} prove \eqref{eq:highenergy}.  They
also prevent blow-up of the stronger $\mathcal X_p$-norm, so the preliminary $\X_p-$
solution persists for every finite time, justifying the pairing used above.
\end{proof}

\begin{proposition}[Higher-energy bound and energy equality]
\label{prop:M2}
Assume \eqref{eq:IC2}, and put \(p=1+\delta\). Let $f$ be the global solution constructed in \cref{prop:bounded-measurable-tests}.  Then, for every \(T>0\), and $t\in [0,T],$
\begin{align}
\int_X(m|v|^2)^pf(t,m,v)\, d m\, d v
 &\leq
 \left[M_0(f_0)+\int_X(m|v|^2)^pf_0(m,v)\, d m\, d v\right]
 e^{K_pT}-M_0(f_0).
 \label{eq:limit-highenergy}
\end{align}
Moreover,
\begin{equation}\label{eq:global-M2-equality}
 M_2(f(t))=M_2(f_0),\qquad t\geq0.
\end{equation}
\end{proposition}

\begin{proof}
Fix \(T>0\) and denote the right-hand side of
\eqref{eq:limit-highenergy} by \(C_T\).  For \(R>0\),
\(\min\{(m|v|^2)^p,R\}\in L^\infty(X)\).  Hence
\eqref{eq:global-Cweak-explicit} and \eqref{eq:highenergy} gives, for every
\(t\in[0,T]\),
\[
 \int_X\min\{(m|v|^2)^p,R\}f(t,x)\, d x
 =\lim_{k\to\infty}\int_X\min\{(m|v|^2)^p,R\}
 f_{N_k}(t,x)\, d x\leq C_T.
\]
Monotone convergence as \(R\to\infty\) proves
\eqref{eq:limit-highenergy}.

For \(A>0\), \eqref{eq:global-Cweak-explicit} gives
\begin{equation}\label{eq:energy-truncated-convergence}
 \sup_{0\leq t\leq T}
 \abs{\int_X\min\{m|v|^2,A\}[f_{N_k}(t,x)-f(t,x)]\, d x}
 \longrightarrow0.
\end{equation}
Since
\begin{equation}\label{eq:energy-tail-pointwise}
 0\leq m|v|^2-\min\{m|v|^2,A\}
 \leq A^{1-p}(m|v|^2)^p,
\end{equation}
then
\begin{align}
    0 &\leq M_2(f_{N_k}(t))-\int_X\min\{m|v|^2,A\}f_{N_k}(t) dx\leq A^{1-p}C_T,\\
    0&\leq M_2(f(t))-\int_X\min\{m|v|^2,A\}f(t) dx\leq A^{1-p}C_T.
\end{align}
Hence \eqref{eq:highenergy} and \eqref{eq:limit-highenergy} imply
\begin{align}
 \sup_{0\leq t\leq T}\abs{M_2(f_{N_k}(t))-M_2(f(t))}
 &\leq\sup_{0\leq t\leq T}
 \abs{\int_X\min\{m|v|^2,A\}[f_{N_k}(t)-f(t)]\, d x}
 +2C_TA^{1-p}.
 \label{eq:energy-total-difference}
\end{align}
First let \(k\to\infty\), then let \(A\to\infty\).  Since \(p>1\),
\begin{equation}\label{eq:energy-uniform-convergence}
 \sup_{0\leq t\leq T}\abs{M_2(f_{N_k}(t))-M_2(f(t))}
 \longrightarrow0.
\end{equation}
The identity \(M_2(f_{N_k}(t))=M_2(f_0)\) from
\eqref{eq:boundedmoments} proves \eqref{eq:global-M2-equality} on
\([0,T]\).  The arbitrariness of \(T\) completes the proof.
\end{proof}

\section{Uniqueness in the energy-dissipating class}
\label{sec:uniqueness}
For the classical spatially homogeneous Boltzmann equation, moment
bounds combined with weighted stability estimates provide a standard
route to uniqueness under Grad cutoff hard-potential assumptions; see
\cite{Wennberg1994,LuMouhot2012}. The restriction to the energy-dissipating class is also consistent with the classical homogeneous theory: without imposing an energy condition, weak
solutions need not be unique; see \cite{Wennberg1999}. In the present
model, the collision frequency contains the hard-potential factor
\(E^\gamma\), while the masses and velocities of both outgoing
particles change simultaneously. 
The stability estimate must therefore control the particle-energy
weight together with the growth of the collision frequency. The
\(1+\gamma\) higher-energy moment is precisely the additional
integrability that closes this estimate.  We will first
show that the kinetic-energy inequality alone propagates every
higher-energy moment already present in the initial datum.  Thus
\eqref{eq:IC3} supplies the required \(1+\gamma\) moment for every
energy-dissipating solution with initial datum \(f_0\).
\subsection{Higher-energy propagation}
\begin{proposition}[Higher-energy propagation for energy-dissipating solutions]
\label{prop:dissipating-higherenergy}
Assume \eqref{eq:IC2}, put \(p=1+\delta\), and let \(q\) be an
energy-dissipating global \(L^1\)-integral weak solution with initial
datum \(f_0\).  Then, for every \(T>0\), and $t\in [0,T]$
\begin{align}
\int_X(m|v|^2)^p q(t,m,v) d m d v
 &\leq
 \left[M_0(f_0)+\int_X(m|v|^2)^p f_0(m,v) d m d v\right]
 e^{K_pT}-M_0(f_0),
 \label{eq:dissipating-highenergy}\\
 K_p&=p2^{p-1+\gamma}\norm a_\infty\norm b_{L^1}M_2(f_0).
 \label{eq:dissipating-Kp}
\end{align}
If, in addition, \(p\geq1+\gamma\), then \(q\) conserves kinetic
energy:
\begin{equation}\label{eq:dissipating-energy-equality}
 M_2(q(t))=M_2(f_0),\qquad t\geq0.
\end{equation}
\end{proposition}

\begin{proof}
We divide the proof into six steps. First, we record conservation of the particle number $M_0(q(t))=M_0(f_0)$. We then introduce a convex truncation $\Phi_R(z)$ of $z^p$ and derive an integral inequality for the corresponding truncated moment $H_R(t)$. The main estimate is a uniform bound on the positive collision increment $[\Delta\Phi_R]_+$ of this truncation. After integrating this estimate, Gronwall’s inequality gives the desired higher-energy bound, and monotone convergence removes the truncation. Finally, under $p\geq 1+\gamma$, we use the resulting integrability to pass to the unbounded kinetic-energy test and prove energy conservation.

Step 1: conservation of the particle number and notation. Pairing the \(L^1(X)\)-valued integral equation for \(q\) with the
bounded test \(1\) gives
\begin{equation}\label{eq:dissipating-M0}
 M_0(q(t))=M_0(f_0),\qquad t\geq0.
\end{equation}
Write
\[
 e=m|v|^2,\qquad e_1=m_1|v_1|^2,
 \qquad e'=m'|v'|^2,\qquad e_1'=m_1'|v_1'|^2.
\]

Step 2: convex truncation $\Phi_R(z)$ and the truncated-moment inequality. Since $z^p$ is not a bounded admissible test, we first replace it by a convex function with affine growth. For \(R>0\), define the convex \(C^1\) function
\begin{equation}\label{eq:affine-tail-Phi}
 \Phi_R(z)=
 \begin{cases}
  z^p,&0\leq z\leq R,\\
  pR^{p-1}z-(p-1)R^p,&z>R,
 \end{cases}
 \qquad z\geq0.
\end{equation}
Its tail is affine, and
\begin{equation}\label{eq:bounded-affine-defect}
 \chi_R(z):=pR^{p-1}z-\Phi_R(z)
 \quad\text{satisfies}\quad
 0\leq\chi_R(z)\leq(p-1)R^p.
\end{equation}
Consequently, \(x\mapsto\chi_R(m|v|^2)\) is an admissible bounded
Borel test.  Set
\begin{equation}\label{eq:HR-definition}
 H_R(t)=\int_X\Phi_R(m|v|^2)q(t,m,v) d m d v.
\end{equation}
This quantity is finite because
\(0\leq\Phi_R(z)\leq pR^{p-1}z\) and \(M_2(q(t))\leq M_2(f_0)<\infty\).
By \eqref{eq:bounded-affine-defect},
\begin{equation}\label{eq:HR-defect-identity}
 H_R(t)=pR^{p-1}M_2(q(t))
 -\int_X\chi_R(m|v|^2)q(t,m,v) d m d v.
\end{equation}
The representation \eqref{eq:HR-defect-identity} allows us to combine the bounded-test identity for $\chi_R$ with the energy-dissipation inequality for $M_2$. Indeed, $M_2(q(t))\leq M_2(f_0)$ and the bounded-test integral identity
for \(\chi_R(m|v|^2)\) yield
\begin{align}
 H_R(t)
 &\leq pR^{p-1}M_2(f_0)
 -\int_X\chi_R(m|v|^2)f_0(m,v) d m d v
 -\int_0^t\mathcal Q(q(s),q(s))[\chi_R(m|v|^2)] d s\notag\\
 &=H_R(0)+\int_0^t\mathcal I_R(s) d s,
 \label{eq:HR-pre-Gronwall}
\end{align}
where pair-energy conservation \(e'+e_1'=e+e_1\) gives
\begin{align}
 \mathcal I_R(s)
 &:=-\mathcal Q(q(s),q(s))[\chi_R(m|v|^2)]\notag\\
 &=\frac12\int_{X^2}\int_0^1\int_{\Sph}
 aE^\gamma b\left(\frac u{|u|}\cdot\omega\right)
 \Delta\Phi_R\,q(s,x)q(s,x_1)
  d\omega d\alpha d x_1 d x,
 \label{eq:IR-definition}\\
 \Delta \Phi_R&=\Phi_R(e')+\Phi_R(e_1')-\Phi_R(e)-\Phi_R(e_1).
 \label{eq:DeltaR-definition}
\end{align}
Indeed,
\(\Delta\chi_R=pR^{p-1}(e'+e_1'-e-e_1)-\Delta \Phi_R=-\Delta\Phi_R\). Thus, it remains to estimate the positive part of the collision increment $\Delta \Phi_R$, with a bound independent of $R$.

Step 3: uniform estimate of the positive collision increment $[\Delta\Phi_R]_+$. In this step, we prove the pointwise estimate
\[
E^\gamma[\Delta\Phi_R]_+\leq p 2^{p-1+\gamma}(e_1(1+\Phi_R(e))+e(1+\Phi_R(e_1))),
\]
which is uniformly in $R.$

We first work on the region $X^2_1=\{(x,x_1)\in X^2:e\geq e_1\}.$  Since \(\Phi_R\) is convex,
nondecreasing, and \(\Phi_R(0)=0\),
\begin{align}
 [\Delta\Phi_R]_+
 &\leq\Phi_R(e+e_1)-\Phi_R(e)=\int_{e}^{e+e_1}\Phi'_R(z) dz
 \leq e_1\Phi_R'(e+e_1)
 \leq e_1\Phi_R'(2e).
 \label{eq:DeltaR-positive-first}
\end{align}
Here the first inequality uses \(e'+e_1'=e+e_1\) and
\(\Phi_R(e')+\Phi_R(e_1')\leq\Phi_R(e+e_1)\); the latter follows by
applying convexity to the two nonnegative summands with fixed sum.

For every \(e\geq0\),
\begin{equation}\label{eq:PhiR-derivative-bound}
 e^\gamma\Phi_R'(2e)
 \leq p2^{p-1}\bigl[1+\Phi_R(e)\bigr].
\end{equation}
To verify this, first
let \(e\leq R\).  Then
\(\Phi_R'(2e)\leq p(2e)^{p-1}\), and
\(e^{p-1+\gamma}\leq1+e^p=1+\Phi_R(e)\), because
\(0<\gamma<1\).  \eqref{eq:PhiR-derivative-bound} holds. If \(e>R\), then
\(\Phi_R'(2e)=pR^{p-1}\) and
\(\Phi_R(e)\geq R^{p-1}e\).  When \(e\geq1\), use
\(e^\gamma\leq e\); when \(e<1\), necessarily \(R<1\), and
\(R^{p-1}e^\gamma\leq1\).  These two cases prove
\eqref{eq:PhiR-derivative-bound}.

Since \(E\leq e+e_1\leq2e\) on \(X_1^2\),
\eqref{eq:DeltaR-positive-first} and
\eqref{eq:PhiR-derivative-bound} give
\begin{equation}\label{eq:DeltaR-rate-first}
 E^\gamma[\Delta\Phi_R]_+
 \leq p2^{p-1+\gamma}
 e_1\bigl[1+\Phi_R(e)\bigr]
 \qquad\text{on }X_1^2.
\end{equation}
On the complementary region $X_2^2=\{(x,x_1)\in X^2: e_1>e\}$, the same argument with the two incoming particles interchanged gives
\begin{equation}\label{eq:DeltaR-rate-second}
 E^\gamma[\Delta\Phi_R]_+
 \leq p2^{p-1+\gamma}
 e\bigl[1+\Phi_R(e_1)\bigr]
 \qquad\text{on }X_2^2.
\end{equation}

Step 4: integration of the collision estimate and closure of the truncated moment bound. Apply \eqref{eq:DeltaR-rate-first}--\eqref{eq:DeltaR-rate-second},
and enlarge each of the two regions $X_1^2,X_2^2$ to \(X^2\).  Factoring the two
resulting product integrals gives
\begin{align}
 \mathcal I_R(s)
 &\leq p2^{p-1+\gamma}\norm a_\infty\norm b_{L^1}
 M_2(q(s))\bigl[M_0(q(s))+H_R(s)\bigr]\notag\\
 &\leq K_p\bigl[M_0(f_0)+H_R(s)\bigr],
 \label{eq:IR-factorized}
\end{align}
where the last line uses \eqref{eq:dissipating-M0} and
\(M_2(q(s))\leq M_2(f_0)\). Here the factors containing $e$ or $e_1$ give $M_2(q(s))$ while the factors $1+\Phi_R$ give $M_0(q(s))+H_R(s).$ Combining \eqref{eq:HR-pre-Gronwall} and
\eqref{eq:IR-factorized}, Gronwall's inequality yields
\begin{equation}\label{eq:HR-bound}
 H_R(t)\leq
 \bigl[M_0(f_0)+H_R(0)\bigr]e^{K_pt}-M_0(f_0).
\end{equation}

Step 5: Removal of the truncation. Since \(\Phi_R(z)\leq z^p\), the initial term satisfies
\[H_R(0)\leq\int_X(m|v|^2)^pf_0(x) dx.\] Moreover, 
\(\Phi_R(z)\uparrow z^p\) for every \(z\geq0\) as \(R\to\infty\).
Monotone convergence in \eqref{eq:HR-bound}, followed by taking the
supremum over \(0\leq t\leq T\), proves
\eqref{eq:dissipating-highenergy}.

Step 6: conservation of kinetic energy.  The estimate proved in Steps 1–5 provides the integrability needed to approximate the unbounded kinetic-energy test $m|v|^2$ by bounded tests. Assume \(p\geq1+\gamma\),
fix \(T>0\), and set
\[
 \psi_A(m,v)=\min\{m|v|^2,A\},\qquad A>0.
\]
This is a bounded Borel test, \(\psi_A\uparrow m|v|^2\), and for every
collision configuration
\begin{align}
 |\Delta\psi_A|
 &\leq e'+e_1'+e+e_1=2(e+e_1),
 \label{eq:energy-cutoff-increment}
\end{align}
while \(\Delta\psi_A\to e'+e_1'-e-e_1=0\) as \(A\to\infty\).
Moreover, \eqref{eq:dissipating-highenergy} and
\(e^{1+\gamma}\leq1+e^p\) imply
\begin{equation}\label{eq:dissipating-one-plus-gamma-bound}
 \sup_{0\leq s\leq T}\int_Xe^{1+\gamma}q(s,x) d x<\infty.
\end{equation}
Using \(E\leq e+e_1\) and
\((e+e_1)^{1+\gamma}\leq
2^\gamma(e^{1+\gamma}+e_1^{1+\gamma})\), we obtain
\begin{align}
 &\int_0^T\int_{X^2}\int_0^1\int_{\Sph}
 aE^\gamma b\left(\frac u{|u|}\cdot\omega\right)
 (e+e_1)q(s,x)q(s,x_1)
  d\omega d\alpha d x_1 d x d s\notag\\
 &\quad\leq
 2^{1+\gamma}T\norm a_\infty\norm b_{L^1}M_0(f_0)
 \sup_{0\leq s\leq T}\int_Xe^{1+\gamma}q(s,x) d x<\infty.
 \label{eq:energy-cutoff-dominating-rate}
\end{align}
Therefore \eqref{eq:energy-cutoff-increment},
\eqref{eq:energy-cutoff-dominating-rate}, and dominated convergence
give
\begin{equation}\label{eq:energy-cutoff-collision-vanishes}
 \int_0^t\mathcal Q(q(s),q(s))[\psi_A] d s\longrightarrow0
 \qquad\text{for every }t\in[0,T].
\end{equation}
The bounded-test identity reads
\[
 \int_X\psi_A(x)q(t,x) d x
 =\int_X\psi_A(x)f_0(x) d x
 +\int_0^t\mathcal Q(q(s),q(s))[\psi_A] d s.
\]
Letting \(A\to\infty\), monotone convergence in the two one-particle
terms and \eqref{eq:energy-cutoff-collision-vanishes} prove
\eqref{eq:dissipating-energy-equality} on \([0,T]\).  Since \(T\) is
arbitrary, the equality holds for all \(t\geq0\).
\end{proof}

The stability estimate is naturally formulated with the weight
\begin{equation}\label{eq:uniqueness-weight}
 \varpi(m,v)=1+m|v|^2,
 \qquad
 \mathcal Y=L^1(X;\varpi(x)\, d x),
 \qquad
 \norm g_{\mathcal Y}=\int_X\varpi(x)\abs{g(x)}\, d x.
\end{equation}
The two-particle collision structure and the energy identity
\eqref{eq:energycons} give
\begin{equation}\label{eq:varpi-pair-conservation}
 \varpi(m',v')+\varpi(m_1',v_1')
 =\varpi(m,v)+\varpi(m_1,v_1).
\end{equation}
Under \eqref{eq:IC3}, \cref{prop:dissipating-higherenergy} applies with
\(\delta=\gamma\) and \(p=1+\gamma\) to every energy-dissipating
solution with initial datum \(f_0\).  This is the moment input used in
the weighted stability argument below.

\subsection{Weighted Kato estimate and Uniqueness}

\begin{lemma}[Weighted four-marginal realization]
\label{lem:weighted-four-marginal}
Let \(\rho\in L^1(X^2\times(0,1)\times\Sph)\) be real-valued and
assume
\begin{align}
 &\int_{X^2}\int_0^1\int_{\Sph}
 \abs{\rho(x,x_1,\alpha,\omega)}
 \left[\varpi(x)+\varpi(x_1)\right]
 \, d\omega\, d\alpha\, d x_1\, d x<\infty.
 \label{eq:weighted-rho-assumption}
\end{align}
Then the four-marginal map \(\mathcal T\) from the proof of
\cref{lem:hard-Q-density} satisfies
\begin{align}
 \mathcal T\rho&\in\mathcal Y,\label{eq:T-rho-Y}\\
 \norm{\mathcal T\rho}_{\mathcal Y}
 &\leq
 2\int_{X^2}\int_0^1\int_{\Sph}
 \abs{\rho(x,x_1,\alpha,\omega)}
 \left[\varpi(x)+\varpi(x_1)\right]
 \, d\omega\, d\alpha\, d x_1\, d x.
 \label{eq:weighted-T-bound}
\end{align}
For every Borel measurable \(\psi:X\to\mathbb R\) satisfying
\(\abs{\psi(x)}\leq C\varpi(x)\),
\begin{align}
 \int_X\psi(z)\mathcal T\rho(z)\, d z
 &=
 \int_{X^2}\int_0^1\int_{\Sph}
 \rho(x,x_1,\alpha,\omega)
 \left[\psi(x')+\psi(x_1')-\psi(x)-\psi(x_1)\right]
 \, d\omega\, d\alpha\, d x_1\, d x.
 \label{eq:weighted-T-pairing}
\end{align}
The four marginal terms in \eqref{eq:weighted-T-pairing} are absolutely
convergent separately.
\end{lemma}

\begin{proof}
Let \(G^1,G^2,L^1,L^2\) be the four signed marginals associated with
\(\rho\), as in the proof of \cref{lem:hard-Q-density}.  The collision
change of variables and \eqref{eq:varpi-pair-conservation} give
\begin{align}
 &\|G^1\|_{\mathcal Y}+\|G^2\|_{\mathcal Y}
 +\|L^1\|_{\mathcal Y}+\|L^2\|_{\mathcal Y}\notag\\
 &\quad\leq
 \int_{X^2}\int_0^1\int_{\Sph}
 \abs\rho
 \left[
 \varpi(x')+\varpi(x_1')+\varpi(x)+\varpi(x_1)
 \right]
 \, d\omega\, d\alpha\, d x_1\, d x\notag\\
 &\quad=
 2\int_{X^2}\int_0^1\int_{\Sph}
 \abs\rho\left[\varpi(x)+\varpi(x_1)\right]
 \, d\omega\, d\alpha\, d x_1\, d x.
 \label{eq:weighted-four-marginal-total}
\end{align}
Since
\(
 \mathcal T\rho=G^1+G^2-L^1-L^2,
\)
\eqref{eq:T-rho-Y} and \eqref{eq:weighted-T-bound} follow.

For \(R>0\), set
\(
 \psi_R=(-R)\vee(\psi\wedge R).
\)
Formula \eqref{eq:hard-marginal-pairing} applies to \(\psi_R\).
On each of the four marginals,
\(\psi_R\to\psi\) pointwise and
\(\abs{\psi_R}\leq\abs\psi\leq C\varpi\).
The bound in \eqref{eq:weighted-four-marginal-total} permits dominated
convergence separately in the two gain terms and the two loss terms.
Letting \(R\to\infty\) proves \eqref{eq:weighted-T-pairing}.
\end{proof}

For real-valued functions \(h,k\), define
\begin{align}
 \rho_{h,k}(x,x_1,\alpha,\omega)
 =\frac14a(m,m_1,\alpha)E^\gamma
 b\left(\frac u{|u|}\cdot\omega\right)
 \left[h(x)k(x_1)+k(x)h(x_1)\right],
 \label{eq:uniqueness-polarized-density}
\end{align}
whenever the right-hand side satisfies
\eqref{eq:weighted-rho-assumption}, and set
\begin{equation}\label{eq:uniqueness-polarized-Q}
 \mathbf Q(h,k)=\mathcal T\rho_{h,k}.
\end{equation}
This convention agrees with the quadratic operator when \(h=k\).

\begin{lemma}[Weighted Kato estimate]
\label{lem:weighted-Kato}
Let \(k\geq0\) satisfy
\begin{equation}\label{eq:Kato-k-assumption}
 \int_X
 \left[
 1+m|v|^2+(m|v|^2)^{1+\gamma}
 \right]k(m,v)\, d m\, d v<\infty.
\end{equation}
Let \(h\in\mathcal Y\) be real-valued and satisfy
\(\abs h\leq k\) almost everywhere.  Then
\(\mathbf Q(h,k)\in\mathcal Y\), and
\begin{align}
 \int_X\varpi(x)\operatorname{sgn}(h(x))
 \mathbf Q(h,k)(x)\, d x\leq
 \norm a_\infty\norm b_{L^1}\norm h_{\mathcal Y}
 \int_X\varpi(x)
 [1+(m|v|^2)^\gamma]
 k(x)\, d x.
 \label{eq:weighted-Kato-estimate}
\end{align}
\end{lemma}

\begin{proof}
Since \(\abs h\leq k\),
\[
 \abs{\rho_{h,k}}
 \leq
 \frac12a(m,m_1,\alpha)E^\gamma
 b\left(\frac u{|u|}\cdot\omega\right)k(x)k(x_1).
\]
By \eqref{eq:E-energy} and \(0<\gamma<1\),
\(E^\gamma\leq(m|v|^2)^\gamma+(m_1|v_1|^2)^\gamma.\)
Consequently,
\begin{align}
 &\int_{X^2}\int_0^1\int_{\Sph}
 aE^\gamma b\,k(x)k(x_1)
 \left[\varpi(x)+\varpi(x_1)\right]
 \, d\omega\, d\alpha\, d x_1\, d x\notag\\
 &\leq
 2\norm a_\infty\norm b_{L^1}
 \left[
 M_0(k)\int_X
 \varpi(x)(m|v|^2)^\gamma k\, d x+
 \left(\int_X\varpi k\, d x\right)
 \left(\int_X(m|v|^2)^\gamma k\, d x\right)
 \right]<\infty.
 \label{eq:Kato-weighted-integrability}
\end{align}
Thus \cref{lem:weighted-four-marginal} applies to $\rho_{h,k}$, which implies $\mathbf Q(h,k)\in \mathcal Y$. Put
\[
 \sigma=\operatorname{sgn}h,\qquad
 \sigma_1=\operatorname{sgn}h(x_1),\qquad
 \sigma'=\operatorname{sgn}h(x'),\qquad
 \sigma_1'=\operatorname{sgn}h(x_1').
\]
Using
\(\abs\sigma,\abs{\sigma_1},\abs{\sigma'},\abs{\sigma_1'}\leq1\)
and \eqref{eq:varpi-pair-conservation}, we obtain
\begin{align}
 &h(x)
 \left[
 \varpi(x')\sigma'+\varpi(x_1')\sigma_1'
 -\varpi(x)\sigma-\varpi(x_1)\sigma_1
 \right]\notag\\
 &\quad\leq
 \abs{h(x)}
 \left[
 \varpi(x')+\varpi(x_1')-\varpi(x)+\varpi(x_1)
 \right]=
 2\abs{h(x)}\varpi(x_1).
 \label{eq:Kato-pointwise-first}
\end{align}
Interchanging \(x\) and \(x_1\) gives
\begin{align}
 h(x_1)
 \left[
 \varpi(x')\sigma'+\varpi(x_1')\sigma_1'
 -\varpi(x)\sigma-\varpi(x_1)\sigma_1
 \right]\leq
 2\abs{h(x_1)}\varpi(x).
 \label{eq:Kato-pointwise-second}
\end{align}
Apply \eqref{eq:weighted-T-pairing} with
\(\psi=\varpi\operatorname{sgn}h\), and use
\eqref{eq:Kato-pointwise-first} and \eqref{eq:Kato-pointwise-second}.  This gives
\begin{align}
 &\int_X\varpi\operatorname{sgn}(h)\mathbf Q(h,k)\, d x\notag\\
 &\quad\leq
 \frac12\int_{X^2}\int_0^1\int_{\Sph}
 aE^\gamma b
 \left[
 \abs{h(x)}k(x_1)\varpi(x_1)
 +k(x)\abs{h(x_1)}\varpi(x)
 \right]
 \, d\omega\, d\alpha\, d x_1\, d x.
 \label{eq:Kato-before-factorization}
\end{align}
Hence the first term on the right satisfies
\begin{align}
 &\int_{X^2}\int_0^1\int_{\Sph}
 aE^\gamma b\,
 \abs{h(x)}k(x_1)\varpi(x_1)
 \, d\omega\, d\alpha\, d x_1\, d x\notag\\
 &\quad\leq
 \norm a_\infty\norm b_{L^1}
 \left[
 \left(\int_X(m|v|^2)^\gamma\abs h\, d x\right)
 \left(\int_X\varpi k\, d x\right)+
 \left(\int_X\abs h\, d x\right)
 \left(\int_X\varpi(x)(m|v|^2)^\gamma k\, d x\right)
 \right]\notag\\
 &\quad\leq
 \norm a_\infty\norm b_{L^1}
  \|h\|_{\mathcal Y}
 \int_X\varpi
 \left[1+(m|v|^2)^\gamma\right]k\, d x
 \label{eq:Kato-factor-first}
\end{align}
The last inequality uses
\(
 (m|v|^2)^\gamma\leq1+m|v|^2=\varpi(m,v).
\)
The second term in \eqref{eq:Kato-before-factorization} satisfies the
same estimate after the measure-preserving substitution
\(
 (x,x_1,\alpha,\omega)\mapsto(x_1,x,\alpha,-\omega),
\)
using \eqref{eq:a-sym}.  The factor \(1/2\) in
\eqref{eq:Kato-before-factorization} proves
\eqref{eq:weighted-Kato-estimate}.
\end{proof}

\begin{theorem}[Uniqueness in the energy-dissipating class]
\label{thm:uniqueness}
Assume \textnormal{(K1)}, \textnormal{(K2)}, \textnormal{(K3)},
\eqref{eq:IC1}, and \eqref{eq:IC3}.  Let \(f\) and \(g\) be two global
\(L^1\)-integral weak solutions in the sense of
\cref{def:L1-integral}, with initial datum \(f_0\).  If both \(f\) and
\(g\) are energy-dissipating, then
\begin{equation}\label{eq:uniqueness-result}
 f(t)=g(t)\quad\text{in }L^1(X),\qquad t\geq0.
\end{equation}
\end{theorem}

\begin{proof}

We divide the proof into five steps. We first obtain the higher-energy bounds of $f,g$ required to control the collision operator $\mathbf Q(q,q)$. We then upgrade the two $L^1(X)$-valued integral equations to equations in the weighted space $\mathcal Y$. After subtracting the equations, we apply a weighted $L^1$-chain rule and the stability estimate of \cref{lem:weighted-Kato}. Finally, the higher-energy bound makes the coefficient in the resulting differential inequality integrable, so that Gronwall’s inequality applies.

Step 1: uniform higher-energy bounds. Fix \(T>0\).  Apply \cref{prop:dissipating-higherenergy} to \(f\) and
\(g\) with \(p=1+\gamma\).  Since they have the same initial datum,
\begin{align}
 &\sup_{0\leq t\leq T}
 \int_X(m|v|^2)^{1+\gamma}
 \left[f(t,m,v)+g(t,m,v)\right] d m d v\notag\\
 &\quad\leq
 2\left[
 M_0(f_0)+\int_X(m|v|^2)^{1+\gamma}f_0(m,v) d m d v
 \right]e^{K_{1+\gamma}T}-2M_0(f_0)<\infty.
 \label{eq:fg-higher-moment-class}
\end{align}
Together with conservation of $M_0$ and the energy-dissipation inequalities $M_2(f(t))\leq M_2(f_0)$ , $M_2(g(t))\leq M_2(f_0),$ estimate \eqref{eq:fg-higher-moment-class} provides uniform control on all the moments that will occur below.

Step 2: upgrade to the weighted space \(\mathcal Y.\) We next upgrade the two \(L^1(X)\)-valued equations to
\(\mathcal Y\)-valued equations.  Let \(q=f\) or \(q=g\), and define
\[
 \rho_q(t,x,x_1,\alpha,\omega)
 =\frac12a(m,m_1,\alpha)E^\gamma
 b\left(\frac u{|u|}\cdot\omega\right)
 q(t,x)q(t,x_1).
\]
Using
\[
 E^\gamma
 \leq (m|v|^2)^\gamma+(m_1|v_1|^2)^\gamma
\]
and symmetry in \(x,x_1\), we have
\begin{align*}
    \int_{X^2} E^\gamma q(x)q(x_1)[\varpi(x)+\varpi(x_1)] dx_1 dx &\leq  2\int_{X^2} (m|v|^2)^\gamma (2+m|v|^2+m_1|v_1|^2)q(x)q(x_1)dx_1dx\\&=2M_0(q)\int_X [(m|v|^2)^\gamma +(m|v|^2)^{1+\gamma}]q dx\\&+2(M_0(q)+M_2(q))\int_X(m|v|^2)^\gamma qdx,
\end{align*}
The right-hand side is finite by
\eqref{eq:fg-higher-moment-class}, conservation of \(M_0\), and the
energy-dissipation inequality. Since \(a\in L^\infty\) and
\(b\in L^1(\Sph)\), it follows that \(\rho_q(t)\) satisfies
\eqref{eq:weighted-rho-assumption}. We may therefore apply
\cref{lem:weighted-four-marginal}, which gives
\begin{align}
 \norm{\mathbf Q(q,q)}_{\mathcal Y}
 &\leq
 \int_{X^2}\int_0^1\int_{\Sph}
 aE^\gamma b\,q(x)q(x_1)
 \left[\varpi(x)+\varpi(x_1)\right]
 \, d\omega\, d\alpha\, d x_1\, d x\notag\\
 &\leq
 2\norm a_\infty\norm b_{L^1}
 \left[
 M_0(q)\int_X
 \left[
 (m|v|^2)^\gamma+(m|v|^2)^{1+\gamma}
 \right]q\, d x\right.\notag\\
 &\hspace{35mm}\left.+
 \left[M_0(q)+M_2(q)\right]
 \int_X(m|v|^2)^\gamma q\, d x
 \right].
 \label{eq:Q-in-Y-bound}
\end{align}
The defining weighted bound for weak solutions,
\eqref{eq:fg-higher-moment-class}, and
\((m|v|^2)^\gamma\leq1+m|v|^2\) show that the right-hand side
is uniformly bounded for \(0\leq t\leq T\).

To make the change of Banach space explicit, set
\[
 \mathcal Z=L^1\!\left(
 X^2\times(0,1)\times\Sph;
 [\varpi(x)+\varpi(x_1)]
 \, d x\, d x_1\, d\alpha\, d\omega
 \right).
\]
Joint measurability of the collision density, together with the bound
used in \eqref{eq:Q-in-Y-bound}, shows that
\(t\mapsto\rho_q(t)\) is strongly measurable in the separable space
\(\mathcal Z\) and belongs to
\(L^\infty(0,T;\mathcal Z)\).  By
\cref{lem:weighted-four-marginal},
\(\mathcal T:\mathcal Z\to\mathcal Y\) is bounded.  Therefore
\(
 \mathbf Q(q,q)=\mathcal T\rho_q
 \in L^\infty(0,T;\mathcal Y).
\)
In particular, the \(\mathcal Y\)-valued Bochner integral
\[
 F_q(t)=f_0+\int_0^t\mathbf Q(q(s),q(s))\, d s
\]
is well defined and belongs to \(C([0,T];\mathcal Y)\).  The defining
\(L^1(X)\)-integral identity for \(q\) gives \(q(t)=F_q(t)\) in
\(L^1(X)\) for every \(t\).  Since the continuous embedding
\(\mathcal Y\hookrightarrow L^1(X)\) is injective, \(F_q\) is the
\(\mathcal Y\)-valued representative of \(q\), and the same Bochner
identity holds in \(\mathcal Y\).  Consequently,
\(
 f,g\in W^{1,\infty}(0,T;\mathcal Y).
\)

Step 3: the equation for the difference. Set
\(
 h=f-g, k=f+g.
\)
Then \(k\geq0\), \(\abs h\leq k\), and expansion of
\eqref{eq:uniqueness-polarized-density} gives
\begin{equation}\label{eq:polarization-difference}
 \mathbf Q(f,f)-\mathbf Q(g,g)=\mathbf Q(h,k).
\end{equation}
Thus
\begin{equation}\label{eq:h-equation-Y}
 h\in W^{1,\infty}(0,T;\mathcal Y),\qquad
 \partial_th=\mathbf Q(h,k)
 \quad\text{in }\mathcal Y
 \quad\text{for a.e. }t\in(0,T).
\end{equation}

Step 4: weighted \(L^1\)-estimate for the difference. The \(L^1\)-norm chain rule on the measure space
\((X,\varpi(x)\, d x)\) gives, for almost every \(t\in(0,T)\),
\begin{equation}\label{eq:L1-chain-rule}
 \frac{ d}{ d t}\norm{h(t)}_{\mathcal Y}
 =
 \int_X\varpi(x)\operatorname{sgn}(h(t,x))
 \partial_th(t,x)\, d x.
\end{equation}
For completeness, apply the Sobolev chain rule first to
\[
 \beta_\varepsilon(r)=\sqrt{r^2+\varepsilon^2}-\varepsilon,
 \qquad
 \beta_\varepsilon'(r)=\frac{r}{\sqrt{r^2+\varepsilon^2}},
\]
integrate with respect to \(\varpi(x)\, d x\), and let
\(\varepsilon\downarrow0\).  Dominated convergence applies because
\(\abs{\beta_\varepsilon'}\leq1\) and
\(\partial_th\in\mathcal Y\).

By \cref{lem:weighted-Kato,eq:h-equation-Y},
\begin{align}
 \frac{ d}{ d t}\|h(t)\|_{\mathcal Y}
 &\leq
 \| a\|_\infty\| b\|_{L^1}\mathcal A_k(t)
 \|h(t)\|_{\mathcal Y},
 \label{eq:uniqueness-Gronwall-differential}\\
 \mathcal A_k(t)
 &=
 \int_X\varpi(x)
 \left[1+(m|v|^2)^\gamma\right]k(t,x)\, d x.
 \label{eq:Ak-definition}
\end{align}
Step 5: boundedness of the Gronwall coefficient. Since
\begin{align}
 \varpi(m,v)\left[1+(m|v|^2)^\gamma\right]
 &=
 1+m|v|^2+(m|v|^2)^\gamma
 +(m|v|^2)^{1+\gamma}\notag\\
 &\leq
 2\left[1+m|v|^2\right]
 +(m|v|^2)^{1+\gamma},
 \label{eq:Ak-moment-bound}
\end{align}
the basic weighted bounds and
\eqref{eq:fg-higher-moment-class} imply
\(
 \mathcal A_k\in L^\infty(0,T).
\)
Because \(h(0)=0\), Gronwall's inequality applied to
\eqref{eq:uniqueness-Gronwall-differential} yields
\(
 \norm{h(t)}_{\mathcal Y}=0, 0\leq t\leq T.
\)
Since \(T>0\) was arbitrary, \eqref{eq:uniqueness-result} follows.
\end{proof}

\begin{remark}[Failure of instantaneous higher-energy moment generation]
\label{rem:no-moment-generation}
\normalfont
For the classical spatially homogeneous hard-potential Boltzmann
equation, Povzner-type inequalities yield instantaneous production of polynomial velocity
moments
\cite{Povzner1962,Bobylev1997,Desvillettes1993,
MischlerWennberg1999,Wennberg1997,LuMouhot2012}.
The mass-exchange collision geometry does not possess the same
coercive mechanism for convex moments of the one-particle energy
\(m|v|^2\). Consequently, Assumption \eqref{eq:IC3} is a genuine
additional restriction on the initial datum and cannot in general be
recovered at positive times from \eqref{eq:IC1}.

The obstruction is already visible in the collision geometry.  Fix
\(p>1\) and take
\(
 m=m_1=M, v=w, v_1=-w.
\)
Then \(V=0\),
\(
 E=2M|w|^2,
 m|v|^2=m_1|v_1|^2=E/2,
\)
whereas \eqref{eq:postmass} and \eqref{eq:postvel} give
\(
 m'|v'|^2=(1-\alpha)E,
 m_1'|v_1'|^2=\alpha E.
\)
Consequently,
\begin{align}
 (m'|v'|^2)^p+(m_1'|v_1'|^2)^p
 -(m|v|^2)^p-(m_1|v_1|^2)^p=
 E^p\left[\alpha^p+(1-\alpha)^p-2^{1-p}\right]\geq0,
 \label{eq:Povzner-positive-configuration}
\end{align}
with strict inequality for \(\alpha\neq1/2\).  For
\(a\equiv1\), integration in \(\alpha\) gives
\begin{align}
 &\int_0^1
 \left[
 (m'|v'|^2)^p+(m_1'|v_1'|^2)^p
 -(m|v|^2)^p-(m_1|v_1|^2)^p
 \right]\, d\alpha=
 E^p\left[\frac{2}{p+1}-2^{1-p}\right]>0.
 \label{eq:Povzner-positive-average}
\end{align}
Here the strict inequality is equivalent to \(2^p>p+1\), which holds
for every \(p>1\).
Thus mass exchange can increase a convex particle-energy moment even
for an equal-energy incoming pair.

We next give an \(L^1\)-density counterexample showing that an infinite
higher-energy moment remains infinite at every finite positive time.
Take the admissible kernels
\begin{equation}\label{eq:no-generation-kernels}
 a(m,m_1,\alpha)\equiv1,\qquad b(\xi)\equiv1,
\end{equation}
put \(p=1+\gamma\), and choose
\[
 \phi\in C_c^\infty(\mathbb R^d),\qquad
 \phi\geq0,\qquad
 \int_{\mathbb R^d}\phi(v)\, d v=1,\qquad
 \operatorname{supp}\phi\subset\{v:1<|v|<2\}.
\]
Define
\begin{equation}\label{eq:no-generation-datum}
 f_0^{\mathrm{ex}}(m,v)
 =\one_{\{m\geq1\}}m^{-(p+1)}\phi(v).
\end{equation}
Then
\begin{align}
 M_0(f_0^{\mathrm{ex}})&=\frac1p<\infty,\label{eq:no-generation-M0}\\
 M_1(f_0^{\mathrm{ex}})&=\frac1{p-1}<\infty,\label{eq:no-generation-M1}\\
 M_2(f_0^{\mathrm{ex}})
 &=\frac1{p-1}
 \int_{\mathbb R^d}|v|^2\phi(v)\, d v<\infty,
 \label{eq:no-generation-M2}
\end{align}
while
\begin{align}
 \int_X(m|v|^2)^pf_0^{\mathrm{ex}}(m,v)\, d m\, d v
 &=
 \left(\int_1^\infty\frac{ d m}{m}\right)
 \left(\int_{\mathbb R^d}|v|^{2p}\phi(v)\, d v\right)
 =\infty.
 \label{eq:no-generation-infinite-initial-moment}
\end{align}

Let \(f_N^{\mathrm{ex}}\) be the bounded-kernel solution from
\cref{thm:bounded} with initial datum \(f_0^{\mathrm{ex}}\).
For \(x=(m,v)\) in the support of \(f_0^{\mathrm{ex}}\), one has
\(|v|<2\).  Since
\(
 {mm_1}/({m+m_1})\leq m_1
\)
and
\[
 |v-v_1|^{2\gamma}
 \leq c_\gamma
 \left[|v|^{2\gamma}+|v_1|^{2\gamma}\right],
 \qquad
 c_\gamma=\max\{1,2^{2\gamma-1}\},
\]
we have
\begin{equation}\label{eq:no-generation-E-bound}
 E^\gamma
 \leq
 c_\gamma
 \left[
 2^{2\gamma}m_1^\gamma+
 (m_1|v_1|^2)^\gamma
 \right].
\end{equation}
Since \(B_N\leq E^\gamma\) for the kernels in
\eqref{eq:no-generation-kernels}, the collision frequency
\eqref{eq:collision-frequency} satisfies
\begin{align}
 \nu_{N,f_N^{\mathrm{ex}}}(t,m,v)
 &\leq
 \abs{\Sph}c_\gamma
 \left[
 2^{2\gamma}\int_Xm_1^\gamma
 f_N^{\mathrm{ex}}(t,x_1)\, d x_1+
 \int_X(m_1|v_1|^2)^\gamma
 f_N^{\mathrm{ex}}(t,x_1)\, d x_1
 \right].
 \label{eq:no-generation-frequency-first}
\end{align}
H\"older's inequality and \eqref{eq:boundedmoments} give
\begin{align}
 \int_Xm^\gamma f_N^{\mathrm{ex}}(t,x)\, d x
 &\leq
 M_0(f_0^{\mathrm{ex}})^{1-\gamma}
 M_1(f_0^{\mathrm{ex}})^\gamma,
 \label{eq:no-generation-mass-interpolation}\\
 \int_X(m|v|^2)^\gamma
 f_N^{\mathrm{ex}}(t,x)\, d x
 &\leq
 M_0(f_0^{\mathrm{ex}})^{1-\gamma}
 M_2(f_0^{\mathrm{ex}})^\gamma.
 \label{eq:no-generation-energy-interpolation}
\end{align}
Therefore
\begin{equation}\label{eq:no-generation-frequency-bound}
 \nu_{N,f_N^{\mathrm{ex}}}(t,m,v)\leq C_{\mathrm{loss}}
 \qquad\text{on }\operatorname{supp}f_0^{\mathrm{ex}},
\end{equation}
where
\begin{align}
 C_{\mathrm{loss}}
 &=
 \abs{\Sph}c_\gamma
 M_0(f_0^{\mathrm{ex}})^{1-\gamma}
 \left[
 2^{2\gamma}M_1(f_0^{\mathrm{ex}})^\gamma
 +M_2(f_0^{\mathrm{ex}})^\gamma
 \right].
 \label{eq:no-generation-Closs}
\end{align}
The constant \(C_{\mathrm{loss}}\) is independent of both \(N\) and
the incoming mass \(m\).

The gain--loss decomposition \eqref{eq:gainloss} and the bounded-kernel
strong equation give an identity in \(L^1(X)\).  Since
\(f_N^{\mathrm{ex}}\in C^1([0,T];\X)\) and the gain and loss terms are
integrable on \((0,T)\times X\), Fubini's theorem shows that, for almost
every \(x\), the map \(t\mapsto f_N^{\mathrm{ex}}(t,x)\) has an
absolutely continuous representative satisfying the corresponding
scalar equation.  The integrating-factor formula therefore gives
\begin{align}
 f_N^{\mathrm{ex}}(t,x)
 &=
 f_0^{\mathrm{ex}}(x)
 \exp\left(-\int_0^t
 \nu_{N,f_N^{\mathrm{ex}}}(s,x)\, d s\right)\notag\\
 &\quad+
 \int_0^t
 \exp\left(-\int_s^t
 \nu_{N,f_N^{\mathrm{ex}}}(\tau,x)\, d\tau\right)
 \mathbf Q_N^+(f_N^{\mathrm{ex}},f_N^{\mathrm{ex}})(s,x)
 \, d s.
 \label{eq:no-generation-Duhamel}
\end{align}
Since the gain term is nonnegative,
\eqref{eq:no-generation-frequency-bound} and \eqref{eq:no-generation-Duhamel}
imply
\begin{equation}\label{eq:no-generation-lower-bound-N}
 f_N^{\mathrm{ex}}(t,x)
 \geq e^{-C_{\mathrm{loss}}t}f_0^{\mathrm{ex}}(x)
 \qquad\text{for almost every }x\in X.
\end{equation}
On \(\operatorname{supp}f_0^{\mathrm{ex}}\), this follows from the
frequency bound; outside that set, the right-hand side vanishes and
the inequality follows directly from the nonnegative gain term.

For \(L>0\), set
\[
 \Psi_L(m,v)=\min\{(m|v|^2)^p,L\}.
\]
Multiplying \eqref{eq:no-generation-lower-bound-N} by \(\Psi_L\) and
integrating gives
\begin{equation}\label{eq:no-generation-truncated-N}
 \int_X\Psi_L(x)f_N^{\mathrm{ex}}(t,x)\, d x
 \geq
 e^{-C_{\mathrm{loss}}t}
 \int_X\Psi_L(x)f_0^{\mathrm{ex}}(x)\, d x.
\end{equation}
Apply \cref{prop:global-limit} and \cref{prop:bounded-measurable-tests} to
\((f_N^{\mathrm{ex}})_N\), extract a subsequence
\(N_k\to\infty\), and denote the resulting global
\(L^1\)-integral weak solution by \(f^{\mathrm{ex}}\).  Since
\(\Psi_L\in L^\infty(X)\),
\eqref{eq:global-Cweak-explicit} yields
\begin{equation}\label{eq:no-generation-truncated-limit}
 \int_X\Psi_L(x)f^{\mathrm{ex}}(t,x)\, d x
 \geq
 e^{-C_{\mathrm{loss}}t}
 \int_X\Psi_L(x)f_0^{\mathrm{ex}}(x)\, d x.
\end{equation}
Letting \(L\to\infty\), monotone convergence and
\eqref{eq:no-generation-infinite-initial-moment} give
\begin{equation}\label{eq:no-generation-positive-time}
 \int_X(m|v|^2)^{1+\gamma}
 f^{\mathrm{ex}}(t,m,v)\, d m\, d v=\infty
 \qquad\text{for every finite }t\geq0.
\end{equation}

Thus finite \(M_0,M_1,M_2\) do not generate the moment required in
\eqref{eq:IC3}.  The reduced-mass bound
\(mm_1/(m+m_1)\leq m_1\) keeps the collision frequency uniformly
bounded along a heavy-mass tail with bounded velocity, so a fixed
positive fraction of that tail survives for every finite time.
This example shows that \eqref{eq:IC3} cannot be removed from the
present uniqueness argument through a positive-time moment-generation
step.
\end{remark}

\begin{proof}[Proof of \cref{thm:main}]
The global compact-time limit and the bounds
\eqref{eq:global-weighted-Linfty}, \eqref{eq:global-weighted-Linfty-bound}
are supplied by \cref{prop:global-limit}.  The static
\(L^1\)-realization of the collision operator follows from
\cref{lem:hard-Q-density}.  Applying
\cref{prop:bounded-measurable-tests} gives the global
\(L^1\)-integral equation, the \(W^{1,\infty}\)-in-time regularity,
the bounded-Borel formulation, and the quantitative estimates
\eqref{eq:main-time-derivative-bound}.

Conservation of \(M_0\) and \(M_1\) is
\cref{prop:M01}, while the energy inequality is inherited from
\eqref{eq:global-moment-inequalities}.  Under \eqref{eq:IC2},
\cref{prop:higherenergy,prop:M2} give
\eqref{eq:main-higher-energy}, \eqref{eq:mainM2}.

Finally, assume \eqref{eq:IC3}.  Then \eqref{eq:IC2} holds with
\(\delta=\gamma\).  The constructed solution is energy-dissipating by
\eqref{eq:mainM2ineq}; in fact, \eqref{eq:mainM2} shows that it
conserves kinetic energy.  If \(\widetilde f\) is any other
energy-dissipating global \(L^1\)-integral weak solution with initial
datum \(f_0\), \cref{prop:dissipating-higherenergy} supplies the
locally uniform \(1+\gamma\) higher-energy moment for both solutions,
and \cref{thm:uniqueness} gives \(\widetilde f=f\).  This proves every
assertion of \cref{thm:main}.
\end{proof}

\section{Local theory for linearly growing mass-exchange kernels}\label{sec:linear-growth}

In this section we replace the boundedness assumption in
\textnormal{(K1)} by the following condition, denoted by
\(\textnormal{(K1)}_{\mathrm{lin}}\):
\begin{equation}\label{eq:lin-K1}
 \begin{gathered}
 a:(0,\infty)^2\times(0,1)\longrightarrow[0,\infty)
 \text{ is measurable and locally uniformly continuous},\\
 0\leq a(m,m_1,\alpha)\leq A_a(1+m+m_1)
 \quad\text{on }(0,\infty)^2\times(0,1).
 \end{gathered}
\end{equation}

Set
\begin{equation}\label{eq:lin-weight-moments}
 W(m,v)=1+m+m|v|^2,\qquad
 H_q(g)=\int_XW(x)^qg(x) dx,\qquad q\geq0.
\end{equation}
Thus \(H_1(g)=M_0(g)+M_1(g)+M_2(g)\), and
\begin{equation}\label{eq:lin-W-conservation}
 W(x')+W(x_1')=W(x)+W(x_1).
\end{equation}
A local \(L^1\)-integral weak solution on \([0,T]\) is understood in
the sense of \cref{def:L1-integral}, with \([0,\infty)\) replaced by
\([0,T]\).

This section proves two results.  First, for every
\(p\geq1+\gamma\), finite initial \(H_p\)-moment gives a local
\(H_p\)-solution.  Second, such a solution can be continued through
every finite time \(T\) for which
\(H_{1+\gamma}(f)\in L^1(0,T)\).

\begin{theorem}[Local existence in the \(H_p\) class]
\label{thm:lin-local-existence}
Assume \eqref{eq:lin-K1}, \textnormal{(K2)}, \textnormal{(K3)}, and
\begin{equation}\label{eq:lin-initial-Hp}
 f_0\geq0,\qquad H_p(f_0)<\infty
 \quad\text{for some }p\geq1+\gamma.
\end{equation}
If \(f_0=0\), then \(f\equiv0\) is a global solution.  Assume henceforth
that \(f_0\not\equiv0\), and set
\begin{equation}\label{eq:lin-lifespan-parameters}
 \vartheta=\frac{\gamma}{p-1},\qquad
 \Lambda_p
 =C_{p,\gamma}A_a\|b\|_{L^1(\Sph)}
 H_1(f_0)^{1-\vartheta},
\end{equation}
where \(C_{p,\gamma}>0\) depends only on \(p\) and \(\gamma\). Define
\begin{equation}\label{eq:lin-local-lifespan}
 T_p=
 \begin{cases}
 \displaystyle
 \frac{H_p(f_0)^{-\vartheta}}
 {2\vartheta\Lambda_p},
 &\Lambda_p>0,\\[1.2ex]
 +\infty,
 &\Lambda_p=0.
 \end{cases}
\end{equation}
Then there exists a nonnegative local \(L^1\)-integral weak solution
\(f\) on \([0,T_p]\) such that
\begin{align}
 f&\in L^\infty\bigl(0,T_p;L^1(X;W^p dx)\bigr),
 \label{eq:lin-local-Hp-space}\\
 f&\in W^{1,\infty}\bigl(0,T_p;L^1(X)\bigr)
 \subset C\bigl([0,T_p];L^1(X)\bigr).
 \label{eq:lin-local-time-space}
\end{align}
For every \(t\in[0,T_p]\),
\begin{equation}\label{eq:lin-local-conservation}
 M_j(f(t))=M_j(f_0),\qquad j=0,1,2.
\end{equation}
The collision operator is represented by an \(L^1(X)\) density and
\begin{align}
 \mathbf Q(f,f)&\in L^\infty\bigl(0,T_p;L^1(X)\bigr),
 \label{eq:lin-local-Q-space}\\
 f(t)&=f_0+\int_0^t\mathbf Q(f(s),f(s)) ds
 \quad\text{in }L^1(X),\qquad 0\leq t\leq T_p.
 \label{eq:lin-local-Bochner}
\end{align}
Consequently the weak identity holds for every bounded Borel
measurable test function.
If \(\Lambda_p=0\), then the collision operator vanishes and one may
take \(f(t)=f_0\) for every \(t\geq0\).
\end{theorem}
\subsection{The weighted collision rate and local moment estimate}

\begin{lemma}[Factorization of the linearly growing rate]
\label{lem:lin-kernel-factorization}
There exists \(C_\gamma>0\), depending only on \(\gamma\), such that
\begin{equation}\label{eq:lin-kernel-cross}
 (1+m+m_1)E^\gamma
 \leq C_\gamma
 \left[W(x)W(x_1)^\gamma+W(x)^\gamma W(x_1)\right].
\end{equation}
Consequently,
\begin{align}
 aE^\gamma
 &\leq C_\gamma A_a
 \left[W(x)W(x_1)^\gamma+W(x)^\gamma W(x_1)\right],
 \label{eq:lin-aE-cross}\\
 aE^\gamma
 &\leq C_\gamma A_a
 \left[W(x)^{1+\gamma}+W(x_1)^{1+\gamma}\right],
 \label{eq:lin-aE-separated}\\
 aE^\gamma
 &\leq2C_\gamma A_aW(x)W(x_1).
 \label{eq:lin-aE-product}
\end{align}
If \(p\geq1+\gamma\), then
\begin{equation}\label{eq:lin-aE-p}
 aE^\gamma\leq C_\gamma A_a
 \left[W(x)^p+W(x_1)^p\right].
\end{equation}
\end{lemma}

\begin{proof}
Since \(0<\gamma\leq1\), $W(x)\ge m|v|^2,W(x_1)\ge m_1|v_1|^2$ and $W(x),W(x_1)\ge 1$,
\[
 E^\gamma\leq (m|v|^2)^\gamma + (m_1|v_1|^2)^\gamma\leq  
 W(x)^\gamma W(x_1)+W(x)W(x_1)^\gamma. 
\]
If \(m\leq m_1\), then \(m+m_1\leq2m_1\) and
\(E\leq m|v-v_1|^2\leq 2m(|v|^2+|v_1|^2)\),
whence
\[
 (m+m_1)E^\gamma
 \leq C_\gamma\left[
 m_1(m|v|^2)^\gamma
 +m^\gamma m_1^{1-\gamma}(m_1|v_1|^2)^\gamma\right]
 \leq C_\gamma W(x)^\gamma W(x_1).
\]
Here we used
\(m_1^{1-\gamma}(m_1|v_1|^2)^\gamma\leq (1-\gamma)m_1+\gamma m_1|v_1|^2\leq W(x_1)\).
Interchanging \(x\) and \(x_1\) treats the case \(m_1\leq m\), and since $(1+m+m_1)E^\gamma=E^\gamma+(m+m_1)E^\gamma$,
\eqref{eq:lin-kernel-cross} follows.

The assumption \(a\leq A_a(1+m+m_1)\) gives
\eqref{eq:lin-aE-cross}. Young's inequality \[UV^\gamma\leq \dfrac{1}{1+\gamma}U^{1+\gamma}+\dfrac{\gamma}{1+\gamma}V^{1+\gamma}\] gives
\[
 W(x)W(x_1)^\gamma+W(x)^\gamma W(x_1)
 \leq  W(x)^{1+\gamma}+W(x_1)^{1+\gamma},
\]
which proves \eqref{eq:lin-aE-separated}. Moreover, since \(W\geq1\)
and \(0<\gamma\leq1\),
\[
 W(x)W(x_1)^\gamma\leq W(x)W(x_1),
 \qquad
 W(x)^\gamma W(x_1)\leq W(x)W(x_1),
\]
which proves \eqref{eq:lin-aE-product}. Finally,
\eqref{eq:lin-aE-p} follows from \(W^{1+\gamma}\leq W^p\) when
\(p\geq1+\gamma\).
\end{proof}

For \(n\geq1\), introduce the double truncation
\begin{equation}\label{eq:lin-double-truncation}
 a_n=\min\{a,n\},\qquad
 B_n(E,\xi)=\min\{E^\gamma b(\xi),n\}.
\end{equation}
The bounded-kernel theory in \cref{sec:bounded}, applied to \(a_nB_n\),
gives a global nonnegative solution \(f_n\).  Since
\begin{equation}\label{eq:lin-output-Wp}
 W(x')^p+W(x_1')^p
 \leq[W(x)+W(x_1)]^p
 \leq2^{p-1}[W(x)^p+W(x_1)^p],
\end{equation}
the four-marginal estimate \eqref{eq:four-marginals-condensed} gives
\begin{align}
 \|\mathbf Q_n(g,h)\|_{L^1(W^p)}
 &\leq C_{n,p}\left[
H_p(g)M_0(h)
 +M_0(g)H_p(h)\right].
 \label{eq:lin-Qn-Xp}
\end{align}
Hence $\mathbf Q_n(g,h)$ is locally Lipschitz on bounded subsets of
\(L^1(X;W^pdx)\), and the standard Banach-space Picard--Lindel\"of theorem gives a local solution in this weighted space. Moreover,
\begin{equation}
    H_p(f_n(t))\leq H_p(f_0)+2C_{n,p}M_0(f_0)\int_0^t H_p(f_n(s))ds.
\end{equation}
Gronwall's inequality implies
\begin{equation}
    H_p(f_n(t))\leq H_p(f_0)\exp(2C_{n,p}M_0(f_0)t).
\end{equation}
Thus we can extend the weighted solution to every finite time interval. Since $L^1(X;W^p dx)\hookrightarrow L^1(X)$, it agrees with the unique $L^1$-solution in \cref{sec:bounded}. Hence 
\begin{equation}\label{eq:truncated-weight}
    f_n\in C^1([0,T];L^1(X;W^p dx)) \qquad\text{for any } T>0.
\end{equation}
Hence for any measurable $\psi$ such that $|\psi|\lesssim W^p$, then
\begin{equation}
    \dfrac{d}{dt}\int_X \psi(x) f_n(t,x)dx=\int_X\psi\mathbf Q_n(f_n(t),f_n(t)) dx=\mathcal Q_n(f_n(t),f_n(t))[\psi].
\end{equation}

\begin{lemma}[Uniform local \(p\)-moment estimate]
\label{lem:lin-moment}
Let \(p>1\). Let $f_n$ be the truncated solution constructed in \eqref{eq:truncated-weight}. Then $f_n$ satisfies
\begin{equation}\label{eq:lin-Hp-differential-n}
 \frac{d}{dt}H_p(f_n(t))
 \leq C_{p,\gamma}A_a\|b\|_{L^1(\Sph)}
 H_p(f_n(t))H_{1+\gamma}(f_n(t)),
\end{equation}
where \(C_{p,\gamma}\) is independent of \(n\).  If \(p\geq1+\gamma\) and \(f_0\not\equiv0\), set
\[
 \vartheta={\gamma}/({p-1}),\qquad
 \Lambda_p
 =C_{p,\gamma}A_a\|b\|_{L^1(\Sph)}
 H_1(f_0)^{1-\vartheta}.
\]
Then
\begin{equation}\label{eq:lin-H-interpolation-n}
 H_{1+\gamma}(f_n(t))
 \leq H_1(f_0)^{1-\vartheta}H_p(f_n(t))^\vartheta,
 \qquad \vartheta=\frac{\gamma}{p-1},
\end{equation}
and if $\Lambda_p=C_{p,\gamma}A_a\|b\|_{L^1} H_1(f_0)^{1-\vartheta},$
\begin{equation}\label{eq:lin-Hp-bound-n}
 H_p(f_n(t))
 \leq
 \left[H_p(f_0)^{-\vartheta}-\vartheta \Lambda_pt\right]^{-1/\vartheta}
\end{equation}
whenever the bracket is positive.
\end{lemma}

\begin{proof}
Put
\[
 z=W(x),\quad z_1=W(x_1),\quad
 z'=W(x'),\quad z_1'=W(x_1').
\]
Test the bounded-kernel equation by $W^p$ implies
\begin{equation}\label{eq:testwp}
    \dfrac{d}{dt}H_p(f_n(t))=\mathcal Q_n(f_n,f_n)[W^p]\leq\int_{X^2}\int_0^1\int_{\Sph} aE^\gamma b [z'^p+z_1'^p-z^p-z_1^p]_+ f_n(x)f_n(x_1).
\end{equation}
On \(X_1^2=\{(x,x_1)\in X^2:z\ge z_1\}\), pair conservation and convexity give
\begin{align}
 [z'^p+z_1'^p-z^p-z_1^p]_+
 \leq(z+z_1)^p-z^p-z_1^p\leq p2^{p-1}z^{p-1}z_1.
 \label{eq:lin-positive-increment}
\end{align}
On the same set $X_1^2$, \eqref{eq:lin-aE-cross} shows
\(
 aE^\gamma\leq C_\gamma A_a(zz_1^\gamma+z^\gamma z_1^{1-\gamma} z_1^\gamma) \leq 2C_\gamma A_azz_1^\gamma.
\)
Therefore
\begin{equation}\label{eq:lin-ordered-moment-integrand}
 aE^\gamma[z'^p+z_1'^p-z^p-z_1^p]_+
 \leq C_{p,\gamma}A_az^pz_1^{1+\gamma}.
\end{equation}
The symmetric estimate on \(X_2^2:=\{(x,x_1)\in X^2:z<z_1\}\) has right-hand side
\(C_{p,\gamma}A_az^{1+\gamma}z_1^p\). Hence \eqref{eq:testwp} shows
\eqref{eq:lin-Hp-differential-n}.

If \(p\geq1+\gamma\), then
\(1+\gamma=(1-\vartheta)+\vartheta p\).  H\"older's inequality with
respect to \(f_n(x)dx\) proves \eqref{eq:lin-H-interpolation-n}.  Since
\(H_1(f_n(t))=H_1(f_0)\),
\begin{equation}\label{eq:lin-Bihari-n}
 \frac{d}{dt}H_p(f_n(t))
 \leq \Lambda_pH_p(f_n(t))^{1+\vartheta}.
\end{equation}
Since \(f_0\not\equiv0\), conservation of \(M_0\) gives
\[
 H_p(f_n(t))\geq M_0(f_n(t))=M_0(f_0)>0.
\]
Thus multiplication of \eqref{eq:lin-Bihari-n} by
\(-\vartheta H_p(f_n(t))^{-1-\vartheta}\) is legitimate and gives
\[
 \frac{d}{dt}H_p(f_n(t))^{-\vartheta}
 \geq-\vartheta\Lambda_p.
\]
Integrating from \(0\) to \(t\) proves
\eqref{eq:lin-Hp-bound-n}.
\end{proof}

\subsection{Compactness on the uniform moment interval}

For the remainder of the compactness construction, assume
\(f_0\not\equiv0\) and \(\Lambda_p>0\). The cases \(f_0=0\) and
\(\Lambda_p=0\) will be treated directly in the proof of
\cref{thm:lin-local-existence}. Fix \(0<T\leq T_p\), where \(T_p\)
is defined by \eqref{eq:lin-local-lifespan}, and set
\begin{equation}\label{eq:lin-Pp}
 P_p=\sup_{n\geq1}\sup_{0\leq t\leq T}H_p(f_n(t))
 \leq2^{1/\vartheta}H_p(f_0).
\end{equation}
For a measurable set \(A\subset X\), put
\begin{align}
 F_n(A,t)&=\int_Af_n(t,x) dx,
 \label{eq:lin-FA}\\
 \mathcal R_n(A,t)
 &=\int_{A\times X}\int_0^1\int_{\Sph}
 a_nB_nf_n(t,x)f_n(t,x_1)
  d\omega d\alpha dx_1 dx.
 \label{eq:lin-local-rate-definition}
\end{align}

\begin{lemma}[Localized collision rate]
\label{lem:lin-local-rate}
For every measurable \(A\subset X\),
\begin{align}
 \mathcal R_n(A,t)
 &\leq C_\gamma A_a\|b\|_{L^1}
 \left[
 H_\gamma(f_n(t))\int_AWf_n(t) dx
 +H_1(f_0)\int_AW^\gamma f_n(t) dx
 \right]
 \label{eq:lin-local-rate-cross}\\
 &\leq C_\gamma A_a\|b\|_{L^1}
 \left[
 H_\gamma(f_n(t))P_p^{1/p}F_n(A,t)^{1-1/p}
 +H_1(f_0)P_p^{\gamma/p}F_n(A,t)^{1-\gamma/p}
 \right].
 \label{eq:lin-local-rate-holder}
\end{align}
In particular,
\begin{equation}\label{eq:lin-global-rate}
 \mathcal R_n(X,t)
 \leq2C_\gamma A_a\|b\|_{L^1}H_1(f_0)^2.
\end{equation}
The same estimates hold for every nonnegative \(g\) with finite
\(H_p(g)\), after replacing \(f_n(t)\) by \(g\).
\end{lemma}

\begin{proof}
Since \(a_nB_n\leq aE^\gamma b\), \eqref{eq:lin-aE-cross} gives
\eqref{eq:lin-local-rate-cross}.  For \(0< q\leq p\), H\"older inequality implies
\begin{equation}\label{eq:lin-local-holder}
 \int_AW^qf_n dx
 \leq H_p(f_n)^{q/p}F_n(A,t)^{1-q/p}.
\end{equation}
Taking \(q=1,\gamma\) proves \eqref{eq:lin-local-rate-holder}.
Taking \(A=X\) in \eqref{eq:lin-local-rate-cross}, and using
\(H_\gamma(f_n)\leq H_1(f_n)=H_1(f_0)\), proves
\eqref{eq:lin-global-rate}.
\end{proof}
For \(r>0\), write
\begin{equation}\label{eq:lin-small-mass-function}
 F_n(r,t)=\int_{\{m<r\}}f_n(t,x) dx.
\end{equation}

\begin{lemma}[Small-mass estimate]
\label{lem:lin-small-mass}
For \(0<r<\rho<1\),
\begin{equation}\label{eq:lin-small-mass-differential}
 \frac{d}{dt}F_n(r,t)
 \leq C F_n(\rho,t)^{2-\gamma}
 +Cr(1+\rho^{-1}),
\end{equation}
where \(C\) is independent of \(n,r,\rho,t\).  Consequently,
\begin{equation}\label{eq:lin-small-mass-limit}
 \lim_{r\downarrow0}\sup_{n\geq1}\sup_{0\leq t\leq T}F_n(r,t)=0.
\end{equation}
\end{lemma}

\begin{proof}
Using \(\one_{\{m<r\}}\) in the bounded-test identity for the truncated equation gives
\begin{equation}
    \dfrac{d}{dt}F_n(r,t)\leq \int_{X^2}\int_0^1\int_{\Sph} a_nE^\gamma bf_n(x)f_n(x_1)(\one_{\{\alpha(m+m_1)<r\}}+\one_{\{(1-\alpha)(m+m_1)<r\}})d\omega d\alpha dx_1 dx
\end{equation}

On \(\{m+m_1<\rho\}\subset X^2\), one has \(a_n\leq2A_a\). Since \(
 E^\gamma\leq(m|v|^2)^\gamma+(m_1|v_1|^2)^\gamma
\), by H\"older's inequality on \(\{m<\rho\}\) to each variable, exactly in the proof of \cref{prop:smallmass},
\begin{equation}\label{eq:lin-small-pair}
 \int_{\{m+m_1<\rho\}}a_nE^\gamma
 f_n(x)f_n(x_1) dx_1 dx
 \leq C F_n(\rho,t)^{2-\gamma}.
\end{equation}

On \(\{m+m_1\geq\rho\}\), the union of the two sets of exchange
parameters producing an outgoing mass below \(r\) has length at most
\(2r/(m+m_1)\).  Therefore
\begin{align}
& \int_0^1a_n(m,m_1,\alpha)(\one_{\{\alpha(m+m_1)<r\}}+\one_{\{(1-\alpha)(m+m_1)<r\}}) d\alpha\notag\\ &\qquad \leq2A_ar\frac{1+m+m_1}{m+m_1}
 \leq2A_ar(1+\rho^{-1}).
 \label{eq:lin-alpha-leakage}
\end{align}
Using \(B_n\leq E^\gamma b\),
\begin{align}
&\int_{\{m+m_1\geq\rho\}}\int_0^1\int_{\Sph}
 a_nE^\gamma b
 \bigl(
 \one_{\{\alpha(m+m_1)<r\}}
 +\one_{\{(1-\alpha)(m+m_1)<r\}}
 \bigr)
 f_n(x)f_n(x_1)
 \notag\\
&\qquad\leq
 2A_ar(1+\rho^{-1})\|b\|_{L^1}
 \int_{X^2}E^\gamma f_n(x)f_n(x_1)dx_1dx
 \notag\\
&\qquad\leq
 4A_a\|b\|_{L^1}
 M_0(f_0)^{2-\gamma}M_2(f_0)^\gamma
 r(1+\rho^{-1}).
\end{align}
Combining this estimate with \eqref{eq:lin-small-pair} proves \eqref{eq:lin-small-mass-differential}.

Choose \(\rho=\sqrt r\), integrate on
\(I=[t_0,t_0+\tau]\subset[0,T]\), and define
\begin{equation}\label{eq:lin-LI}
 L_I=\limsup_{r\downarrow0}\sup_n\sup_{t\in I}F_n(r,t).
\end{equation}
If the small-mass limit vanishes at \(t_0\), then
\begin{equation}\label{eq:lin-small-mass-bootstrap}
 L_I\leq C\tau L_I^{2-\gamma}.
\end{equation}
Choose \(\tau>0\) so that
\(C\tau M_0(f_0)^{1-\gamma}<1\).  Since \(2-\gamma>1\),
\eqref{eq:lin-small-mass-bootstrap} forces \(L_I=0\).  The starting
property at \(t_0=0\) follows from \(f_0\in L^1(X)\).  Iteration over
finitely many intervals proves \eqref{eq:lin-small-mass-limit}.
\end{proof}
Define
\begin{equation}\label{eq:lin-Un}
 U_n(q,t)=\sup_{\substack{A\subset X\text{ measurable}\\|A|\leq q}}
 \int_A f_n (t,x)dx,
\end{equation}
and for $D\subset X^2\times (0,1)\times\Sph$ measurable, define the collision-rate measure
\begin{equation}\label{eq:tildeRN}
  \widetilde{\mathcal R}_n(D,t):=\int_{X^2}\int_0^1\int_{\Sph} \one_D a_n B_n f_n(x)f_n(x_1) d\omega d\alpha dx_1 dx.
\end{equation}

\begin{proposition}[Uniform integrability under linear growth]
\label{prop:lin-UI}
For every \(0<T\leq T_p\),
\begin{equation}\label{eq:lin-UI-limit}
 \lim_{q\downarrow0}\sup_{n\geq1}\sup_{0\leq t\leq T}U_n(q,t)=0.
\end{equation}
\end{proposition}

\begin{proof}
Use \(G_{\varepsilon,\kappa}\), \(j_{\varepsilon,\kappa}\), and
\(\mathcal B(\sigma)\) from
\eqref{eq:goodset}, \eqref{eq:jgood}, and \eqref{eq:Bmod}.  For
\(L>1\), set
\begin{equation}\label{eq:lin-zkL}
 z_{\kappa,L}
 =\sup_n\sup_{0\leq t\leq T}F_n(\kappa L,t)
 +\frac{2M_1(f_0)}{L}.
\end{equation}
As in the proof of \cref{lem:goodgain}, write
\begin{align*}
G_{\varepsilon,\kappa}^c &=D_\alpha^\varepsilon\cup D_1^\kappa\cup D_2^\kappa\cup D_3^\kappa \cup D_4^\kappa,\\
    D_\alpha^\varepsilon &= \{(x,x_1,\alpha,\omega)\in X^2\times (0,1)\times \Sph: \alpha<\varepsilon \text{ or }\alpha>1-\varepsilon\},\\
    D_1^\kappa&=\{(x,x_1,\alpha,\omega)\in X^2\times (0,1)\times \Sph: |\theta-\alpha|<\kappa\},\\
    D_2^\kappa&=\{(x,x_1,\alpha,\omega)\in X^2\times (0,1)\times \Sph: |\theta-(1-\alpha)|<\kappa\},\\
    D_3^\kappa &= \{(x,x_1,\alpha,\omega)\in X^2\times (0,1)\times \Sph: \theta<\kappa\},\\
    D_4^\kappa &= \{(x,x_1,\alpha,\omega)\in X^2\times (0,1)\times \Sph: \theta>1-\kappa\}.
\end{align*}
Since the collision integrand is nonnegative,
\[
\widetilde{\mathcal R}_n(G_{\varepsilon,\kappa}^c,t)\leq \widetilde{\mathcal R}_n(D_\alpha^\varepsilon,t)+\sum_{i=1}^4 \widetilde{\mathcal R}_n(D_i^\kappa,t).
\]
Then \eqref{eq:lin-aE-cross} implies
\begin{align}
    &\widetilde{\mathcal R}_n(D_\alpha^\varepsilon,t)+ \widetilde{\mathcal R}_n(D_1^\kappa,t)+\widetilde{\mathcal R}_n(D_2^\kappa,t)\notag\\&\qquad \leq C_\gamma A_a\int_{X^2}\int_0^1\int_{\Sph}(\one_{D_\alpha^\varepsilon}+\one_{D_1^\kappa}+\one_{D_2^\kappa}) [W(x)W(x_1)^\gamma+W(x)^\gamma W(x_1)]f_n(x)f_n(x_1)\notag\\&\qquad \lesssim C_\gamma A_a(2\varepsilon+4\kappa)\|b\|_{L^1}H_1(f_n(t))H_\gamma(f_n(t))\lesssim C_\gamma A_a(\varepsilon+\kappa)\|b\|_{L^1}H_1(f_0)^2.
\end{align}
For $L>1$, set
\[
A_{\kappa,L}:=\{x\in X:m<\kappa L\} \cup \{x\in  X:m>L/2\}.
\]
Hence as in the proof of \cref{lem:goodgain},
\[
D_3^\kappa\cup D_4^\kappa \subset [(A_{\kappa,L}\times X)\cup (X\times A_{\kappa,L})]\times (0,1)\times \Sph.
\]
Since $F_n(A_{\kappa,L},t)\leq F_n(\kappa L,t)+2M_1(f_0)/L\leq z_{\kappa,L}$, applying \eqref{eq:lin-local-rate-holder} with $A=A_{\kappa,L}$ gives
\begin{align}
    \widetilde{\mathcal R}_n(D_3^\kappa,t)+\widetilde{\mathcal R}_n(D_4^\kappa,t)\ &\leq 2\mathcal R_n( A_{\kappa,L},t)\notag\\&\leq C_\gamma A_a\|b\|_{L^1}H_1(f_0)[P_p^{1/p}z_{\kappa,L}^{1-1/p}+P_p^{\gamma/p}z_{\kappa,L}^{1-\gamma/p}].
\end{align}
Taking $R^{\mathrm{lin}}_{\varepsilon,\kappa,L}:=\sup_n\sup_{t\in [0,T]} \widetilde{\mathcal {R}}_n(G^c_{\varepsilon,\kappa},t)$, we obtain
\begin{align}
 R^{\mathrm{lin}}_{\varepsilon,\kappa,L}
 &\leq C_\gamma A_a\|b\|_{L^1}
 \Bigl[
 (\varepsilon+\kappa)H_1(f_0)^2
 +H_1(f_0)P_p^{1/p}z_{\kappa,L}^{1-1/p}
 +H_1(f_0)P_p^{\gamma/p}z_{\kappa,L}^{1-\gamma/p}
 \Bigr].
 \label{eq:lin-bad-rate}
\end{align}
For the gain on \(G_{\varepsilon,\kappa}\), as in the proof of \cref{lem:goodgain}, \eqref{eq:lin-aE-p} shows that the good part gain is bounded by four terms. A representative term is
\begin{align} &\int_{X^2}\int_0^1\int_{\Sph}\one_{G_{\varepsilon,\kappa}}\one_{A}(x')bW(x)^pf_n(t,x)f_n(t,x_1) d\omega d\alpha dx_1 dx \notag\\&\qquad \leq \mathcal B(\sigma)\int_{X^2} W(x)^p f_n(t,x)f_n(t,x_1) dx_1 dx+\|b\|_{L^1}U_n\left(\frac{|\Sph||A|}
 {j_{\varepsilon,\kappa}\sigma},t\right)\int_X W^pf_n dx\notag\\ &\qquad\leq P_pM_0(f_0)\mathcal B(\sigma)
 +P_p\|b\|_{L^1}
 U_n\left(\frac{|\Sph||A|}
 {j_{\varepsilon,\kappa}\sigma},t\right)
\end{align}
The other three terms obtained from $W(x_1)^p$ in place of $W(x)^p$, $\one_A(x_1')$ in place of $\one_A(x')$, and their combination satisfy the same estimate by the symmetry transformations used in \cref{lem:goodgain}. 

Define
\begin{equation}
    G_n(A,t)=\dfrac{1}{2}\int_{X^2}\int_0^1\int_{\Sph} a_nB_n[\one_A(x')+\one_A(x_1')]f_n(t,x)f_n(t,x_1)d\omega d\alpha dx_1 dx.
\end{equation}
Combining the estimates above gives
\begin{align}
 G_n(A,t)
 &\leq R^{\mathrm{lin}}_{\varepsilon,\kappa,L}
 +C_\gamma A_aP_p
 \Biggl[
 M_0(f_0)\mathcal B(\sigma)
 +\|b\|_{L^1}
 U_n\left(\frac{|\Sph||A|}
 {j_{\varepsilon,\kappa}\sigma},t\right)
 \Biggr].
 \label{eq:lin-good-gain}
\end{align}

Let \(I=[t_0,t_0+\tau]\subset[0,T]\). Using \eqref{eq:lin-good-gain}, same argument in the proof of \cref{lem:goodgain} yields
\begin{align}
 \sup_n\sup_{t\in I}U_n(q,t)
 &\leq\sup_nU_n(q,t_0)
 +\tau R^{\mathrm{lin}}_{\varepsilon,\kappa,L}
 +C_\gamma A_aP_pM_0(f_0)\tau\mathcal B(\sigma)\notag\\
 &\quad+C_\gamma A_aP_p\|b\|_{L^1}\tau
 \sup_n\sup_{t\in I}
 U_n\left(\frac{|\Sph|q}
 {j_{\varepsilon,\kappa}\sigma},t\right).
 \label{eq:lin-UI-bootstrap}
\end{align}
Choose \(\tau>0\), independently of \(t_0\), so that
\begin{equation}\label{eq:lin-UI-absorption-time}
 C_\gamma A_aP_p\|b\|_{L^1}\tau\leq\frac12.
\end{equation}
Take $\Lambda_I=\limsup_{q\downarrow 0}\sup_n\sup_{t\in I} U_n(q,t).$ Take $q\downarrow 0$, we have
\begin{equation}
    \Lambda_I\leq 2\tau R^{\mathrm{lin}}_{\varepsilon,\kappa,L}+2C_\gamma A_aP_p M_0(f_0)\tau \mathcal B(\sigma).
\end{equation}
First choose \(L\) sufficiently large so that
\(2M_1(f_0)/L\) is arbitrarily small. Keeping this \(L\) fixed,
choose \(\kappa>0\) sufficiently small. Then
\cref{lem:lin-small-mass} makes
\(
 \sup_n\sup_{0\leq t\leq T}F_n(\kappa L,t)
\)
arbitrarily small. Next choose \(\varepsilon\) sufficiently small,
and finally choose \(\sigma\) sufficiently small. Consequently,
\(\Lambda_I=0\). Starting with \(t_0=0\), repeat the argument on
\(
 I_k=[k\tau,\min\{(k+1)\tau,T\}],
\)
using the conclusion on \(I_k\) as the initial
absolute continuity at the left endpoint of \(I_{k+1}\) ends the proof.
\end{proof}
\begin{proof}[Proof of \cref{thm:lin-local-existence}]
If \(f_0=0\), take \(f(t)\equiv0\). All the asserted properties then
hold globally.

Assume \(f_0\not\equiv0\). Since \(H_1(f_0)>0\), the factor
\(H_1(f_0)^{1-\vartheta}\) is strictly positive. Hence
\(\Lambda_p=0\) implies
\(
 A_a\|b\|_{L^1(\Sph)}=0.
\)
If \(A_a=0\), then \eqref{eq:lin-K1} gives \(a=0\) almost everywhere;
if \(\|b\|_{L^1(\Sph)}=0\), then \(b=0\) almost everywhere. In either
case the collision operator vanishes, and \(f(t)=f_0\) is a global
solution.

It remains to consider \(f_0\not\equiv0\) and \(\Lambda_p>0\).
Fix \(0<T\leq T_p\). The preceding estimates provide approximate solutions $f_n$ on $[0,T]$ with a uniform $H_p$-bound. We first use this bound to obtain uniform integrability and tightness, and hence extract a limit in $C([0,T];L^1(X)\text{-weak})$. We then identify the nonlinear collision term $\mathcal Q_n(f_n,f_n)[\psi]\to \mathcal Q(f,f)[\psi]$ by localization and approximation. Finally, we pass to the weak equation, recover the conserved physical moments, and upgrade the limit to the $L^1$-valued formulation. The four steps parallel the arguments in
\cref{sec:compact}--\cref{sec:strong-global}; we record only the
modifications needed for the linearly growing rate.

Step 1: compactness. As in the proof of \cref{prop:tight}, the small-mass estimate \cref{lem:lin-small-mass} and conservation of \(M_1,M_2\) imply
\begin{equation}\label{eq:lin-tightness}
 \lim_{r\downarrow0,\,R\uparrow\infty}
 \sup_n\sup_{0\leq t\leq T}
 \int_{\{m<r\}\cup\{m>R\}\cup\{|v|>R\}}f_n(t,x) dx=0.
\end{equation}

Moreover, \eqref{eq:lin-global-rate} and the four-marginal
realization give
\begin{equation}\label{eq:lin-time-equicontinuity}
 \|f_n(t)-f_n(s)\|_{L^1(X)}
 \leq C_\gamma A_a\|b\|_{L^1}H_1(f_0)^2|t-s|.
\end{equation}
Uniform integrability and \eqref{eq:lin-tightness} gives the relative weak compactness of $\mathcal F_T=\{f_n(t):n\ge 1,0\leq t\leq T\}$ in $L^1(X)$ by Dunford--Pettis theorem \cite{DiestelUhl1977}. Combining this with \eqref{eq:lin-time-equicontinuity} and the diagonal Arzelà–Ascoli argument used in the proof of \cref{thm:compact}, we obtain, after extraction,
\begin{equation}\label{eq:lin-Cweak}
 f_n\longrightarrow f
 \quad\text{in }C\bigl([0,T];L^1(X)\textnormal{--weak}\bigr).
\end{equation}
For \(\ell\in\mathbb N\), set
\(
 V_\ell(x)=\min\{W(x)^p,\ell\}.
\)
Since \(V_\ell\in L^\infty(X)\), \eqref{eq:lin-Cweak} gives, for every
\(t\in[0,T]\),
\[
 \int_XV_\ell(x)f(t,x)\,dx
 =
 \lim_{n\to\infty}\int_XV_\ell(x)f_n(t,x)\,dx
 \leq P_p.
\]
Letting \(\ell\to\infty\) and using monotone convergence yields
\begin{equation}\label{eq:lin-limit-Hp-bound}
 H_p(f(t))\leq P_p,\qquad 0\leq t\leq T.
\end{equation}
Applying the same argument with
\(\min\{W,\ell\}\) in place of \(V_\ell\) gives
\begin{equation}\label{eq:lin-limit-H1-bound}
 H_1(f(t))\leq H_1(f_0),\qquad 0\leq t\leq T.
\end{equation}

It remains to verify strong measurability in the weighted space.
For every fixed \(\ell\), multiplication by \(V_\ell\) defines a
bounded linear operator on \(L^1(X)\). Hence
\(
 t\mapsto V_\ell f(t)
\)
is weakly continuous, and therefore weakly measurable, as an
\(L^1(X)\)-valued map. Since \(L^1(X)\) is separable, Pettis'
measurability theorem implies that \(t\mapsto V_\ell f(t)\) is
strongly measurable in \(L^1(X)\).

For every fixed \(t\in[0,T]\), \eqref{eq:lin-limit-Hp-bound} gives
\[
 \|V_\ell f(t)-W^pf(t)\|_{L^1(X)}
 =
 \int_X\bigl(W^p-V_\ell\bigr)f(t,x)\,dx
 \longrightarrow0
 \qquad\text{as }\ell\to\infty.
\]
Thus \(t\mapsto W^pf(t)\) is strongly measurable in \(L^1(X)\), being
the pointwise \(L^1(X)\)-limit of strongly measurable maps. Therefore \(t\mapsto f(t)\) is strongly
measurable as an \(L^1(X;W^pdx)\)-valued map. Together with
\eqref{eq:lin-limit-Hp-bound}, this proves
\eqref{eq:lin-local-Hp-space}.

Step 2: identification of the nonlinear collision term. It remains to identify the nonlinear term. The only additional issue caused by the linear growth of $a$ is the control of the collision tails.
Define
\begin{equation}\label{eq:lin-tail-modulus}
 \Phi_p(z)=C_\gamma A_a\|b\|_{L^1}
 \left[H_1(f_0)P_p^{1/p}z^{1-1/p}
 +H_1(f_0)P_p^{\gamma/p}z^{1-\gamma/p}\right].
\end{equation}For the limit \(f\), put
\[
 F(A,t)=\int_Af(t,x)\,dx
\]
and define \(\mathcal R(A,t)\) as in
\eqref{eq:lin-local-rate-definition}, with \(aE^\gamma b\) and \(f\)
in place of \(a_nB_n\) and \(f_n\). By
\cref{lem:lin-local-rate}, \eqref{eq:lin-limit-Hp-bound}, and
\eqref{eq:lin-limit-H1-bound},
\begin{equation}\label{eq:lin-rate-modulus-both}
 \mathcal R_n(A,t)\leq\Phi_p(F_n(A,t)),
 \qquad
 \mathcal R(A,t)\leq\Phi_p(F(A,t)).
\end{equation}
Choose \(K_q\Subset X\) so that
\begin{equation}\label{eq:lin-Kq-tail}
 \sup_{n,t}\int_{K_q^c}f_n(t,x) dx\leq q,\qquad \sup_t\int_{K_q^c}f(t,x)dx\leq q.
\end{equation}
Denote $\mathcal Q_n^{K_q}(f_n,f_n)[\psi], \mathcal Q^{K_q}(f,f)[\psi]$ by the collision form truncated in $K_q\times K_q$, then
\begin{align}
 &\sup_{0\leq t\leq T}
 \left|
 \mathcal Q_n(f_n(t),f_n(t))[\psi]
 -\mathcal Q_n^{K_q}(f_n(t),f_n(t))[\psi]
 \right|
 \leq4\|\psi\|_\infty\Phi_p(q),
 \notag\\
 &\sup_{0\leq t\leq T}
 \left|
 \mathcal Q(f(t),f(t))[\psi]
 -\mathcal Q^{K_q}(f(t),f(t))[\psi]
 \right|
 \leq4\|\psi\|_\infty\Phi_p(q).
 \label{eq:lin-collision-tail-localized}
\end{align}

Restrict next to \(\alpha\in[\rho,1-\rho]\) to obtain $\mathcal Q_n^{K_q,\rho}$ and $\mathcal Q^{K_q,\rho}$. 
For fixed $\alpha$ and $g=f(t),f_n(t)$, since
\[
\int_{X^2}\int_{\Sph}a E^\gamma bg(x)g(x_1)d\omega dx_1 dx\leq 2 C_\gamma A_a\|b\|_{L^1} H_1(f_0)^2,
\]
the omitted contribution
\begin{align}
 &\sup_{0\leq t\leq T}
 \left|
 \mathcal Q_n^{K_q,\rho}(f_n(t),f_n(t))[\psi]
 -\mathcal Q_n^{K_q}(f_n(t),f_n(t))[\psi]
 \right| \leq
 8C_\gamma A_a\|b\|_{L^1}
 H_1(f_0)^2\|\psi\|_\infty\rho,
 \notag\\
 &\sup_{0\leq t\leq T}
 \left|
 \mathcal Q^{K_q,\rho}(f(t),f(t))[\psi]
 -\mathcal Q^{K_q}(f(t),f(t))[\psi]
 \right|\leq
 8C_\gamma A_a\|b\|_{L^1}
 H_1(f_0)^2\|\psi\|_\infty\rho.
 \label{eq:lin-endpoint}
\end{align}

On
\(
 K_q\times K_q\times[\rho,1-\rho]
\)
the function \(a\) is bounded and uniformly continuous, and \(a_n=a\)
for sufficiently large \(n\).  Put
\[
 E_q=\sup_{(x,x_1)\in K_q\times K_q}E(x,x_1)\vee 1<\infty
\]
and let \(\beta_H\) be defined by \eqref{eq:identify-betaH}.  Then
\begin{equation}\label{eq:lin-Bn-local-error}
 \sup_{x,x_1\in K_q}
 \int_{\Sph}|B_n-E^\gamma b| d\omega
 \leq E_q^\gamma\beta_{n/E_q^\gamma}\longrightarrow0.
\end{equation}
Define
\[
A_{q,\rho}=\sup_{K_q^2\times [\rho,1-\rho]} a<\infty.
\]
For \(n\geq\lceil A_{q,\rho}\rceil\), one has \(a_n=a\) on this
compact set. Hence
\begin{equation}\label{eq:lin-n-error}
    \sup_{t\in [0,T]}|\mathcal Q_n^{K_q,\rho}(f_n,f_n)[\psi]-\mathcal Q^{K_q,\rho}(f_n,f_n)[\psi]|\leq 2A_{q,\rho}\|\psi\|_\infty M_0(f_0)^2E_q^\gamma \beta_{n/E_q^\gamma}\to 0.
\end{equation}
Approximate \(b\) by the continuous nonnegative functions \(b_j\) in
\eqref{eq:identify-deltaj}. Denote by $\mathcal Q^{K_q,\rho,j}$ the truncated collision operator of $b_j$ in place of $b$. Let $\delta_j$ be the $L^1$-error introduced in \eqref{eq:identify-deltaj}. Hence
\begin{equation}\label{eq:lin-bj-error}
    \sup_{t\in [0,T]}|\mathcal Q^{K_q,\rho,j}(g,g)[\psi]-\mathcal Q^{K_q,\rho}(g,g)[\psi]|\leq 2A_{q,\rho}\|\psi\|_\infty M_0(f_0)^2 E_q^\gamma \delta_j.
\end{equation}
for $g=f_n(t)$ and $g=f(t).$

For fixed \(q,\rho,j\), define on \(K_q\times K_q\)
\[
 \mathscr K_{q,\rho,j}(x,x_1)
 =
 \int_\rho^{1-\rho}\int_{\Sph}
 a(m,m_1,\alpha)E^\gamma
 b_j\left(\frac{v-v_1}{|v-v_1|}\cdot\omega\right)
 \Delta\psi d\omega d\alpha.
\]
As in Steps~5 and~6 of the proof of \cref{prop:identify},
\(
 \mathscr K_{q,\rho,j}\in C(K_q\times K_q),
\)
and the Stone--Weierstrass product approximation, together with
\eqref{eq:lin-Cweak}, gives
\begin{equation}\label{eq:lin-local-regularized-convergence}
 \sup_{0\leq t\leq T}
 \left|
 \mathcal Q^{K_q,\rho,j}(f_n(t),f_n(t))[\psi]
 -
 \mathcal Q^{K_q,\rho,j}(f(t),f(t))[\psi]
 \right|
 \longrightarrow0.
\end{equation}

We now remove the auxiliary parameters.  For fixed \(q,\rho,j\),
apply the triangle inequality through the chain
\[
\begin{aligned}
 \mathcal Q_n(f_n,f_n)
 &\longrightarrow \mathcal Q_n^{K_q}(f_n,f_n)
 \longrightarrow \mathcal Q_n^{K_q,\rho}(f_n,f_n)\\
 &\longrightarrow \mathcal Q^{K_q,\rho}(f_n,f_n)
 \longrightarrow \mathcal Q^{K_q,\rho,j}(f_n,f_n)\\
 &\longrightarrow \mathcal Q^{K_q,\rho,j}(f,f)
 \longrightarrow \mathcal Q^{K_q,\rho}(f,f)\\
 &\longrightarrow \mathcal Q^{K_q}(f,f)
 \longrightarrow \mathcal Q(f,f).
\end{aligned}
\]
Using \eqref{eq:lin-collision-tail-localized},
\eqref{eq:lin-endpoint}, \eqref{eq:lin-n-error},
\eqref{eq:lin-bj-error}, and
\eqref{eq:lin-local-regularized-convergence}, we obtain
\begin{align}
 &\limsup_{n\to\infty}\sup_{0\leq t\leq T}
 \left|
 \mathcal Q_n(f_n(t),f_n(t))[\psi]
 -
 \mathcal Q(f(t),f(t))[\psi]
 \right|
 \notag\\
 &\quad\leq
 8\|\psi\|_\infty\Phi_p(q)
 +16C_\gamma A_a\|b\|_{L^1}
 H_1(f_0)^2\|\psi\|_\infty\rho
 \notag\\
 &\qquad
 +4A_{q,\rho}\|\psi\|_\infty
 M_0(f_0)^2 E_q^\gamma\delta_j.
 \label{eq:lin-identification-final-bound}
\end{align}
First let \(j\to\infty\), then let \(\rho\downarrow0\), and finally
let \(q\downarrow0\). Since \(p>1\), both exponents
\(1-1/p\) and \(1-\gamma/p\) in \eqref{eq:lin-tail-modulus} are
positive, and hence \(\Phi_p(q)\to0\). Therefore
\begin{equation}\label{eq:lin-identification}
 \lim_{n\to\infty}\sup_{0\leq t\leq T}
 \left|
 \mathcal Q_n(f_n(t),f_n(t))[\psi]
 -
 \mathcal Q(f(t),f(t))[\psi]
 \right|=0,
 \qquad \psi\in C_c^1(X).
\end{equation}

Step 3: passage to the weak equation and conservation laws. Passing to the limit in the time-integrated equations gives the weak
equation for \(f\).  Since \(p>1\),
\begin{equation}\label{eq:lin-W-tail}
 \sup_n\sup_{0\leq t\leq T}
 \int_{\{W>R\}}Wf_n(t,x) dx
 \leq R^{1-p}P_p\longrightarrow0, \quad \text{as } R\to \infty.
\end{equation}
Therefore \(M_0,M_1,M_2\) pass without loss, proving
\eqref{eq:lin-local-conservation}.

Step 4: the $L^1$-valued equation and bounded Borel tests. For every nonnegative \(g\) with finite \(H_1(g)\),
\begin{equation} \label{eq:lin-total-rate-limit}
 \int_{X^2}\int_0^1\int_{\Sph}
 aE^\gamma b\,g(x)g(x_1)
  d\omega d\alpha dx_1 dx\leq2C_\gamma A_a\|b\|_{L^1}H_1(g)^2.
\end{equation}
The four-marginal construction in \cref{lem:hard-Q-density} therefore
defines \(\mathbf Q(g,g)\in L^1(X)\).  Applying this to \(f(t)\) and
using \eqref{eq:lin-local-conservation} gives
\eqref{eq:lin-local-Q-space}.  The proof of
\cref{prop:bounded-measurable-tests}, with
\eqref{eq:lin-total-rate-limit} replacing
\eqref{eq:hard-total-rate-bm}, gives
\eqref{eq:lin-local-time-space}, \eqref{eq:lin-local-Bochner}, and the
bounded Borel formulation.  In the present case \(\Lambda_p>0\), and hence \(T_p<\infty\).
Moreover,
\[
 H_p(f_0)^{-\vartheta}-\vartheta\Lambda_pT_p
 =\frac12H_p(f_0)^{-\vartheta}>0.
\]
Thus all the preceding uniform estimates remain valid on
\([0,T_p]\). Taking \(T=T_p\) completes the proof.
\end{proof}

\subsection{Continuation and blow-up in the
\texorpdfstring{\(H_p\)}{Hp} class}

Fix \(p\geq1+\gamma\).  We call a local solution on \([0,T)\) an
\(H_p\)-solution if
\begin{equation}\label{eq:lin-Hp-solution-class}
 \sup_{0\leq t\leq T'}H_p(f(t))<\infty
 \qquad\text{for every }T'<T,
\end{equation}
where the pointwise \(L^1(X)\)-continuous representative is used.
Throughout this subsection, an \(H_p\)-solution is understood to be
conservative in the basic weight, namely
\begin{equation}\label{eq:lin-Hp-H1-conservative}
 H_1(f(t))=H_1(f(0)),\qquad 0\leq t<T.
\end{equation}

An \(H_p\)-solution branch is called maximal, if
there are no \(\varepsilon>0\) and no \(H_p\)-solution \(\widetilde f\)
on \([0,T+\varepsilon)\) such that
\(\widetilde f=f\) on \([0,T)\).

\begin{lemma}[Intrinsic \(H_p\)-moment inequality]
\label{lem:lin-Hp-intrinsic}
Let \(f\) be an \(H_p\)-solution on \([0,T)\).  Then, for every
\(0\leq s\leq t<T\),
\begin{align}
 H_p(f(t))
 &\leq H_p(f(s))
 +C_{p,\gamma}A_a\|b\|_{L^1(\Sph)}
 \int_s^tH_p(f(\tau))H_{1+\gamma}(f(\tau)) d\tau.
 \label{eq:lin-Hp-intrinsic-inequality}
\end{align}
Consequently,
\begin{equation}\label{eq:lin-Hp-critical-Gronwall}
 H_p(f(t))
 \leq H_p(f(s))
 \exp\left\{
 C_{p,\gamma}A_a\|b\|_{L^1(\Sph)}
 \int_s^tH_{1+\gamma}(f(\tau)) d\tau
 \right\}.
\end{equation}
\end{lemma}

\begin{proof}
Let \(\Phi_R\) and \(\chi_R\) be defined by
\eqref{eq:affine-tail-Phi}--\eqref{eq:bounded-affine-defect}.
Since \(\chi_R(W)\) is bounded, it is admissible in the bounded-Borel
formulation. Moreover,
\(
 \Phi_R(W)+\chi_R(W)=pR^{p-1}W.
\) Write
\[
 \Delta\Phi_R(W)
 :=
 \Phi_R(W(x'))+\Phi_R(W(x_1'))
 -\Phi_R(W(x))-\Phi_R(W(x_1)).
\]
By the pairwise conservation of \(W\),
\(
 \Delta\chi_R(W)=-\Delta\Phi_R(W).
\)
Hence the pairwise conservation of \(W\), together with
\(H_1(f(t))=H_1(f(0))\) gives, exactly as in the proof of
\cref{prop:dissipating-higherenergy},
\begin{align}
 &\int_X\Phi_R(W)f(t) dx-\int_X\Phi_R(W)f(s) dx\notag\\
 &\quad=
 \frac12\int_s^t\int_{X^2}\int_0^1\int_{\Sph}
 aE^\gamma b\,
 \Delta\Phi_R(W)f(\tau,x)f(\tau,x_1)
  d\omega d\alpha dx_1 dx d\tau.
 \label{eq:lin-PhiR-weak-identity}
\end{align}

Put \(z=W(x)\) and \(z_1=W(x_1)\). On \(\{z\geq z_1\}\), the convexity
argument in Step~3 of the proof of
\cref{prop:dissipating-higherenergy} gives
\[
 [\Delta\Phi_R(W)]_+
 \leq z_1\Phi_R'(z+z_1)
 \leq z_1\Phi_R'(2z)
 \leq p2^{p-1}z^{p-1}z_1.
\]
On the same region, \eqref{eq:lin-aE-cross} implies
\[
 aE^\gamma
 \leq C_\gamma A_a
 \bigl(zz_1^\gamma+z^\gamma z_1\bigr)
 \leq2C_\gamma A_a zz_1^\gamma.
\]
Therefore
\[
 aE^\gamma[\Delta\Phi_R(W)]_+
 \leq C_{p,\gamma}A_a z^pz_1^{1+\gamma}
 \qquad\text{on }\{z\geq z_1\}.
\]
Exchanging \(x\) and \(x_1\) gives
\[
 aE^\gamma[\Delta\Phi_R(W)]_+
 \leq C_{p,\gamma}A_a z^{1+\gamma}z_1^p
 \qquad\text{on }\{z<z_1\}.
\]
Using these two estimates in
\eqref{eq:lin-PhiR-weak-identity} yields
\[
 \int_X\Phi_R(W)f(t) dx
 \leq
 \int_X\Phi_R(W)f(s) dx
 +C_{p,\gamma}A_a\|b\|_{L^1(\Sph)}
 \int_s^tH_p(f(\tau))H_{1+\gamma}(f(\tau)) d\tau.
\]
The time integral is finite because \(1+\gamma\leq p\) and \(f\) is an
\(H_p\)-solution. Since \(\Phi_R(W)\uparrow W^p\), monotone convergence
gives \eqref{eq:lin-Hp-intrinsic-inequality}. Applying Gronwall's
inequality gives \eqref{eq:lin-Hp-critical-Gronwall}.
\end{proof}

\begin{theorem}[Critical-moment continuation criterion]
\label{thm:lin-Hp-continuation}
Let \(p\geq1+\gamma\), and let \(f\) be an \(H_p\)-solution on
\([0,T)\), where \(T<\infty\).  If
\begin{equation}\label{eq:lin-Hp-continuation-condition}
 \int_0^T H_{1+\gamma}(f(t)) dt<\infty,
\end{equation}
then there exist \(\varepsilon>0\) and an \(H_p\)-solution
\(\widetilde f\) on \([0,T+\varepsilon)\) satisfying
\begin{equation*}
 \widetilde f(t)=f(t),\qquad 0\leq t<T.
\end{equation*}

Consequently, let \(f\) be any maximal \(H_p\)-solution branch on
\([0,T_{\max}^{(p)})\).  If \(T_{\max}^{(p)}<\infty\), then
\begin{equation}\label{eq:lin-critical-Hp-blowup}
 \int_0^{T_{\max}^{(p)}}H_{1+\gamma}(f(t)) dt=\infty.
\end{equation}
Moreover, writing \(f_0=f(0)\), one necessarily has
\(f_0\not\equiv0\) and
\begin{equation}\label{eq:lin-continuation-parameters}
 \vartheta=\frac{\gamma}{p-1}\in(0,1],\qquad
 \Lambda_p
 =C_{p,\gamma}A_a\|b\|_{L^1(\Sph)}
 H_1(f_0)^{1-\vartheta}>0.
\end{equation}
One further has
\begin{align}
 \int_0^{T_{\max}^{(p)}}H_p(f(t))^\vartheta dt
 &=\infty,
 \label{eq:lin-Hp-power-blowup}\\
 \lim_{t\uparrow T_{\max}^{(p)}}H_p(f(t))
 &=\infty,
 \label{eq:lin-Hp-norm-blowup}\\
 H_p(f(t))
 &\geq
 \left[\vartheta \Lambda_p(T_{\max}^{(p)}-t)\right]^{-1/\vartheta},
 \qquad 0\leq t<T_{\max}^{(p)}.
 \label{eq:lin-Hp-minimal-rate}
\end{align}
\end{theorem}

\begin{proof}
Assume \eqref{eq:lin-Hp-continuation-condition}.  Applying
\eqref{eq:lin-Hp-critical-Gronwall} with \(s=0\) gives
\begin{equation}\label{eq:lin-Hp-uniform-before-endpoint}
 P:=\sup_{0\leq t<T}H_p(f(t))<\infty.
\end{equation}
By \eqref{eq:lin-total-rate-limit}, conservation of \(H_1\), and the
four-marginal realization of the collision operator,
\begin{equation}\label{eq:lin-endpoint-L1-Lipschitz}
 \|f(t)-f(s)\|_{L^1(X)}
 \leq C_\gamma A_a\|b\|_{L^1(\Sph)}H_1(f_0)^2|t-s|.
\end{equation}
Hence there exists \(f_*\geq0\) such that
\begin{equation}\label{eq:lin-Hp-endpoint-L1}
 f(t)\longrightarrow f_*
 \quad\text{in }L^1(X)
 \quad\text{as }t\uparrow T.
\end{equation}
Taking an almost-everywhere convergent subsequence and using Fatou's
lemma, we find
\begin{equation}\label{eq:lin-Hp-endpoint-moment}
 H_p(f_*)\leq P<\infty.
\end{equation}

For every \(R>1\),
\begin{equation}\label{eq:lin-Hp-endpoint-tail}
 \sup_{0\leq t<T}\int_{\{W>R\}}Wf(t) dx
 \leq R^{1-p}P,
\end{equation}
and the same estimate holds for \(f_*\) by Fatou's lemma.  Splitting
the weighted norm into \(\{W\leq R\}\) and \(\{W>R\}\), we obtain
\begin{align*}
 \|f(t)-f_*\|_{L^1(Wdx)}
 &\leq R\|f(t)-f_*\|_{L^1(X)}+2R^{1-p}P.
\end{align*}
First letting \(t\uparrow T\) and then \(R\uparrow\infty\) proves
\begin{equation}\label{eq:lin-Hp-endpoint-weighted-convergence}
 f(t)\longrightarrow f_*
 \quad\text{in }L^1(X;Wdx),
\end{equation}
and therefore \(H_1(f_*)=H_1(f_0)\).

Apply \cref{thm:lin-local-existence} with initial datum \(f_*\).
It produces an \(H_p\)-solution \(g\) on \([0,\varepsilon)\) with
\(g(0)=f_*\).  Define
\begin{equation*}
 \widetilde f(t)=
 \begin{cases}
  f(t),&0\leq t<T,\\
  g(t-T),&T\leq t<T+\varepsilon.
 \end{cases}
\end{equation*}
The convergence \eqref{eq:lin-Hp-endpoint-L1} and the two Bochner
identities show, by splitting the time integral at \(T\), that
\(\widetilde f\) is an \(H_p\)-solution extending \(f\).  This proves
the continuation assertion.  Applying it to a maximal branch proves
\eqref{eq:lin-critical-Hp-blowup}.

If \(f_0=0\), conservation of \(H_1\) and \(W\geq1\)
give \(f(t)=0\) for every \(t\), so the branch is global. Therefore a
finite maximal lifespan implies \(f_0\not\equiv0\), and hence
\(H_1(f_0)>0\).

If \(A_a\|b\|_{L^1(\Sph)}=0\), then the collision operator vanishes
and \(f(t)=f_0\) globally. Thus a finite maximal lifespan also implies
\[
 \Lambda_p
 =C_{p,\gamma}A_a\|b\|_{L^1(\Sph)}
 H_1(f_0)^{1-\vartheta}>0.
\]

Since
\begin{equation}\label{eq:lin-Hp-interpolation-limit}
 H_{1+\gamma}(f(t))
 \leq H_1(f_0)^{1-\vartheta}H_p(f(t))^\vartheta,
\end{equation}
the divergence in \eqref{eq:lin-critical-Hp-blowup} implies
\eqref{eq:lin-Hp-power-blowup}.  Since \(T_{\max}^{(p)}<\infty\),
the latter implies
\begin{equation}\label{eq:lin-Hp-limsup-intermediate}
 \limsup_{t\uparrow T_{\max}^{(p)}}H_p(f(t))=\infty.
\end{equation}

Combining \eqref{eq:lin-Hp-intrinsic-inequality} with
\eqref{eq:lin-Hp-interpolation-limit} and applying Bihari's inequality
gives, for \(0\leq s\leq t<T_{\max}^{(p)}\),
\begin{equation}\label{eq:lin-Hp-Bihari-branch}
 H_p(f(t))
 \leq
 \left[H_p(f(s))^{-\vartheta}
 -\vartheta \Lambda_p(t-s)\right]^{-1/\vartheta}
\end{equation}
whenever the bracket is positive.  If, for some
\(s<T_{\max}^{(p)}\),
\begin{equation*}
 H_p(f(s))^{-\vartheta}
 >\vartheta \Lambda_p(T_{\max}^{(p)}-s),
\end{equation*}
then \eqref{eq:lin-Hp-Bihari-branch} would bound \(H_p(f(t))\)
uniformly for \(s\leq t<T_{\max}^{(p)}\).  Together with local
boundedness on \([0,s]\), this would contradict
\eqref{eq:lin-Hp-limsup-intermediate}.  Hence the reverse inequality holds
for every \(s<T_{\max}^{(p)}\), which is exactly
 \eqref{eq:lin-Hp-minimal-rate}.  This lower bound also proves the
full-limit assertion \eqref{eq:lin-Hp-norm-blowup}.
\end{proof}
\section*{Acknowledgments}
S. Luo gratefully acknowledges the hospitality of the Department of
Mathematics at Duke University during his visit in summer 2026, where
part of this work was carried out. S. Luo used OpenAI's ChatGPT for English-language editing and literature-search assistance. All the authors assume responsibility for all content.

\bibliography{references}
\end{document}